\documentclass[12pt]{amsart}

\usepackage{graphicx}
\usepackage{amssymb}
\usepackage{amsmath}
\usepackage{enumerate}
\usepackage{amsfonts}
\usepackage[mathcal]{euscript}
\usepackage{amsthm}
\usepackage[all,2cell,import]{xy} \UseAllTwocells
\usepackage{xcolor}
\usepackage{rotating} 
\usepackage{url}
\usepackage{lmodern}
\makeatletter
\def\url@leostyle{%
	\@ifundefined{selectfont}{\def\UrlFont{\sf}}{\def\UrlFont{\small\ttfamily}}}
\makeatother
\usepackage[refpage]{nomencl}

\makenomenclature
\numberwithin{equation}{section}
\usepackage[colorlinks=true, pdfstartview=FitV, linkcolor=black, 
			   citecolor=black, urlcolor=black]{hyperref}
\usepackage{bookmark}
\usepackage[titletoc]{appendix} 

\theoremstyle{definition}
\newtheorem{sectional}{Theorem}[subsection]
\newtheorem{theorem}[sectional]{Theorem}

\newtheorem{prop}[sectional]{Proposition}
\newtheorem{lemma}[sectional]{Lemma}
\newtheorem{cor}[sectional]{Corollary}
\newtheorem{corollary}[sectional]{Corollary}

\newtheorem{definition}[sectional]{Definition}
\newtheorem{defn}[sectional]{Definition}
\newtheorem{notation}[sectional]{Notation}
\newtheorem{warning}[sectional]{Warning}
\newtheorem{example}[sectional]{Example}
\newtheorem{construction}[sectional]{Construction}
\newtheorem{convention}[sectional]{Convention}

\newtheorem{remark}[sectional]{Remark}

\newcommand{\nc}{\newcommand}
\nc{\DMO}{\DeclareMathOperator}	

\nc{\newnotation}{\nomenclature}
\nc{\wrap}{\cW}
\nc{\bsd}{\mathsf{bsd}}
\nc{\Cob}{\mathsf{Cob}}
\nc{\mul}{\mathsf{Mul}}
\nc{\fat}{\mathsf{fat}}
\nc{\cob}{\mathsf{Cob}}
\nc{\coh}{\mathsf{Coh}}
\nc{\Top}{\mathsf{Top}}
\nc{\Fun}{\mathsf{Fun}}
\nc{\rep}{\mathsf{Rep}}
\nc{\RepA}{\mathsf{Rep}\cA}
		\DMO{\LMod}{LMod}
\nc{\core}{\mathsf{core}}
\nc{\bary}{\mathsf{Bary}}
\nc{\idem}{\mathsf{Idem}}
\nc{\sets}{\mathsf{Sets}}
\nc{\near}{\mathsf{near}}
\nc{\symp}{\mathsf{Symp}}
\nc{\perf}{\mathsf{Perf}}
\nc{\ssets}{\mathsf{sSets}}
\nc{\cmpct}{\mathsf{cmpct}}
\nc{\finite}{\mathsf{Finite}}
\nc{\compact}{\mathsf{cmpct}}
\nc{\pwrap}{\mathsf{PWrap}}
\nc{\coder}{\mathsf{Coder}}
\nc{\bimod}{\mathsf{Bimod}}
\nc{\grmod}{\mathsf{GrMod}}
\nc{\spaces}{\mathsf{Spaces}}
\nc{\pwrms}{\mathsf{PWrFuk}_{M,S}}
\nc{\pwrmf}{\mathsf{PWrFuk}_{M,F}}
\nc{\pwrapmf}{\mathsf{PWrFuk}_{M,F}}
\nc{\fuk}{\mathsf{Fukaya}}
\nc{\infwr}{\mathsf{InfWr}}
\nc{\fukaya}{\mathsf{Fukaya}}
\nc{\autml}{\mathsf{Aut}_{M,\Lambda}}
\nc{\fukml}{\mathsf{Fukaya}_{M,\Lambda}}
\nc{\fukmle}{\mathsf{Fukaya}_{M,\Lambda,\epsilon}}
\nc{\fukmod}{\wrfukcompact(M)\modules}
\nc{\lag}{\mathsf{Lag}}
\nc{\Lag}{\mathsf{Lag}}
\nc{\lagm}{\lag_M}
\nc{\lago}{\lag^o}
\nc{\lagml}{\lag_{M,\Lambda}} 
\nc{\lagmle}{\lag_{M,\Lambda,\epsilon}}
\nc{\fun}{\mathsf{Fun}}
\nc{\vect}{\mathsf{Vect}}
\nc{\chain}{\mathsf{Chain}}
\nc{\wrfuk}{\mathsf{WrFukaya}}
\nc{\wrfukcompact}{\mathsf{WrFukaya}_{\mathsf{cmpct}}}
\nc{\pwrfuk}{\mathsf{PWrFukaya}}
\nc{\inffuk}{\mathsf{InfFuk}}
\nc{\pwrfukml}{\mathsf{PWrFukaya}_{M,\Lambda}}
\nc{\inffukml}{\mathsf{InfFuk}_{M,\Lambda}}
\nc{\nattrans}{\mathsf{NatTrans}}
\nc{\corres}{\mathsf{Corres}}
\nc{\fukep}{\fukaya_\Lambda(M,\epsilon)}
\nc{\fukepop}{\fukaya_\Lambda(M,\epsilon)^{\op}}
\nc{\lagep}{\lag_\Lambda(M,\epsilon)}
\DMO{\cyl}{cyl} 
\DMO{\cyc}{cyc} 
\nc{\dbcoh}{D^b\mathsf{Coh}}
\nc{\corr}{\mathsf{Corr}}

\nc{\cat}{\mathsf{Cat}}
\nc{\Cat}{\mathsf{Cat}}
\nc{\ainfty}{\mathsf{A}_\infty}
\nc{\inftycat}{\mathcal{C}\!\operatorname{at}_\infty}
\nc{\inftyCat}{\mathcal{C}\!\operatorname{at}_\infty}
\nc{\inftyGpd}{\mathcal{G}\!\operatorname{pd}_\infty}
\nc{\inftyCatPtd}{\mathcal{C}\!\operatorname{at}_{\infty,\ast}}
\nc{\Ainftycat}{\mathcal{C}\!\operatorname{at}_{A_\infty}}
\nc{\ainftycat}{\mathcal{C}\!\operatorname{at}_{A_\infty}}
\nc{\stablecat}{\mathcal{C}\!\operatorname{at}_\infty^{\Ex}}
\nc{\StableCat}{\mathcal{C}\!\operatorname{at}_\infty^{\Ex}}

\nc{\dgcat}{\mathcal{C}\!\operatorname{at}_{dg}}
\nc{\dgcatt}{dgCat}
\nc{\Ainftycatt}{A_\infty Cat}

\DMO{\im}{im}
\DMO{\ev}{ev}
\DMO{\Ind}{Ind}
\DMO{\inj}{inj}
\DMO{\map}{Map}
\DMO{\fib}{fib}
\DMO{\Emb}{Emb}
\DMO{\conf}{Conf}
\DMO{\chains}{Chains}
\DMO{\cochains}{Cochains}
\DMO{\cone}{Cone}
\DMO{\ran}{Ran}
\DMO{\rot}{Rot}
\DMO{\leg}{Leg}
\DMO{\imm}{imm}
\DMO{\adj}{adj}
\DMO{\Crit}{Crit}
\DMO{\tree}{Tree}
\DMO{\cube}{Cube}
\DMO{\deep}{deep}
\DMO{\back}{back}
\DMO{\front}{front}
\DMO{\flow}{Flow}
\DMO{\floer}{Floer}
\DMO{\maps}{Maps}
\DMO{\exact}{exact}
\DMO{\fingen}{\mathsf{f}\mathsf{g}}
\DMO{\excess}{Excess}
\DMO{\Decomp}{Decomp}
\DMO{\decomp}{Decomp}
\DMO{\collar}{collar}
\DMO{\yoneda}{Yoneda}
\DMO{\hamspace}{Ham}
\DMO{\sympspace}{Symp}
\DMO{\holomaps}{Holomaps}
\DMO{\comp}{Comp}
\DMO{\crit}{Crit}
\DMO{\test}{{test}}
\DMO{\sign}{sign}
\DMO{\topp}{top}
\DMO{\indx}{Index}
\DMO{\Break}{Break} 
\DMO{\zero}{zero} 
\DMO{\ob}{Ob}
\DMO{\gr}{Gr} 
\DMO{\Gr}{Gr} 
\DMO{\cl}{Cl} 
\DMO{\grlag}{GrLag}
\DMO{\GrLag}{GrLag}
\DMO{\Pin}{Pin}
\DMO{\Graph}{Graph}
\DMO{\grph}{Graph}
\DMO{\pin}{Pin}
\DMO{\gap}{Gap}
\DMO{\Ex}{Ex}
\DMO{\id}{id}
\DMO{\End}{End}
\DMO{\sym}{Sym} 
\DMO{\aut}{Aut}
\DMO{\DK}{DK} 
\DMO{\poly}{poly} 
\DMO{\diff}{Diff}
\DMO{\coll}{coll}
\DMO{\dist}{dist} 
\DMO{\coker}{coker} 
\nc{\kernel}{\ker} 
\DMO{\sspan}{span}
\DMO{\hocolim}{hocolim}	
\DMO{\holim}{holim}
\DMO{\sk}{sk}

\DMO{\ho}{ho}
\DMO{\fin}{fin}
\DMO{\tor}{Tor}
\DMO{\ext}{Ext}
\DMO{\ret}{Ret}
\DMO{\ham}{Ham}
\DMO{\con}{con}
\DMO{\leaf}{leaf}
\DMO{\supp}{supp}
\DMO{\edge}{edge}
\DMO{\colim}{colim}
\DMO{\edges}{edges}
\DMO{\Image}{image}
\DMO{\roots}{roots}
\DMO{\height}{height}
\DMO{\finmod}{FinMod}
\DMO{\leaves}{leaves}
\DMO{\planar}{planar}
\DMO{\vertices}{vertices}

\nc{\lagg}{\lag^{\cG}}
\nc{\iso}{\mathsf{Iso}}
\nc{\Set}{\mathsf{Set}}
\nc{\ass}{\mathsf{ \bf Ass}}
\nc{\Mod}{\mathsf{Mod}}
\nc{\modules}{\mathsf{Mod}}
\nc{\sset}{\mathsf{sSet}}
\nc{\liou}{\mathsf{Liou}}
\nc{\poset}{\mathsf{Poset}}
\nc{\trno}{T^*\RR^n_{\geq 0}}
\nc{\spectra}{\mathsf{Spectra}}
\nc{\tensorfin}{\tensor^{\fin}}
\nc{\lagptg}{\lag_{pt,pt}^{\cG}}
\nc{\Fin}{\mathcal{F}\mathsf{in}}
\nc{\lagnl}{\lag_{N,\Lambda}}
\nc{\lagmlg}{\lag_{M,\Lambda}^{\cG}}
\nc{\lagsplit}{\lag^{\mathsf{split}}}
\nc{\lagktimes}{(\lag^{\dd k})^\times}
\nc{\lagplanar}{\lag^{\times,\planar}}

\nc{\smsh}{\wedge}
\nc{\para}{\circlearrowleft}
\nc{\sdotpara}{s_{\bullet}^{\para}}
\nc{\un}{\underline}
\nc{\xto}{\xrightarrow}
\nc{\xra}{\xto}
\nc{\tensor}{\otimes}
\nc{\del}{\partial}
\nc{\dd}{\diamond}
\nc{\tri}{\triangle}
\nc{\bb}{\Box}
\nc{\into}{\hookrightarrow}
\nc{\onto}{\twoheadrightarrow}
\nc{\contains}{\supset}
\nc{\transverse}{\pitchfork}
\nc{\uncirc}{\underline{\circ}}
\nc{\Jbar}{\overline{J}}
\nc{\Fbar}{\overline{F}}
\nc{\delbar}{\overline{\del}}
\nc{\thetabar}{\overline{\theta}}
\nc{\omegabar}{\overline{\omega}}

\nc{\colldiff}{\diff^{\del}} 
\nc{\trbar}{\overline{T^*\RR}}
\nc{\tr}{T^*\RR}
\nc{\tsa}{Ts\cA}
\nc{\tsb}{Ts\cB}
\nc{\cmbar}{\overline{\cM}}
\nc{\crbar}{\overline{\cR}}
\nc{\fcrit}{\ff^{crit}}
\nc{\fsubcrit}{\ff^{subcrit}}

\nc{\vece}{ {\vec \epsilon}}	
\nc{\vecd}{ {\vec \delta}}
\nc{\ov}{\overline}
\DMO{\op}{op}
\nc{\opp}{ ^{\op}}

\nc{\hiro}{\textcolor{blue}}

\nc{\eqn}{\begin{equation}}
\nc{\eqnn}{\begin{equation}\nonumber}
\nc{\eqnd}{\end{equation}}
\nc{\enum}{\begin{enumerate}}
\nc{\enumd}{\end{enumerate}}

\def\cA{\mathcal A}\def\cB{\mathcal B}\def\cC{\mathcal C}\def\cD{\mathcal D}
\def\cE{\mathcal E}\def\cF{\mathcal F}\def\cG{\mathcal G}

\def\cM{\mathcal M}\def\cP{\mathcal P}
\def\cR{\mathcal R}\def\cS{\mathcal S}
\def\cU{\mathcal U}\def\cW{\mathcal W}

\def\DD{\mathbb D}
\def\EE{\mathbb E}

\def\QQ{\mathbb Q}\def\RR{\mathbb R}\def\SS{\mathbb S}

\def\ZZ{\mathbb Z}

\def\sB{\mathsf B}

\def\ff{\mathfrak f}

\nc{\shv}{\mathcal{S}\mathsf{hv}}
\DMO{\surj}{surj}
\DMO{\cbl}{cbl}

\nc{\Lambdainj}{\Lambda^{\inj}}

\nc{\Pt}{\mathrm{Pt}}
\nc{\sing}{\mathrm{Sing}}

\nc{\Ffun}{\mathcal{F}\mathsf{un}}
\nc{\Bsc}{\mathcal{B}\mathsf{sc}}
\nc{\Kan}{\mathcal{K}\mathsf{an}}
\nc{\Psh}{\mathcal{P}\mathsf{sh}}
\nc{\Shv}{\mathcal{S}\mathsf{hv}}
\nc{\disk}{\mathcal{D}\mathsf{isk}}
\nc{\Exit}{\mathcal{E}\mathsf{xit}}
\nc{\Strat}{\mathcal{S}\mathsf{trat}}
\nc{\Stacks}{\mathcal{S}\mathsf{tacks}}
\nc{\Rep}{\mathsf{Rep}}
\nc{\Open}{\mathsf{Open}}
\nc{\Conv}{\mathsf{Conv}}
\nc{\Amalg}{\mathsf{Amalg}}
\nc{\Preordpara}{\mathsf{PreOrd_{\para}}}
\nc{\Preordcyc}{\mathsf{PreOrd_{\cyc}}}
\nc{\broken}{\mathsf{Broken}}
\nc{\brokencyc}{\mathsf{Broken}_{\cyc}}
\nc{\brokenpara}{\mathsf{Broken}_{\para}}
\nc{\Deltainj}{\Delta^{\inj}}
\nc{\Deltaparainj}{\Deltapara^{\inj}}
\nc{\Deltacycinj}{\Deltacyc^{\inj}}
\nc{\Deltaparasurj}{\Deltapara^{\surj}}
\nc{\Deltacycsurj}{\Deltacyc^{\surj}}
\nc{\PreordZ}{\mathsf{PreOrd_{\ZZ}}}
\nc{\paracyclic}{\Delta_{\circlearrowleft}}
\nc{\Deltapara}{\Delta_{\circlearrowleft}}
\nc{\Deltasurj}{\Delta^{\surj}}
\nc{\Deltacyc}{\Delta_{\cyc}}
\nc{\paracyc}{\Deltapara}

\nc{\limarrow}{\varprojlim}
\makeatletter
\newcommand{\colimm@}[2]{%
  \vtop{\m@th\ialign{##\cr
    \hfil$#1\operator@font colim$\hfil\cr
    \noalign{\nointerlineskip\kern1.5\ex@}#2\cr
    \noalign{\nointerlineskip\kern-\ex@}\cr}}%
}
\newcommand{\colimarrow}{%
  \mathop{\mathpalette\colimm@{\rightarrowfill@\textstyle}}\nmlimits@
}

\newcommand\mapsfrom{\mathrel{\reflectbox{\ensuremath{\mapsto}}}}

	\title{Cyclic structures and broken cycles}
	\author{Hiro Lee Tanaka}
	
\begin{document}

	\begin{abstract}
	We introduce a new way to encode semicyclic structures using a stack of broken cycles. (We also prove an analogue for paracyclic structures.) This was motivated not only by higher algebra but also by Fukaya-categorical considerations. 
	
	We also openly speculate about some Fukaya-categorical implications. For example, this stack sees moduli of stopped Liouville disks, and hence yields another platform for gluing together Fukaya categories. We also see that Lagrangian cobordisms with multiple ends may not only serve to detect $K_0$ groups of Fukaya categories, but higher $K$-theory groups as well.
	
	Along the way, we include brief expositions of (i) basic techniques in $\infty$-categories and (ii) the translation between exit path categories and constructible sheaves. These may be of independent interest.
	\end{abstract}

	\maketitle

\tableofcontents

\section{Introduction}
Perhaps the most hands-on definition of the K-theory of a category is through Waldhausen's s-dot construction~\cite{waldhausen1985}. This construction takes as input a category $\cC$ with a good notion of exact sequences---for example, a stable $\infty$-category---and outputs a simplicial space $s_\bullet \cC^{\sim}$ whose higher homotopy groups recover the higher K-theory groups of the category.

In the last few years, new structures have been discovered in Waldhausen's s-dot construction. When $\cC$ is a stable $\infty$-category, its s-dot construction has the structure of a 2-Segal space~\cite{dyckerhoff-kapranov-higher-segal}, and further possesses a paracyclic structure~\cite{dyckerhoff-kapranov-crossed, dyckerhoff-kapranov-triangulated-surfaces, lurie-waldhausen, nadler-cyclic}. When $\cC$ is 2-periodic---meaning it is endowed with a natural equivalence between the identity functor and the functor of ``shift by two''---the s-dot construction has a cyclic structure (not just a paracyclic one).

These structures manifest geometrically when $\cC$ is the Fukaya category of some symplectic manifold $M$---indeed, it is our understanding that the discovery of some of the above structures was inspired by Fukaya-categorical ideas, and that the above two structures are central to articulating how one creates Fukaya categories as a sheaf or cosheaf on a 1-dimensional skeleton. But present techniques seemed unwilling to yield  a clean, geometric and precise articulation of these structures for the s-dot construction of Fukaya categories; even for the Fukaya category of a point, the above structures on the s-dot construction were not articulable without first passing to some combinatorial model.

The purpose of this work is to introduce a new way to encode cyclic and paracyclic structures; this new methodology is amenable to the geometric techniques one can apply in the Fukaya-categorical setting.\footnote{We have in mind the case that $M$ is a Weinstein manifold, but we expect this applies for monotone $M$ as well; indeed, many 2-periodic examples arise in the compact monotone setting.}

To that end, we introduce in this paper the moduli stack  
	\eqnn
	\brokencyc
	\eqnd
of broken cycles (Definition~\ref{defn. brokenpara}). This stack classifies families of oriented circles with orientation-compatible $\RR$-actions having discrete and non-empty fixed point set. Such families arise in at least two situations: (i) when studying moduli of infinite-time orbits in a space with an $\RR$-action (where such orbits can develop fixed-point ``breaks''), and (ii) when studying families of symplectic disks equipped with stopped Liouville structures---the boundary circle at infinity inherits an $\RR$-action from the Reeb flow. In our notion of family, new fixed points can appear along closed subsets of the base of the family. (This is just as nodes can appear along closed subsets in the families of curves classified by $\overline{\cM}_{g,n}$.) 

\begin{remark}
The above is a stack on the site of all topological spaces; however, we present it as a colimit of spaces permitting obvious decorations---in particular, one can profitably consider it as a stack on the site of stratified spaces in the sense of~\cite{aft}.
\end{remark}

Now let $\Deltacyc$ denote Connes's cyclic category\footnote{This is often denoted $\Lambda$ following Connes's original work~\cite{connes-cyclique}.}. This category combinatorially models finite subsets of oriented circles and cyclic maps between them (Definition~\ref{defn:cyclic-category}).  We let $\Deltacycinj \subset \Deltacyc$ denote the subcategory consisting of injective maps. By definition, a semicyclic object in an $\infty$-category $\cD$ is a functor from $(\Deltacycinj)^{\op}$ to $\cD$.

Our main theorem is the following:

\begin{theorem}\label{theorem. cyclic main}
Let $\cD$ be a compactly generated $\infty$-category. Then there exists an equivalence of $\infty$-categories
	\eqnn
	\shv(\brokencyc; \cD)
	\simeq
	\Ffun((\Deltacycinj)^{\op}, \cD).
	\eqnd
That is, informally, a sheaf on $\brokencyc$ is the same thing as a semicyclic object.
\end{theorem}

We also prove an analogue for the paracyclic category $\Deltapara$, which models $\ZZ$-equivariant lifts of finite cyclic sets to countable ordered sets, and $\ZZ$-equivariant maps between them. 
The relevant stack here is the moduli stack $\brokenpara$ of broken paracycles, which models families lifting broken cycles to their universal covers.

\begin{theorem}\label{theorem. main}
Let $\cD$ be a compactly generated $\infty$-category. Then there exists an equivalence of $\infty$-categories
	\eqnn
	\shv(\brokenpara, \cD)
	\simeq
	\Ffun((\Deltaparainj)^{\op}, \cD).
	\eqnd
That is, a sheaf on $\brokenpara$ is the same thing as a semiparacyclic object.
\end{theorem}

\begin{remark}
The reader may compare the above theorems to one of the main theorems of~\cite{broken}, where we introduced a stack $\broken$ of broken lines, and proved
	\eqnn
	\shv(\broken, \cD)
	\simeq
	\Ffun(\Deltasurj, \cD).
	\eqnd
This result parallels the above results further by noting the equivalences $(\Deltaparainj)^{\op} \simeq \Deltaparasurj$ and 
$
(\Deltacycinj)^{\op} \simeq \Deltacycsurj.
$
Informally, if $I$ is an object of $\Deltacycsurj$, the elements of $I$ enumerate the open intervals of a circle (formed as the complements of the fixed points). If $I$ is an object of $\Deltacycinj$, its elements track the fixed points themselves.

The reason we have opted to use the form $(\Delta_{--}^{\inj})^{\op}$ is to notationally suggest an analogy with the simplicial case---in many examples, the ``meat'' of a simplicial object $\Delta^{\op} \to \cD$ is contained in the semisimplicial object $(\Deltainj)^{\op} \to \cD$ obtained by restricting to the injective maps. Such also is the case in the examples we have in mind (e.g., the s-dot construction, where the degeneracy maps may be recovered formally as adjoints to the face maps.) 
\end{remark}

\begin{remark}
Moreover, there is a natural stratification on the stack $\brokencyc$ (and likewise for $\broken$ and for $\brokenpara$) that is compatible with the list of isomorphism classes of objects in the category $\Deltacyc$ (and likewise for $\Delta$ and for $\Deltapara$); the above equivalences allow us to compute the stalks at each stratum of $\brokencyc$ by their values on the corresponding object in $\Deltacyc$. See Corollary~\ref{cor.stalks-of-F^{(I)}}.
\end{remark}

\begin{warning}
While the stack $\broken$ of broken lines from~\cite{broken} also relates to Fukaya categories, it does so in a way orthogonal to the way in which $\brokenpara$ and $\brokencyc$ are motivated here. $\broken$ is universal for moduli of broken {\em holomorphic} disks with 2 boundary marked points, while $\brokenpara$ and $\brokencyc$ aide in parametrizing {\em Liouville} disks with $n \geq 1$ stops.
\end{warning}

\subsection{Motivation from algebra}
Let us first explain why---aside from Fukaya-categorical considerations---the above theorems may be of interest. This also partly motivates our previous work with Lurie~\cite{broken}, where broken techniques were first introduced.

Some algebraic operations, such as multiplication, allow  at least three different ways for us to ``draw'' them---i.e., to think of the operations through a topological lens. One way is by colliding labeled points; this was implicit in topological interpretations of the bar construction and in theorems like the Dold-Thom theorem. It also motivates the equivalence between factorizable cosheaves on a Ran space and various algebras~\cite{higher-algebra}. Another way is to continuously embed disks into one another; this is implicit in Eckmann-Hilton arguments, in the original definition of $\EE_n$-algebras~\cite{boardman-vogt-1968,may-iterated-loop-spaces}, and also in the definition of $\disk_n$-algebras as in~\cite{ayala-francis-topological, aft-2}\footnote{One could also encode these structures as factorizable Weiss cosheaves, as in the locally constant version of the factorization algebras recorded in~\cite{costello-gwilliam}.}. A third way is to {\em unrefine} a subdivision, which is implicitly used in works such as~\cite{kks-infrared} and~\cite{afr-stratified-homotopy-hypothesis}.

The technique of sheaves on stacks classifying ``broken'' objects fits into the third picture. Consider a 1-dimensional space equipped with a triangulation, and consider the unrefining operation of removing a vertex (to obtain a triangulation of the same space, but with one fewer vertex). The operation of removing a vertex is dual to the operation of developing a break in our language; and in our stacks, breaks develop when passing to higher-codimension strata of the stack. Hence the algebraic operation associated to ``unrefinement'' is encoded in the restriction maps of a sheaf (as opposed to, say, the corestriction maps of a cosheaf). 

Following this storyline, the present work gives a new way to articulate the algebraic structures that arise when one removes a marked point from an oriented circle (or from a universal cover thereof) while allowing for this circle to rotate (or for the universal cover to translate). These structures are encoded as sheaves on stacks parametrizing broken objects.

\begin{remark}\label{remark:broken-is-good-for-some-situations}
Depending on the geometry of a given situation, one of the above ways to encode algebraic structure may be more natural than another. Indeed, even in the one-dimensional case, when various combinatorial models of Poincar\'e duality lure us to venture between the above three pictures, the situation at hand does not always yield safe passage. See Section~\ref{section:fukaya motivation} for the example of Fukaya categories.
\end{remark}

\begin{example}[The s-dot construction]\label{example:s-dot-construction}
Let us give the main example of a paracyclic object; this will also inform our coming discussion. 

Fix a stable $\infty$-category $\cC$; in particular, we have a good notion of exact sequences and shifts, along with a zero object. For any integer $n \geq 0$, let $s_n \cC$ denote the collection of $n$-step filtrations of an object of $\cC$:
	\eqnn
	s_n \cC = \{0 \to X_1 \to \ldots \to X_n\}.
	\eqnd
We have an obvious equivalence $s_n \cC \simeq \Rep(A_n,\cC)$ to the $\infty$-category of representations of the $A_n$-quiver. (By convention, a representation of the $A_0$ quiver is a choice of zero object in $\cC$.)

Any such filtration naturally extends to a staircase-shaped diagram as follows. We draw the $n=2$ case to save paper:
	\eqn\label{eqn:staircase-sequences}
	\xymatrix{
	0 \ar[r]
		& \ddots \ar[r] \ar[d] 
		& X_1/X_2[-1] \ar[r] \ar[d]
		& 0 \ar[d] 		
		& \\
	\,	& 0 \ar[r] 
		& X_1 \ar[r] \ar[d] 
		& X_2 \ar[r] \ar[d] 
		& 0 \ar[d] \\	
	\,	&
		& 0 \ar[r] 
		& X_2/X_1 \ar[r] \ar[d] 
		& X_1[1] \ar[r] \ar[d] 
		& 0 \ar[d] \\
	\,	&
		&
		& 0 \ar[r] 
		& X_2[1] \ar[r] \ar[d] 
		& X_2/X_1[1] \ar[d] \\
	\,	&
		&
		&
		& 0 \ar[r]
		& \ddots 
	}
	\eqnd
The original filtration $X_1 \to X_2$ is visible in the above diagram; note that each new row as we descend is created by taking iterated pushouts (each of which may be expressed as mapping cones), while we can ascend and create new rows upwards by taking pullbacks. Because of the universal property of pullbacks and pushouts, any row of the above diagram determines the entire diagram up to contractible space of choices; in particular, there is a natural action called ``shift the rows''---i.e., change which row you consider the 0th row. In the case $n=2$, we note that a $3=(n+1)$-step iteration of this operation returns us to the original filtration we began with, but shifted by $[2]$. 

More generally, $s_n \cC$ admits a natural action by the free abelian group $\langle {\frac{1}{n+1}} \rangle \subset \QQ$ such that the action by $+1 \in \ZZ$ is the functor $X \mapsto X[2]$ of ``shift by 2.''\footnote{Warning: Note that $\langle {\frac{1}{n+1}} \rangle$ is abstractly isomorphic to the abelian group $\ZZ$; we use this notation to emphasize that the element $+1 \in \langle {\frac{1}{n+1}} \rangle$ always acts by $[2]$.} 

On the other hand, one has the usual face maps $\del_i: s_n \cC \to s_{n-1}\cC$ given by $(+1)$-periodically deleting the $i$th row and $i$th column of a diagram. There are also degeneracy maps $d_i: s_n \cC \to s_{n+1} \cC$ given by inserting a redundant $i$th row and $i$th column in $(+1)$-periodic fashion; these render $s_\bullet \cC$ into a simplicial object equipped with extra symmetries by the action of $\langle {\frac{1}{n+1}} \rangle$ on $s_n \cC$ for each $n$; these operations are compatible in a suitable way, and such a conglomerate is called a paracyclic object.
\end{example}

\subsection{Relation to K-theory for Fukaya categories}\label{section:fukaya motivation}
Our original motivation for studying the stack of paracycles was to be able to cohere the (discrete) ``stop-removal'' operations for Fukaya categories with the (continuous) rotational action on the moduli of stops on disks. Informally, it is this coherence that exhibits the paracyclic structure on our eventual model for s-dot of a Fukaya category.

Consider the stopped Weinstein domain $(D^2,\ff_n)$. Here, $D^2$ is a compact, two-dimensional disk equipped with a 1-form $\theta$ whose de Rham derivative is a symplectic form, and $\ff_n \subset \del D^2$ is the data of $n+1$ boundary marked points; we call these marked points {\em stops} following Sylvan~\cite{sylvan}.

Let $M$ be another stopped Weinstein domain.  (We have not explicitly notated the stop on $M$, for brevity of notation.) 
By duality, and by the Kunneth formula for partially wrapped categories of Weinsteins~\cite{gps, gps-2}, one has equivalences
	\eqn\label{eqn.An quivers}
	\fukaya(M \times D^2, \ff_n)
	\simeq
	\Rep(A_{n}, \fukaya(M)).
	\eqnd
If one knows the definition of the s-dot construction, this makes clear that the Fukaya categories
	$
	\fukaya(M \times D^2, \ff_n)
	$
form the $n$-simplices of the Waldhausen s-dot construction for $\fukaya(M)$.

But the equivalence~\eqref{eqn.An quivers} breaks certain symmetries. For example, one naturally has $n+1$ stop removal functors 
	\eqnn
	\fukaya(M \times D^2, \ff_n) \to \fukaya(M \times D^2, \ff_{n-1})
	\eqnd
by removing marked points of the disk. Moreover, each $\fukaya(M \times D^2, \ff_n)$ evidently has a rotational action (by rotating $D^2$) which, after a full rotation, has the effect of shifting the grading of a brane by degree 2.

The main goal of this paper and its follower(s?) is to show that these data can be cohered to form a paracyclic system equivalent to the s-dot construction (as sketched in Example~\ref{example:s-dot-construction}) of $\fukaya(M)$ . That is, one has a more geometric model for the K-theory of the Fukaya category respecting the symmetries of the s-dot construction.

Finally, let us mention that one can think of a Lagrangian cobordism $Q \subset M \times D^2$, with exactly $n+1$ ends, as an object of the Fukaya category of $(M \times D^2,\ff_n)$. This, combined with~\eqref{eqn.An quivers}, is a way to see how Lagrangian cobordisms with multiple ends give rise to filtered objects in the Fukaya category.\footnote{ That Lagrangian cobordisms lead to filtrations was observed using different methods by ~\cite{biran-cornea} in the monotone case and by~\cite{tanaka-exact} from a less Floer-theoretic motivation; indeed, the same statement is true for an $\infty$-category of Lagrangian cobordisms as defined in~\cite{nadler-tanaka}.} Thus Lagrangian cobordisms with multiple ends can give rise to {\em higher} $K$-theory groups of Fukaya categories. 

One can also {\em glue} Lagrangian cobordisms if their collarings agree along some ends. We will encounter this in Section~\ref{section:gluing}.

\begin{example}\label{example:fukaya-s-dot}
Let us connect back to the discussion preceding Remark~\ref{remark:broken-is-good-for-some-situations}---which picture is most convenient for capturing s-dot of Fukaya categories?

The ``unrefining'' picture is the most natural one for understanding stop-removal functors---removing a connected component of a stop induces a functor between the two Fukaya categories. In our examples, a ``stop'' is a marked point on the boundary of a 2-dimensional symplectic manifold with boundary, and removing a stop unrefines a stratification of a boundary contact manifold. Moreover, the face maps arising in the s-dot construction of a Fukaya category are precisely the stop-removal functors associated to removing boundary marked points from Liouville disks; cohering this with the rotation of disks and the movement of marked points is the difficult part of articulating paracyclic structures.

However, writing a sheaf on $\brokenpara$ to encode these structures---by definition---boils down to writing a constructible sheaf of $A_\infty$-categories for any space carrying a family of stopped Liouville disks. This is possible through a stratified version of the localization techniques of~\cite{oh-tanaka}, and is the subject of forthcoming work.

In contrast, attempts at encoding this paracyclic structure on the Ran space of the circle (which, while known to encode paracyclic structures, most naturally jives with the ``collision'' picture) have not succeeded as far as this author knows. Indeed, as two stop points of a Liouville disk's boundary circle collide, it is not at all clear what one should do to a Lagrangian brane whose boundary is ``trapped'' between two colliding stop points.
\end{example}

\subsection{Further motivation and speculation: Gluing}\label{section:gluing}
Let $\cC$ be a 2-periodic Fukaya category (obtained, for example, if one does not equip the branes with gradings). The cyclic structure of $s_\bullet \cC$ is then related to another construction: gluing Fukaya categories along skeleta of 2-dimensional symplectic manifolds. We note that the cyclic category is equivalent to its opposite, so the discussion below will be about both sheaves and cosheaves.

Specifically, let $X$ be a symplectic manifold of real dimension 2 with boundary, and let $\Gamma$ be a ribbon graph embedded in $X$. By virtue of the orientation on $X$, $\Gamma$ carries a co/sheaf of cyclic sets on it---the stalk at $x \in \Gamma$ is given by taking the connected components of $X \setminus \Gamma$ in a small neighborhood of $x$. (This can also be viewed as a sheaf because the cyclic category is self-dual.) Thus, any functor from the cyclic category to stable $\infty$-categories yields a co/sheaf of stable $\infty$-categories. When such a cyclic object is given by the s-dot construction of the (partially wrapped or) infinitesimally wrapped Fukaya category of $M$, global sections of this (co)sheaf yields the (partially or) infinitesimally wrapped Fukaya category of the symplectic manifold $X \times M$. (See also the introduction to~\cite{nadler-cyclic}.)

The previous paragraph suggests, of course, that the cyclic set perspective is a combinatorial proxy for the sheaf of actual Liouville domains given by the germs of the symplectic manifold along the skeleton $\Gamma$. 

I will now abuse the reader's indulgence by speculating a bit.

Nadler~\cite{nadler-arboreal}, Starkston~\cite{starkston-arboreal}, and Eliashberg-Nadler-Starkston have pursued a program for constructing a combinatorially understandable list of germs of Liouville structures at the point of a ``generic'' skeleton. It seems that another tantalizing but difficult trajectory is the construction of the stack of such things, just as we have witnessed a stack of broken cycles as a proxy for a stack of Liouville disks in the present work. The idea is that any Lagrangian skeleton (equipped with yet-unarticulated-by-the-community, but appropriate, decorations) should tautologically carry a family of stopped Liouville domains called the germ of a stopped Liouville domain at a point---in other words, any skeleton has a tautological map to the stack of stopped Liouville domains; hence any sheaf on the stack of stopped Liouville domains gives rise to a sheaf on the skeleton.

But not any sheaf will suffice for symplectic purposes. To see this, notice the compatibility between the appearance of $(k+1)$-valent broken cycles in this paper (i.e., cycles with $k+1$ $\RR$-fixed points) and the ``generic'' picture of skeleta---one can perturb any $(k+1 \geq 4)$-valent singularity into a gluing of 3-valent and 2-valent singularities. To have a well-defined Fukaya category, this resolution must not change what the global sections of the sheaf are! The ability to resolve a skeleton this way is related to the 2-Segal structure of the s-dot construction, as we now explain. This was already identified in~\cite{dyckerhoff-kapranov-triangulated-surfaces} as essential for having a well-defined global section of Fukaya categories.

This 2-Segal structure can be expressed through symplectic geometry as follows: any Lagrangian cobordism with $k+1$ ends can be obtained by (geometrically) gluing together Lagrangians with at most $2+1$ ends and (algebraically) taking mapping cones between them. A 2-Segal space's multiplication is multi-valued; a choice of a Lagrangian cobordism specifies a particular value, and the ``gluing cobordisms along ends'' operation is composition of the multiplicative structure of a 2-Segal space. That this multiplicative operation is coherent is the 2-Segal condition.\footnote{Put another way, the colored planar operad structure induced by the s-dot construction seems to be compatible with the colored planar operad structure induced by concatenating cobordisms.} Thus, in some sense, while one may have a sheaf of categories for any oriented (or graded) skeleton by virtue of the sheaf of cyclic (or paracyclic) sets, that this sheaf is well-defined regardless of perturbations of the skeleton is encoded in the 2-Segal structure of the 2-dot construction. Both the (para)cyclic and 2-Segal structures are essential. 

In higher dimensions, there may be little hope of purely combinatorially articulating stacks of Liouville boundaries and their breaks, but many key ingredients seem to be present: An $\RR$-action given by the Reeb flow, a ``directedness'' condition which appears to be a proxy for the framing of a Legendrian submanifold (in the sense of Weinstein handle attachments; confusingly, this structure manifests in a single piece of data---the $\RR$-action---in the low-dimensional setting), and the operation of gluing two Liouville submanifolds along a Weinstein handle seems to encode operations that generalize the multiplicative structures of 2-Segal spaces. (Roughly, if one restricts to a particular class of Liouville singularities, one expects that these operations assemble to articulate exactly the $2n$-Segal space structure allowing one to decompose an $2n$-dimensional polyhedral complex however one wants. Of course, these $2n$-Segal spaces should not only be simplicial, but be lifted to allow for the tangential structures present on our Liouville manifolds, just as we need cyclic or paracyclic structures in the 2-dimensional case).

\subsection{Informal description of the stacks and their sheaves}
Let us explain $\brokenpara$. Just as the stack $BG$ classifies $G$-bundles, $\brokenpara$ classifies families of {\em broken paracycles}---these are fiber bundles with fiber $(-\infty,\infty)$, equipped with a suitable $\ZZ \times \RR$-action. Informally, one may think of such a family as arising from a family of marked points on an oriented circle, equipped with a lift to the real line. The $\ZZ$-action on $(-\infty,\infty)$ encodes the descent to the circle, while a suitable $\RR$-action simultaneously encodes the orientation of the circle and its marked points (which are the $\RR$-fixed points). 

Likewise, $\brokencyc$ classifies families of {\em broken cycles}. These are fiber bundles with fiber $S^1$, equipped with an $\RR$-action that orients the circles and whose fixed point sets are non-empty but discrete. 

\begin{example}
Any family of broken paracycles yields a family of broken cycles by passing to the $\ZZ$-quotient, and any family of broken cycles locally lifts to a family of broken paracycles.
\end{example}

\begin{example}
As we have mentioned, families of broken cycles also arise from families of stopped Liouville disks. (Such are the data necessary to define, for each fiber, a 2-periodic partially wrapped Fukaya category of a disk.)

If each fiber of this family is coherently endowed with a trivialization of $\det^2$ of the tangent bundle, one obtains a family of broken paracycles. (Such are the data necessary to define, for each fiber, a $\ZZ$-graded partially wrapped Fukaya category of a disk.)
\end{example}

\begin{example}
A family of broken cycles over a point is simply a broken cycle---i.e., a circle equipped with an orienting $\RR$-action with at least one $\RR$-fixed point. Up to $\RR$-equivariant homeomorphism, a broken cycle is classified by the number of its $\RR$-fixed points. A broken cycle with $(n+1)$ fixed points has automorphism group 
	\eqnn
	\ZZ/(n+1)\ZZ \times \RR^{n+1}.
	\eqnd
(This indexing convention is to agree with the $n$ in ``$n$-simplex.'')

Thus, an informal description of $\brokencyc$ is as follows: $\brokencyc$ is stratified by non-negative integers $n \geq 0$, and the $n$th stratum is equivalent to the stack $B(\ZZ/(n+1)\ZZ \times \RR^{n+1})$. This means $\brokencyc$ is a stack of dimension -1, with a codimension $k$ stratum for every $k\geq 0$.
\end{example}

\begin{example}
Likewise, one may describe $\brokenpara$ as a stack with strata of dimension $-(n+1)$, where the $n$th stratum is equivalent to the stack $B(\langle{\frac{1}{n+1}}\rangle \times \RR^{n+1})$.
\end{example}

The lefthand sides of our main theorems concern sheaves on the stacks $\brokencyc$ and $\brokenpara$. To give some idea of what a sheaf---on, say, $\brokenpara$---looks like, consider a topological space $S$ equipped with a family of broken paracycles. By definition of $\brokenpara$, this family exhibits a map $g: S \to \brokenpara$, and we may pull back any sheaf $\cF$ on $\brokenpara$ along $g$. The result is a constructible sheaf on $S$, where $S$ is stratified by the isomorphism type of the fibers of the family classified by $g$. The claim is that the $\infty$-category of all sheaves---i.e., the $\infty$-category of all ways to coherently endow pairs $(S,g)$ with a constructible sheaf as above---is equivalent to the $\infty$-category of semiparacyclic objects. (And in particular, the seemingly large coherence data can be succinctly encoded in the combinatorics of the paracyclic category.) Likewise for the cyclic case.

\subsection{Proof outline}
Our proof proceeds by the same general strategy as in~\cite{broken}. 

We first attack the paracyclic case (Theorem~\ref{theorem. main}).
The starting point is to exhibit a cover of $\brokenpara$. This is done by articulating an analogue of a trivializing section of a $G$-bundle, and this analogue is what we have called {\em $I$-sections}, where $I$ is a suitable preorder with $\ZZ$-action. We will see that the stack representing families equipped with $I$-sections are representable by stratified spaces. This allows us to construct $\brokenpara$ as a colimit of spaces (in the $\infty$-category of stacks), hence it allows us to write the $\infty$-category of sheaves as a limit of $\infty$-categories of constructible sheaves on the spaces appearing in our colimit diagram. Throughout, discrete, combinatorially defined categories allow us to witness adjunctions that enable the passage from one (co)limit diagram to another. 

The result for the cyclic case---i.e., for $\brokencyc$---follows the same strategy, and in fact even recycles many of the same players because of the evident cover $\brokenpara \to \brokencyc$. 

We have reviewed the basics of these techniques in Section~\ref{section:categorical-review} for the reader's convenience. We have also included a discussion of constructible sheaves and exit path $\infty$-categories in Section~\ref{section:constructible-sheaves}. 

\subsection{Conventions and recollections}

\begin{convention}[Addition and subtraction of infinities]\label{convention:adding-subtracting-infinity}
In this paper we will make use of addition and subtraction of elements in $[-\infty,\infty]$, so let us be explicit about our conventions.

Addition is an operation
	\eqnn
	+ :
	[-\infty, \infty] \times 
	[-\infty, \infty] \setminus\{(x,-x) , x = \pm \infty \}
	\to
	[-\infty,\infty].
	\eqnd
That is, $a + b$ is defined so long as the pair $(a,b)$ does not equal $(\infty,-\infty)$ nor $(-\infty,\infty)$. The convention
	\eqnn
	\infty + \infty = \infty,
	\qquad
	-\infty + -\infty = -\infty,
	\qquad
	t + \pm \infty = \pm \infty + t = \pm \infty
	\eqnd
(for $t$ finite) renders $+$ a continuous, commutative operation where defined.

Subtraction is an operation
	\eqnn
	- :
	[-\infty,\infty] \times [-\infty,\infty] \setminus \{ (x,x) , x = \pm \infty\}
	\to
	[-\infty,\infty].
	\eqnd
That is, $a-b$ is defined so long as the pair $(a,b)$ does not satisfy $a = b=\infty$, nor $a=b=-\infty$. The convention
	\eqnn
	\infty - (-\infty) = \infty,
	\qquad
	-\infty - \infty = - \infty,
	\eqnd
	\eqnn
	t - \pm \infty = \mp \infty,
	\qquad
	\pm \infty - t = \pm \infty,
	\eqnd
(for $t$ finite) makes subtraction continuous and skew-commutative, meaning $a-b = -(b-a)$.
\end{convention}

\begin{remark}\label{remark:equations-for-infinity-additions}
Because $[\infty,\infty]$ is not a group under addition, the equations
	\eqnn
	a = b+c,
	\qquad
	b = a-c,
	\qquad
	c = a- b
	\eqnd
have non-equivalent solution sets. In all the situations we deal with in this paper, it will be natural to consider the union of the solution sets to each equation. 
\end{remark}

\begin{convention}[Spaces and Kan complexes]
To a seasoned homotopy theorist, the terms ``topological space'' and ``Kan complex'' become interchangeable. We will do our best to {\em not} make use of this convenience, to distinguish
\enum
\item Situations in which we care about topological spaces only up to homeomorphisms. In such situations, we will speak of ``topological spaces,'' and the notation $\Top$ will denote the category of topological spaces. Its morphism sets are discrete---for any $S, T \in \Top$, $\hom(S,T)$ is the set of continuous maps $S \to T$. Two objects of $\Top$ are equivalent if and only if they are homeomorphic.
\item Situations in which we care about spaces only up to weak homotopy equivalence. In such situations, we will speak of ``Kan complexes,'' and the notation $\Kan$ will denote the $\infty$-category of Kan complexes. One can informally think of the objects of $\Kan$ as topological spaces, but now the morphisms $\hom(S,T)$ ``know'' that there is a path between two continuous maps $f_0,f_1: S \to T$ when there is a homotopy between them. Moreover, two weakly homotopy equivalent objects are equivalent objects in $\Kan$, in contrast to $\Top$.
\enumd
It is our hope that distinguishing the terms ``topological space'' and ``Kan complex'' will make it clearer to the reader when, if ever, we are distinguishing between $\Top$ and $\Kan$. 
\end{convention}

\begin{remark}
In other works, $\Kan$ is what some writers (including myself) may refer to as the ``$\infty$-category of spaces.''  However, we will never again utter this phrase in this paper. This is because here, $\Kan$ will mainly appear as a tool to organize groupoids and $\infty$-groupoids into the language of $\infty$-categories, and not to make any meaningful foray into the homotopy theory of spaces.

(Indeed, a Kan complex is also naturally thought of as an $\infty$-category whose morphisms are all invertible, and two Kan complexes are equivalent if and only if they are equivalent as $\infty$-categories.)
\end{remark}

For our purposes, it will be most convenient to model the collection of stacks as an $\infty$-category, as opposed to utilizing the language of 2-categories. 

Let us recall some basic facts about stacks and $\infty$-categories for the reader. For a more thorough review of sheaves on stacks, we refer the reader to~\cite{broken}.

\begin{convention}[Stacks]
In this paper, when we speak of stacks, we mean stacks over the usual site of topological spaces, with the usual notion of open cover. Note in particular that our topological spaces $S$ need not be Hausdorff nor compactly generated.

The $\infty$-category of stacks can be modeled concretely as follows: A stack $\sB$ is the data of a category $\Pt(\sB)$ equipped with a right fibration $\Pt(\sB) \to \Top$, whose associated presheaf $\Top^{\op} \to \Kan$ is a sheaf. (This means that for any open cover, the \v{C}ech nerve construction exhibits the global sections as a homotopy limit of the local sections.)

In particular, a morphism in this $\infty$-category from $\sB_0 \to \sB_1$ is the data of a functor $F: \Pt(\sB_0) \to \Pt(\sB_1)$ together with a natural isomorphism $\eta$ making the triangle
	\eqnn
	\xymatrix{
	\Pt(\sB_0) \ar[rr]^F \ar[dr]^p && \Pt(\sB_1) \ar[dl]_{p'} \\
	& \Top 
	}
	\eqnd
commute. That is, $\eta$ is a natural isomorphism from $p' \circ F$ to $p$. And a ``commutative triangle'' of stacks consists of a triplet of morphisms
	\eqnn
	(F_{ij},\eta_{ij}): \Pt(\sB_i) \to \Pt(\sB_j),
	\qquad
	0 \leq i < j \leq 2
	\eqnd
together with a natural isomorphism cohering the composite $(F_{12},\eta_{12}) \circ (F_{01},\eta_{01})$ with $(F_{02},\eta_{02})$.
\end{convention}

\begin{remark}
Though we assemble stacks into an $\infty$-category, stacks also form an example of what some call a $(2,1)$-category; which is to say that the only morphisms are 1-morphisms and 2-morphisms, and it so happens that every 2-morphism is an isomorphism.
\end{remark}

\begin{remark}
Note we have used the fact that a right fibration over $\Top$ is equivalent to the data of a functor from $\Top^{\op}$ to the $\infty$-category of Kan complexes. (This is an infinity-categorical version of the classical Grothendieck construction for categories fibered in groupoids.) The interested reader may consult the straightening and unstraightening constructions as outlined in~\cite{htt}.

Informally, given a stack $\Pt(\cB) \to \Top$ and its associated functor $\cF: \Top^{\op} \to \Kan$, the $\cF$ assigns to $S$ the category of objects classified by $\cB$; so for example, when $\cB = BG$, $\cF(S)$ is the category of $G$-bundles over $S$. 
\end{remark}

\begin{remark}
Note that because $\Pt(\sB)$ is assumed to be an ordinary category, rather than an arbitrary $\infty$-category, the fibers of the right fibration $\Pt(\sB) \to \Top$ (which are all Kan complexes by definition of right fibration) only have homotopy groups in degrees 0 and 1. That is, its fibers are equivalent to groupoids in the usual sense, as opposed to arbitrary $\infty$-groupoids.
\end{remark}

\begin{remark}
Finally, we note that any map of right fibrations over $\Top$ is automatically a map respecting Cartesian edges.
\end{remark}

\clearpage

\clearpage 
\section{The paracyclic category}

The paracyclic category was first introduced in~\cite{connes-cyclique} and the name paracyclic was introduced by~\cite{getzler-jones-crossed}. 
We follow the notation from~\cite{lurie-waldhausen}; the reader may also find a discussion in~\cite{dyckerhoff-kapranov-crossed}, where the paracyclic category is denoted $\Lambda_\infty$ rather than our $\Deltapara$.
	
\begin{definition}\label{defn. parasimplex}
Recall that a {\em parasimplex} is the data 
		\eqnn
		(I,\mu)
		\eqnd
of a totally ordered set $I$ equipped with an order-preserving $\ZZ$-action $\mu$ satisfying two conditions: 
		\enum
			\item For any two elements $i \leq i'$, the set of $i''$ satisfying $i \leq i'' \leq i'$ is finite, and 
			\item For every $i \in I$, we have $i < i +1$. (That is, for any $i$, the induced map $\mu(-,i): \ZZ \to I$ is order-preserving.)
		\enumd
A map of parasimplices is a weakly order-preserving, $\ZZ$-equivariant map. This means a map is a function $f: I \to J$ such that $i \leq i' \implies f(i) \leq f(i')$ and $f(i+1) = f(i)+1$.
\end{definition}
	
	\begin{remark}
	A parasimplex may informally be thought of as a cyclically ordered set lifted to an ordered set. The geometrically minded reader may think of a finite subset of an oriented circle, equipped with a lift to the universal cover (with the $\ZZ$-action given by deck transformations). 
		\end{remark}

	\begin{notation}\label{notation. paracyclic category}
	We let $\Deltapara$ denote the category of parasimplices. We call it the {\em paracyclic} category.
	\end{notation}

\begin{defn}
If $\cC$ is an $\infty$-category, a paracyclic object of $\cC$ is a functor $\Deltapara^{\op} \to \cC$.

Let $\Deltaparainj \subset \Deltapara$ denote the subcategory consisting of injective maps. A {\em semiparacyclic object} is a functor $(\Deltaparainj)^{\op} \to \cC$.
\end{defn}

\begin{remark}
The difference between $\Deltaparainj$ and $\Deltapara$ is analogous to the difference between $\Delta^{\inj}$ and $\Delta$. The ``meat'' of the algebraic structure is often in the face maps, so semiparacyclic objects are rich objects. Moreover, in the examples we have in mind, the missing degeneracy maps can be attached formally from the face maps---for example, given only the face maps of the s-dot construction, degeneracy maps can be recovered as adjoints to the  face maps.
\end{remark}

	\begin{example}\label{example.paracyclics are simplicial}
		Let $A$ denote a finite, non-empty, linearly ordered set. Then $\ZZ \times A$ is a parasimplex when endowed with the dictionary order: $(n,a) \leq (n',a')$ if and only if (i) $n < n'$, or (ii) $n=n'$ and $a \leq a'$. The assignment $A \mapsto \ZZ \times A$ defines a functor 
		\eqn\label{eqn:functor-Delta-Deltarapara-times-ZZ}
		\ZZ \times - : \Delta \to \Deltapara
		\eqnd.
	In this way, any paracyclic object restricts to a simplicial object.
	\end{example}
	
	\begin{notation}[$i^{++}$]\label{notation:successor++}
	By Definition~\ref{defn. parasimplex}, any element $i$ of a parasimplex $I$ admits a unique successor---this is the least $i'$ such that $i ' > i$. We denote by
		\eqnn
		i^{++}
		\eqnd
	the successor to $i$. (This borrows from computer science notation.)
	\end{notation}
	
	\begin{notation}[$n_I$]\label{notation:n_I}
	Given any parasimplex $I$, choose an element $i_0 \in I$. Then there is a unique integer 
		\eqnn
		n_I \geq 0
		\eqnd
	for which the collection of elements
		\eqnn
		\{ \text{ $i$ such that $i_0 \leq i < i_0 +1$} \} \subset I
		\eqnd
	is---with the order inherited from $I$---isomorphic to the linear poset $[n_I]$.
	
	Note that $n_I$ is independent of the choice of $i_0$.
	\end{notation}
	
	\begin{remark}
	The functor~\eqref{eqn:functor-Delta-Deltarapara-times-ZZ} is essentially surjective, as any $I$ is isomorphic to the parasimplex $\ZZ \times [n_I]$. (See Notation~\ref{notation:n_I}.)
	\end{remark}

	\begin{remark}
		For any two parasimplices $I_0$ and $I_1$, we have that $\hom_{\Deltapara}(I_0,I_1)$ has a natural $\ZZ$-action. Composition is compatible with this action, so we obtain a category $\Deltacyc$ with the same objects as $\Deltapara$, but with $\hom_{\Deltacyc} = \hom_{\Deltapara}/\ZZ$. We call $\Deltacyc$ the {\em cyclic} category. See Definition~\ref{defn:cyclic-category}.
	\end{remark}
	
	\begin{remark}
		The category of $\ZZ$-torsors is symmetric monoidal---in fact, it is equivalent (as a symmetric monoidal category) to the category whose nerve is the simplicial group $B\ZZ$. The category of $\ZZ$-torsors acts on $\Deltapara$, and hence on the category of paracyclic objects of $\cC$. One can identify the fixed points of the $B\ZZ$-action with cyclic objects; moreover, the functor $\Delta \to \Deltapara$ from Example~\ref{example.paracyclics are simplicial} is final. Hence we conclude that the geometric realization of a cyclic object is endowed with an $S^1$-action.
	\end{remark}
	
\begin{construction}
Let $I$ be a parasimplex. we let
	\eqnn
	\DD I
	:=
	\hom_{\surj}(I, [1])
	\eqnd
denote the collection of weakly order-preserving maps $I \to [1]$ that are surjective. We endow $\DD I$ with an order and a $\ZZ$-action as follows:
\enum
	\item For $e,e' \in \DD I$, we declare $e \leq e'$ if and only if $e(i) \leq e'(i)$ for all $i \in I$.
	\item $(e+1)(i) = e(i - 1)$.
\enumd

Note that any map of parasimplices $I \to I'$ induces a map of parasimplices $\DD I' \to \DD I$. 
\end{construction}

\begin{prop}
The functor
	\eqnn
	\DD: \Deltapara^{\op} \to \Deltapara,
	\qquad
	I \mapsto \DD I
	\eqnd
is an equivalence of categories.
\end{prop}

\begin{proof}
It suffices to exhibit a natural isomorphism between the identity functor and $\DD \circ \DD$.  To do so, define a map
	\eqnn
	I \to \hom_{\surj}( \hom_{\surj}(I, [1]), [1])),
	\qquad
	x \mapsto x^\vee
	\eqnd
where
	\eqnn
	x^\vee(e) = e(x).
	\eqnd
It is straightforward to check that this is a natural isomorphism of parasimplices. 
\end{proof}

\begin{notation}\label{notation:Deltainj}\label{notation:Deltasurj}
Let $\Deltaparainj$ and $\Deltaparasurj$ denote the subcategories of $\Deltapara$ consisting only of injections and surjections, respectively. Then 
\end{notation}

\begin{cor}\label{cor. Deltainj and Deltasurj}
$\DD$ induces an equivalence of categories
	\eqnn
	(\Deltaparainj)^{\op} \simeq \Deltaparasurj.
	\eqnd
\end{cor}

\nc{\TwArr}{\mathsf{TwArr}}

\begin{remark}
If $f: I \to I'$ is an injection, we have a commutative diagram
	\eqnn
	\xymatrix{
	I \ar[d]^f \ar[r] & \DD I \\
	I' \ar[r] & \DD I' \ar[u]^{f^*}	
	}
	\eqnd
where $I \to \DD I$ is the function taking $x$ to the function $e$ such that $\max e^{-1}(0) = x$. Thus we have a functor
	\eqnn
	(\Deltaparainj)^{\op}
	\to
	\TwArr(\Deltapara)	
	\eqnd
to the twisted arrow category of $\Deltapara$.
\end{remark}

\begin{remark}
The topologically minded reader may enjoy thinking of $\DD$ as follows. A parasimplex $I$ defines a directed graph whose vertices are elements of $I$ and for which there exists an edge between two consecutive elements of $I$. $\DD I$ may then be thought of as the collection of edges of $I$. 
\end{remark}

\begin{remark}
Yet another model of $\DD$ is as follows: Every object $I$ is sent to itself, while any morphism $f: I \to I'$ is sent to a map
	\eqnn
	f^\vee : I' \to I,
	\qquad
	f^\vee(x')
	=
	\max \{x \, | \, f(x) \leq x'\}.
	\eqnd
When $f$ is an injection, we have a retraction:
	\eqnn
	f^\vee \circ f = \id.
	\eqnd
\end{remark}

For the reader's edification, we conclude with some remarks about how these structures show up in studying the s-dot construction and $K$-theory. These are not essential to the content of the paper, and these ideas are also found in~\cite{lurie-waldhausen}.

	\begin{remark}
		Let $\cC = \Kan$, and let $s_\bullet \cD$ be the s-dot construction of some stable $\infty$-category $\cD$. If $\cD$ is 2-periodic, noting that the $\ZZ$-action on $s_\bullet \cD$ shifts objects $X \mapsto X[2]$, we conclude that $s_\bullet \cD$ is equipped with the structure of a homotopy fixed point of the $B\ZZ$-action on paracyclic spaces. 
	\end{remark}
	
	\begin{remark}
	In the non-periodic case, this structure is trivial at the level of K-theory: Informally, the generator $1 \in \ZZ$ acts through the trivial $S^1 \simeq B\ZZ$ action by sending a K-theory class $X$ to the K-theory class $X[2]$. (These are the same K-theory class.) Lurie showed in~\cite{lurie-waldhausen} that this structure is compatible with a map $S^1\simeq U(1)$ into $BPic(\SS)$ (via Bott periodicity and the complex $J$-homomorphism), where the associated $\ZZ$-action on any stable $\infty$-category sends $X \mapsto X[2]$.
	\end{remark}
	
	\begin{remark}
	In the 2-periodic case, it follows that the $K$-theory Kan complex $K(\cD)$ is equipped with a loop of automorphisms. Informally, this is because we see ``two reasons'' the map $\ZZ \to \aut(K(\cC))$ is trivial; one because $X$ is always equivalent to $X[2]$ in $K$-theory, and another because of the 2-periodic structure on $\cC$. Thus we have a diagram
		\eqnn
		\xymatrix{
		\ZZ \ar[r] \ar[d] & \ast \ar[d] \\
		\ast \ar[r] & \aut(K(\cC))
		}
		\eqnd
	hence a homomorphism $\ZZ \to \Omega \aut(K(\cC))$.
	\end{remark}

\clearpage
\section{Broken paracycles}

\begin{notation}
Below, we will write $(-\infty,\infty)$ for the usual real line as a topological space, and we will use the symbol $\RR$ to denote the same space considered as a topological group (with addition).
\end{notation}

We recall the following from~\cite{broken}:

\begin{defn}[Broken line]\label{defn. broken line}
A {\em broken line} is the data of a topological space together with a continuous $\RR$-action such that 
\enum
	\item The fixed point set of $\RR$ is discrete, and 
	\item There exists a homeomorphism $\phi$ to $[0,1]$ such that $\phi(x+t) \geq \phi(x)$ for all $t \geq 0$. (Here, $\geq$ is the usual order on $[0,1]$.)
\enumd
(In particular, any broken line is abstractly homeomorphic to a compact interval.)
\end{defn}

In~\cite{broken} we introduced a notion of a family of broken lines to encode representations of the category $\Delta^{\surj}$ as sheaves on the stack $\broken$ of broken lines. 

In what follows we define a variation on this theme, introducing the notion of broken paracycles. Roughly speaking, one obtains a  broken paracycle by concatenating a broken line countably many times. However, not all families of broken paracycles are obtained by concatenating families of broken lines.

\begin{notation}
Let $L$ be any space with an $\RR$-action. We let $L^\RR$ denote the fixed point set and we denote the complement by
	\eqnn
	L^\circ = L \setminus L^\RR.
	\eqnd
\end{notation}

\begin{prop}\label{prop. free Z action}
Let $L$ be a topological space abstractly homeomorphic to $(-\infty,\infty)$, and fix a continuous $\ZZ$-action on $L$. If the action has no fixed points, then there exists an equivariant homeomorphism from $L$ to $(-\infty,\infty)$ equipped with the standard action of $\ZZ$. 
\end{prop}

\begin{proof}
Choose a generator $+1 \in \ZZ$ and let $\phi: L \to L$ be the corresponding homeomorphism. Then the graph of $\phi$ is a connected subset of $L \times L \setminus \Delta_L$; the latter has exactly two connected components which, under a homeomorphism $L \cong (-\infty,\infty)$, are given by $\{x_1 < x_2\}$ and $\{x_2 < x_1\}$. Without loss of generality, one may induce a total order on $L$ from $(-\infty,\infty)$ such that for all $x \in L$, we have $x < \phi(x)$.

Choosing a homeomorphism sending $[x,\phi(x)]$ to $[0,1] $ and extending $\ZZ$-equivariantly, the result follows.
\end{proof}

\newenvironment{paracyclicproperties}{%
	  \renewcommand*{\theenumi}{(L\arabic{enumi})}%
	  \renewcommand*{\labelenumi}{(L\arabic{enumi})}%
	  \enumerate
	}{%
	  \endenumerate
}

\begin{defn}\label{defn:broken-paracycle}
A {\em broken paracycle} is the data of a topological space $L$ equipped with a continuous action of $\ZZ \times \RR$ satisfying the following properties:
\begin{paracyclicproperties}
\item $L$ is abstractly homeomorphic to $(-\infty,\infty)$ as a topological space. 
\item The $\ZZ = \ZZ \times \{0\}$ action is free. 
\item\label{property:directed-action-ordering} The $\RR = \{0\} \times \RR$ action is directed. That is, there exists a homeomorphism $L \cong (-\infty,\infty)$ so that the  induced action on $(-\infty,\infty)$ satisfies $t \cdot x \geq x$ for all $t \geq 0 \in \RR$. We let $\leq_L$ denote the order on $L$ induced from the standard order on $(-\infty,\infty)$.
\item The $\ZZ = \ZZ \times \{0\}$ action is compatibly directed. That is, for any $x \in L$, we have $x \leq_L x+1$. (Here, $+1$ is the action by $ 1 \in \ZZ$.)
\item\label{property:fixed-point-set-discrete-in-line} The fixed point set $L^\RR$ is discrete and non-empty.
\end{paracyclicproperties}
\end{defn}

\begin{remark}
By Proposition~\ref{prop. free Z action}, we may choose a $\ZZ$-equivariant homeomorphism from $L$ to $(-\infty,\infty)$ with the standard $\ZZ$ action. Since the $\ZZ$-action is compatible with the $\RR$-action, the $\RR$-action is determined completely by its effect on an interval of the form $[x,x+1]$, where $x \in L$ is arbitrary. If we choose $x$ to be an $\RR$-fixed point, $[x,x+1]$ is a broken line in the sense of Definition~\ref{defn. broken line}.
\end{remark}

\begin{remark}
It also follows that for any pair $(x,y) \in L \times L$, there exists $n \in \ZZ$ such that $x <_L y + n$.
\end{remark}

\begin{warning}
Though we have used multiplicative and additive notation to denote the actions of $\RR$ and $\ZZ$, respectively, we warn the reader that the action is {\em not} distributive; so
	\eqnn
	t\cdot(x+n) \neq t\cdot x + t  n.
	\eqnd
(Indeed, the notation ``$+tn$'' does not even make sense for most $t$ and $n$.)
We also caution that the action $t \cdot x$ does not in any way resemble the usual multiplicative action of $\RR$ on $(-\infty,\infty)$. 
\end{warning}

\begin{remark}\label{remark:R-equivariance-is-order-preserving}
Suppose $L$ and $L'$ are broken paracycles and $\tilde f: L \to L'$ is a continuous, $\RR$-equivariant map. Then $\tilde f$ is weakly order-preserving: $\tilde f(x) \leq_{L'} \tilde f(y)$ whenever $x \leq_L y$.
(See Definition~\ref{defn:broken-paracycle}~\ref{property:directed-action-ordering} for the notation $\leq_L$.)
\end{remark}

\begin{defn}
An isomorphism of broken paracycles is an isomorphism of $(\ZZ \times \RR)$-spaces---that is, a homeomorphism commuting with the $\ZZ \times \RR$ action.
\end{defn}

\begin{remark}\label{remark. isom classes of paracycles}
Up to isomorphism, a broken paracycle is classified by the number $n +1$ of $\RR$-fixed points in the quotient circle $L / \ZZ$. (Here $n \geq 0$.) For a broken paracycle $L$ yielding $n+1$ such fixed points, the automorphism group of $L$ is given by 
	\eqnn
	\aut(L) \cong \langle 1/(n+1) \rangle \times \RR^{n+1}.
	\eqnd
Here, $\langle 1/(n+1) \rangle \subset \QQ$ is the additive subgroup generated by $1/(n+1)$. It is abstractly isomorphic to $\ZZ$, but our notation is meant to elucidate the relationship of the $\langle 1/(n+1) \rangle$ action of the automorphism group with the $\ZZ$ action defining $L$ as a broken paracycle.
\end{remark}

\begin{notation}\label{notation:n_L for broken line L}
Given a broken paracycle $L$, we let $n_L$ denote the integer $n$ in Remark~\ref{remark. isom classes of paracycles}. 
\end{notation}

\begin{remark}
In Remark~\ref{remark. isom classes of paracycles} we have modded out a broken paracycle $L$ by its $\ZZ$-action. The resulting $\RR$-space is an example of a broken {\em cycle} (Definition~\ref{defn:broken cycle}).
\end{remark}

\begin{notation}[Distance]\label{notation. distance}
Let $L$ be a broken paracycle.
Define a function
	\eqnn
	d_L: (L \times L) \setminus \Delta_{L^\RR} \to [-\infty,\infty]
	\eqnd
by
	\eqnn
	d_L(x,y) =
	\begin{cases}
	t \in \RR & t\cdot x = y \\
	-\infty & t \cdot x \geq y \text{ for all $t \in \RR$} \\
	\infty & t \cdot x \leq y \text{ for all $t \in \RR$} 
	\end{cases}
	\eqnd
(Note that the domain does not contain the diagonal pairs $(x,x)$ for which $x$ is an $\RR$-fixed point.)  We call $d_L(x,y)$ the translation distance from $x$ to $y$. 
\end{notation}

\begin{remark}\label{remark:distance-skew-symmetric}
$d$ is skew-symmetric. That is, 
	\eqnn
	d_L(x,y) = - d_L(y,x).
	\eqnd
\end{remark}

\begin{remark}\label{remark:distance-respects-ordering}
If follows from \ref{property:directed-action-ordering}  that $d(x,y) \geq 0$ if and only if $x \leq_L y$. (Of course, we follow the convention that $\infty \geq 0$.)
\end{remark}

\begin{remark}
$d_L$ is a continuous function. We prove this later in Proposition~\ref{lemma. distance continuous}. 

Note $d_L$ cannot be extended continuously to the diagonal fixed points. To see this, it is illustrative to replace $L$ by a broken line (Definition~\ref{defn. broken line}) with exactly one interior fixed point. Choose an isomorphism with $[-1,1]$ such that the origin of $[-1,1]$ is the interior fixed point. Then any extension $\overline{d_L}: L \times L \to [-\infty,\infty]$ may be modeled as a function on a closed square $[-1, 1] \times [-1,1]$. 

This function takes the values $\infty$ and $-\infty$ on the open quadrants II and IV, respectively (while the values at quadrants I and III interpolate). In particular, there is no value of $d_L$ that one can assign to the origin rendering $\overline{d_L}$ continuous. By analyzing the value of $d_L$ on the boundary edges of the squares, we also see that the fixed-point corners (i.e., the corners of quadrants I and III) cannot admit any values rendering $\overline{d_L}$ continuous.
\end{remark}

\subsection{Families}

We now present the definition of a family of broken paracycles. We have chosen to present a definition that is technically useful, but can seem opaque on a first pass. For a less technically involved description, we refer the reader to Theorem~\ref{theorem:broken-definitions-equivalent}---this theorem states that one may replace the technical conditions
\ref{property:quotient-closed},
\ref{property:separated}, 
and
\ref{property:local lifting}
below with a more familiar local triviality condition: Every family of broken paracycles is a fiber bundle with fiber $(-\infty,\infty)$, equipped with a standard $\ZZ$-action. 

\newenvironment{brokenproperties}{%
	  \renewcommand*{\theenumi}{(Q\arabic{enumi})}%
	  \renewcommand*{\labelenumi}{(Q\arabic{enumi})}%
	  \enumerate
	}{%
	  \endenumerate
}
\begin{defn}[Family of broken paracycles]\label{defn:family-of-broken-paracycles}
Fix $S$ a topological space. A {\em family of broken paracycles} over $S$ is the data of
	\eqnn
	( \pi: L_S \to S, \mu: (\ZZ \times \RR) \times L_S \to L_S)
	\eqnd
where $L_S$ is a topological space, $\pi$ is a continuous map, and $\mu$ is a continuous $\ZZ \times \RR$-action on $L_S$ preserving the fibers, such that
\begin{brokenproperties}
	\item\label{property:fibers paracyclic} For every $s \in S$, $\mu$ renders the fiber $L_s := \pi^{-1}(s)$ a broken paracycle (Definition~\ref{defn:broken-paracycle}).
	\item\label{property:unramified} Let $L_S^{\RR}$ denote the fixed point set of the $\RR = \{0\} \times \RR$ action. Then the map
		\eqnn
		L_S^\RR \to S
		\eqnd
	induced by $\pi$ is unramified modulo $\ZZ$. (See Remark~\ref{remark. unramified} below.)
	\item\label{property:quotient-closed}  The induced projection from the quotient space
		\eqnn
		L_S/ \ZZ \to S
		\eqnd
	is a closed map. (See also Proposition~\ref{prop:properly-discontinuous} below.)
	\item\label{property:separated} The set $\{(x,y) \text{ such that $x \leq y$}\}$ is a closed subset of the fiber product $L_S \times_S L_S$.
	\item\label{property:local lifting} Let $L_S^\circ = L_S \setminus L_S^{\RR}$ be the complement of the fixed point set. Then the map
		\eqnn
		L_S^\circ \to S
		\eqnd
	induced by $\pi$ has the local lifting property.
\end{brokenproperties}
\end{defn}

\begin{remark}
Property~\ref{property:local lifting} means that for every $\tilde s \in L_S^\circ$, there exists a neighborhood $U$ of $s = \pi(\tilde s)$ and a continuous map $\sigma: U \to L_S^\circ$ such that $\sigma(s) = \tilde s$.
\end{remark}

\begin{remark}[Unramified]\label{remark. unramified}
Because the $\RR$-action is compatible with the $\ZZ$-action, the $\RR$-action on $L_S$ descends to an $\RR$-action on the quotient $L_S/\ZZ$, and the  $\RR$-fixed point locus $L_S^\RR$ is a $\ZZ$-torsor over the $\RR$-fixed point locus of $L_S/\ZZ$.

Property~\ref{property:unramified} means that, locally on $S$, the map
	\eqnn
	(L_S/\ZZ)^\RR \to S
	\eqnd
is locally a closed embedding. (Both instances of the word locally are necessary.) Equivalently, for every $s \in S$, there exists a neighborhood $U$ containing $s$ such that the fixed point locus
	\eqnn
	(L_U/\ZZ)^\RR
	=
	\coprod_{\alpha} K_\alpha
	\eqnd
may be expressed as a finite disjoint union of closed subsets of $L_U$. Moreover, the projection map $K_\alpha \to U$ is a closed embedding to $U$ for every $\alpha$.
\end{remark}

\begin{defn}\label{defn. maps of families}
A map of families of broken paracycles is a commutative diagram
	\eqnn
	\xymatrix{
	L_S \ar[r]^{\tilde f} \ar[d] & L_T \ar[d] \\
	S \ar[r]^f & T
	}
	\eqnd
where $f$ is a continuous map, $\tilde f$ exhibits the space $L_S$ as the pullback of $L_T$ along $f$, and $\tilde f$ commutes with the $\ZZ \times \RR$ action.

An isomorphism is a map such that $f$ (and hence $\tilde f$) is a homeomorphism.
\end{defn}

\begin{example}[Pullbacks]\label{example:pullback}
Let $L_T \to T$ be a family of broken paracycles, and fix a continuous map $f: S \to T$. Then the fiber product $f^* L_T = S \times_T L_T$ is a family of broken paracycles over $S$. 
\end{example}

\begin{notation}[Pullbacks and subscripts]\label{notation.pullback subscript}
Let $K \subset S$ be a subset and $(\pi: L_S \to S, \mu)$ a family of broken paracycles over $S$. Then we denote by $L_K$ the pullback of $L_S$ along the inclusion $K \subset S$. 

We will most often use this notation for $L_U$ when $U \subset S$ is open, and for $L_s$ (i.e., for the fiber) when $s \in S$ is an element.
\end{notation}

\begin{example}
Let $L$ be any broken paracycle and $S$ any topological space. The product $L \times S$ with the obvious $\ZZ \times \RR$ action is a {\em trivial} family of broken paracycles. It is isomorphic (Definition~\ref{defn. maps of families}) to the family $L \to \ast$ pulled back along the constant map $S \to \ast$ to a point.
\end{example}

\begin{warning}
The fibers of a family of broken paracycles may, even locally, witness different isomorphism types of broken paracycles; when we construct local models $\tilde F^{(I)} \to F^{(I)}$ of families of broken paracycles, we will see explicit examples of families whose fibers can ``bleed'' from one isomorphism type to another.  See Section~\ref{section.F^{(I)} construction}.
\end{warning}

\subsection{Some facts about families}\label{section.facts-about-families}
We conclude this section with some useful topological facts about families of broken paracycles. 

Throughout this section, by ``a pair $(\pi:L_S \to S, \mu)$'' we mean the data of a continuous map $\pi: L_S \to S$ and a continuous action $\mu: \ZZ \times \RR \times L_S \to L_S$ preserving the fibers of $\pi$. 

\begin{prop}\label{prop:section-orders-are-open}
Fix a pair $(\pi:L_S \to S, \mu)$ satisfying ~\ref{property:fibers paracyclic} and~\ref{property:separated}.  Fix a continuous section $\sigma: S \to L_S$. Then each of the subspaces
	\eqnn
	\{ \text{x such that $x > \sigma(\pi(x))$} \}, 
	\{ \text{x such that $x < \sigma(\pi(x))$} \}
	\subset
	L_S
	\eqnd
are open.
\end{prop}

\begin{remark}
In the statement of Proposition~\ref{prop:section-orders-are-open}, we have written $<$ to mean $<_{L_{\pi(x)}}$. (See Definition~\ref{defn:broken-paracycle}~\ref{property:directed-action-ordering} for the notation $<_L$.) It is the total order  given by the broken paracycle structure on the fiber.
\end{remark}

\begin{proof}
We have the pullback diagram
	\eqnn
	\xymatrix{
	\{ \text{x such that $x > \sigma(\pi(x))$} \} \ar[r] \ar[d]
		& 	\{ \text{ (x,y) such that $x > y$} \} \ar[d] \\
	L_S \cong L_S \times_S S\ar[r]^{\id_{L_S} \times \sigma }
		& L_S \times_S L_S.
	}
	\eqnd
The right vertical arrow is an open embedding by \ref{property:separated}. Hence the lefthand vertical arrow is also an open embedding, proving the claim for one of the subspaces. The claim for the other follows by applying the swap homeomorphism $L_S \times_S L_S \to L_S \times_S L_S$ sending $(x,y) \mapsto (y,x)$. 
\end{proof}

Recall the translation distance $d$ defined on any broken paracycle (Notation \ref{notation. distance}). In particular, given any family $L_S \to S$ and two points $x,y$ in the same fiber, one can measure the translation distance from $x$ to $y$.

\begin{notation}
Fix a pair $(\pi:L_S \to S, \mu)$. Given $s \in S$, we let $d_{L_s}$ denote the distance function of the broken paracycle $L_s$. We define the function
	\eqnn
	d_{L_S}: (L_S \times_S L_S) \setminus \Delta_{L_S^\RR} \to [-\infty,\infty]
	\qquad
	(x,y) 
	\mapsto d_{L_s}(x,y).
	\eqnd
Note that the domain consists of those pairs $(x,y)$ such that $x$ and $y$ are in the same fiber of $L_S \to S$, and such that if $x$ and $y$ are $\RR$ fixed points, then $x \neq y$.
\end{notation}

\begin{lemma}\label{lemma. distance continuous}
Fix a pair $(\pi:L_S \to S, \mu)$ satisfying~\ref{property:fibers paracyclic} and~\ref{property:separated}. Then the distance function is continuous.
\end{lemma}

\begin{proof}[Proof of Lemma~\ref{lemma. distance continuous}.]
For brevity, let us set $d = d_{L_S}$.

Let $U \subset [-\infty,\infty]$ be an open set. We must verify that the preimage is open. We may restrict our attention to those $U$ of the form $(t,\infty]$ or $[-\infty,t)$, as such sets form a subbasis for the topology of $[-\infty,\infty]$. (Here, $t$ is any real number; in particular, we need not consider $t = \pm \infty$.) By skew-symmetry of $d_{L_S}$ (Remark~\ref{remark:distance-skew-symmetric}) and because the swap map $L_S \times_S L_S \to L_S \times_S L_S$ is a homeomorphism, we are reduced to showing that the preimage of $[-\infty,t)$ is open. Finally, because we have that
	\eqn\label{eqn:translation-invariance-of-distance-relation}
	d(x,y) < t \iff d(x+t,y) < 0 \iff d(x, y-t) < 0
	\eqnd
and the $\RR$-action is continuous, we may assume $t=0$.\footnote{The logical equivalences in \eqref{eqn:translation-invariance-of-distance-relation} do not hold if we include the diagonal $\RR$-fixed points in the domain of $d$.} So we must now show that the preimage of $[0,\infty]$ is closed.

But by Remark~\ref{remark:distance-respects-ordering},  $d(x,y) \geq 0$ if and only if $x \leq y$. The proof is complete by invoking \ref{property:separated}.
\end{proof}

\begin{prop}\label{prop:properly-discontinuous}
Fix a pair $(\pi:L_S \to S, \mu)$ satisfying
\ref{property:fibers paracyclic},
\ref{property:separated}, and
\ref{property:local lifting}.
Then the quotient map  $L_S \to L_S/\ZZ$ is a covering map.
\end{prop}

\begin{proof}
It suffices to prove that for any $\tilde s \in L_S$, there exists a neighborhood $\tilde U$ of $\tilde s$ such that $\tilde U$ and $n+\tilde U$ have empty intersection whenever $n \neq 0$.

Given $\tilde s$ above a point $s \in S$, choose two points $x_1, x_2$ in the same fiber as $\tilde s$ satisfying the following conditions:
\enum
\item $x_1,x_2$ are not $\RR$-fixed points.
\item We have that $x_1 < \tilde s < x_2$.
\item $x_1 + 1 > x_2$, where $+1$ denotes the action by the generator of $\ZZ$.
\enumd
Then by \ref{property:local lifting} we may choose a neighborhood $U$ admitting sections $\sigma_1,\sigma_2: U \to L_U$ with $\sigma_i(s) = x_i$. Because $d$ is continuous by Lemma~\ref{lemma. distance continuous}, we may assume $\sigma_1(s') < \sigma_2(s')$ for all $s' \in U$ (shrinking $U$ if necessary). Now consider 
	\eqnn
	\tilde U 
	=
	\{\text{$y \in L_U$ such that $\sigma_1(\pi(y)) < y < \sigma_2(\pi(y))$}
	\}.
	\eqnd
By Proposition~\ref{prop:section-orders-are-open} $\tilde U$ is an open subset of $L_U$, hence of $L_S$. By construction, $\tilde U$ satisfies the property we seek.
\end{proof}

\newenvironment{brokenequiv}{%
	  \renewcommand*{\theenumi}{(Q\arabic{enumi}')}%
	  \renewcommand*{\labelenumi}{(Q\arabic{enumi}')}%
	  \enumerate
	}{%
	  \endenumerate
}

\begin{lemma}\label{lemma:closed-map-lemma}
Assume that the data $(\pi:L_S \to S,\mu)$ satisfy properties
\ref{property:fibers paracyclic},
\ref{property:separated}, and
\ref{property:local lifting}.
Then property~\ref{property:quotient-closed} is equivalent to the following property:
\begin{brokenequiv}
		\setcounter{enumi}{2}
        \item\label{property:closed-prime} Fix a local section $\sigma: U \to L_U$, and define
        \eqnn
        K_\sigma := \{\tilde x \in L_U \text{ such that $\sigma(\pi(\tilde x)) \leq \tilde x \leq \sigma(\pi(\tilde x))+1$.} \}
        \eqnd
        Then the restriction of $\pi$ to $K_\sigma$ is a closed map to $U$.
\end{brokenequiv}
\end{lemma}

\begin{proof}
In what follows, we let $q: L_S \to L_S/\ZZ$ denote the quotient map, and we let $p: L_S/\ZZ \to S$ denote the projection map induced from $\pi$. 

\ref{property:quotient-closed} $\implies$ \ref{property:closed-prime}: 
Fix a local section $\sigma: U \to L_U$. 
Consider the $2:1$ cover $L_U/(2\ZZ) \to L_U/\ZZ$ (which is a closed map, being a finite covering), and note that $K_\sigma \subset L_U \to L_U/(2\ZZ)$ is closed. Thus the composite $K_\sigma \to L_U/(2\ZZ) \to L_U/\ZZ \to U$ is a closed map (the last map is closed by \ref{property:quotient-closed}). On the other hand, this composite is equal to the restriction of $\pi$ to $K_\sigma$.

\ref{property:closed-prime} $\implies$ \ref{property:quotient-closed}: Fix a closed subset $A \subset L_S/\ZZ$. It suffices to show that $p(A) \subset S$ is locally closed. So for any $x \in S$, choose a local section $\sigma: U \to L_U$ from an open set $U$ containing $x$. (One can choose such a section by~\ref{property:fibers paracyclic} and \ref{property:local lifting}.) Note that 
	\eqnn
	\pi(K_\sigma \cap q^{-1}(A))
	=
	p(A) \cap U.
	\eqnd
By \ref{property:closed-prime}, $\pi(A) \cap U$ is closed. That is, $\pi(A)$ is locally closed.
\end{proof}

The following should be reminiscent of the result that a continuous bijection is a homeomorphism provided properness and separatedness restrictions. Such are the roles that~\ref{property:quotient-closed} and~\ref{property:separated} play. 

\begin{prop}\label{prop:R-equivariant-is-homeo}
Fix two pairs $(\pi:L_S \to S,\mu)$ and $(\pi':L_S' \to S,\mu')$ both satisfying properties
\ref{property:fibers paracyclic},
~\ref{property:quotient-closed},
\ref{property:separated}, and
\ref{property:local lifting}. Consider a map $\tilde f: L_S \to L_S'$ such that $\pi = \pi' \circ \tilde f$. 

If $\tilde f$ is continuous, a bijection, and $\RR$-equivariant, then $\tilde f$ is a homeomorphism.
\end{prop}

\begin{proof}
It will suffice to show that $\tilde f$ is a closed map. So let $A \subset L_S$ be a closed subset. Fix some point $x' \in L_{S'}$ not contained in the image $\tilde f(A)$. We wish to exhibit an open set $V' \subset L_{S'}$ containing $x'$ but disjoint from $\tilde f(A)$.

Let $x =\tilde f^{-1}(x')$ and consider the fiber $L_{\pi(x)} \subset L_S$, which by~\ref{property:fibers paracyclic} is a broken paracycle. In particular, the closed subset $A \cap L_{\pi(x)}$ can be identified as a disjoint union of (possibly unbounded) closed intervals. We choose points $x_1, x_2 \in L_{\pi(x)}$ such that  $x_1 < x < x_2$ and such that the open interval $(x_1,x_2)$ is disjoint from $A \cap L_{\pi(x)}$. We may assume that both $x_1, x_2$ are not $\RR$-fixed points by the discreteness of $\RR$-fixed points~\ref{property:fixed-point-set-discrete-in-line}.

By~\ref{property:local lifting}, we can choose sections $\sigma_i: U \to L_U^\circ$ for some open $U \subset S$ containing $\pi(x)$, with $\sigma_i(\pi(x)) = x_i$. We claim that, by shrinking $U$ to an open subset $W$ if necessary, we may assume that
	\eqn\label{eqn:V-W}
	V = \{ \text{ $y \in L_W$ such that $\sigma_1(\pi(y)) < y < \sigma_2(\pi(y))$} \} \subset L_W
	\eqnd
is disjoint from $A$. To see this, for every $y \in L_{\pi(x)}$ in the closed interval $[x_1,x_2]$, choose an open subset $V_y \subset L_S$ which is disjoint from $A$. By compactness of the closed interval $[x_1,x_2]$, a finite collection $\{V_{y_i}\}$ may be chosen to cover $[x_1,x_2]$. Moreover, note that there exists a finite integer $N$ for which $x_2 \leq x_1 + N$. Thus by an $N$-fold application of \ref{property:closed-prime} and Lemma~\ref{lemma:closed-map-lemma} and shrinking $U$ if necessary, the restriction of $\pi: L_U \to U$ to the set
	\eqnn
	K = \{ \text{ $y \in L_U$ such that $\sigma_1(\pi(y)) \leq y \leq \sigma_2(\pi(y))$} \}
	\eqnd
is a closed map. Letting $A_i = K \cap V_{y_i}^c$,---here $V_{y_i}^c$ is the complement of $V_{y_i}$---each $\pi(A_i)$ is closed in $U$, and hence so is the finite union $\cup_i \pi(A_i)  \subset U$. It follows that the complement of $\cup_i \pi_i(A_i)$ inside $U$ is open, and this complement clearly contains $\pi(x)$. Taking the $W$ of~\eqref{eqn:V-W} to be this complement, our claim about $V$ follows.

In particular, we observe that the set
	\eqnn
	\tilde f(V')
	=
	\{ y' \in (\pi')^{-1}(W)
	\text{ such that $\tilde f \sigma_0(\pi'(y')) < y' < \tilde f \sigma_1(\pi'(y'))$ }
	\}
	\eqnd
is disjoint from $\tilde f(A)$, contains $x'$, and is open by Proposition~\ref{prop:section-orders-are-open}. This shows that $\tilde f(A)$ is closed.
\end{proof}

\begin{remark}
Suppose that $L_S$ is not a family of broken paracycles, but of broken lines. (See the definition in~\cite{broken}). Then Proposition~\ref{prop:R-equiv-is-isom-of-families} still holds; the proof needs only minor modification, namely when $x$ is either an initial or terminal vertex of a fiber. In the initial case there is no need to choose either $x_1$ and one may simply look at the open subset of points that are strictly less than $\sigma_2$; likewise for the terminal case.
\end{remark}

\clearpage

\section{A quick review of some categorical techniques}\label{section:categorical-review}

Because we will be using some techniques from the theory of $\infty$-categories, let us review them briefly to make our proofs accessible to a broader audience. Those familiar with classical category theory will find these techniques to be virtually identical to those from the classical realm.

Let us remind the reader that the term $\infty$-category is synonymous with weak Kan complex, or quasi-category. These are simplicial sets satisfying the inner horn filling condition. For a more thorough introduction we refer the reader to the appendix of~\cite{nadler-tanaka} or to~\cite{htt}.

\begin{notation}\label{notation. groupoid completion}
Given any quasi-category $\cC$, the associated groupoid completion can be modeled by (the singular complex of) the geometric realization $|\cC|$. Informally, $|\cC|$ is the $\infty$-category obtained by inverting every morphism of $\cC$.
\end{notation}

\begin{notation}
Fix a functor $F: \cC \to \cD$.
We will denote a colimit and limit of $F$ by the symbols
	\eqnn
	\colimarrow_{\cC} F	
	\qquad \text{and} \qquad
	\limarrow_{\cC} F,
	\eqnd
respectively. The (redundant) arrows are meant to invoke that an example of a colimit is a so-called ``injective limit'' in algebra, while an example of a limit is a so-called ``projective limit'' or an ``inverse limit.''
\end{notation}

\begin{defn}[Final functors]
We say a functor $f: K \to K'$ between $\infty$-categories is {\em final} (or {\em left final} when seeking to further avoid ambiguity) if for any functor $G: K' \to \cC$, restriction along $f$ preserves colimits. More precisely, denoting $(K')^{\triangleright}$ by the $\infty$-category obtained by affixing a terminal object to $K'$, $f$ is final if and only if the restriction map
	\eqnn
	f^*: \Ffun( (K')^{\triangleright}, \cC) \to \Ffun( K^{\triangleright}, \cC)
	\eqnd
sends colimit diagrams to colimit diagrams. 

Dually, we call $f$ {\em initial}, or {\em right initial}, if and only if the restriction map
	\eqnn
	f^*: \Ffun( (K')^{\triangleleft}, \cC) \to \Ffun( K^{\triangleleft}, \cC)
	\eqnd
sends limit diagrams to limit diagrams.
\end{defn}

\begin{remark}
There are other characterizations of final and initial functors; see Proposition~4.1.1.8 and Definition~4.1.1.1 of~\cite{htt}. For us the main examples of final functors will be right adjoints, and the main examples of initial functors will be left adjoints. See Remark~\ref{remark:right-adjoints-are-final}.
\end{remark}

\begin{warning}
The literature is inconsistent about the use of the words final and cofinal. What we have called final is often called cofinal, perhaps because ``co''final functors preserve ``co''limits. 

At the same time, what we have called final functors $f: K \to K'$ do indeed have a final intuition, in that all objects in $K'$ eventually map to an object of $f(K)$; informally, $f$ has ``final'' image in $K'$. This is our reason for preferring the term final.

To relax the conflict, we will redundantly use the term {\em left} final; we feel that the adjective ``left'' may help remove ambiguity by following the intuition that ``left'' things play well with colimits.\footnote{The only exception here is the notion of left exactness from homological algebra, which plays well only with finite limits.} There is of course the unavoidable and possibly confusing fact that the prime examples of left final functors are right adjoints. 
\end{warning}

\subsection{Kan extensions}
Let $\cC, \cD, \cE$ be $\infty$-categories.
Fix functors $F: \cC \to \cE$ and $j: \cC \to \cD$.
Fix also the data of a pair
	\eqnn
	(j_! F, \eta)
	\eqnd
where $j_! F: \cD \to \cE$ is a functor, and $\eta$ is a natural transformation from $F$ to the composition $j_! F \circ j$.

\begin{defn}
A pair $(j_! F, \eta)$ is called a {\em left Kan extension of $F$ along $j$} if it is {\em initial} among all such pairs. 

By abuse, we will often call $j_! F$ a left Kan extension, with the data of $\eta$ implicit.
\end{defn}

\begin{remark}
By usual universal property arguments, a left Kan extension is unique up to contractible space of natural equivalences. 
\end{remark}

\begin{example}\label{example. left Kan along a point}
Let $\cD = \ast$ be the trivial category. Then $j_!F$ computes the colimit of $F$.
\end{example}

\begin{remark}\label{remark. left Kan extension formula}
When $\cE$ has all colimits, a left Kan extension always exists for any $F$ and any $j$. Moreover, one may compute the value of the left Kan extension on any object of $\cD$ as follows:
	\eqnn
	j_! F (d) \simeq \colimarrow_{x \in \cC_{/d}} F(x).
	\eqnd
Explicitly, the colimit is taken over the slice category; an object is the data of an object $x \in \cC$ together with a morphism $jx \to d$ in $\cD$.
\end{remark}

\begin{remark}\label{example. colimit of left Kan extension is colimit}
Left Kan extensions respect composition in the $j$ variable. That is,
	\eqnn
	(j' \circ j)_! F \simeq j'_!(j_! F). 
	\eqnd
Combining this fact with Example~\ref{example. left Kan along a point}, we see that 
	\eqnn
	\colimarrow_{\cC} F \simeq  \colimarrow_{\cD} j_! F.
	\eqnd
That is, the colimit of a left Kan extension is the original colimit.
\end{remark}

\begin{defn}
Dually, a right Kan extension is a pair 
	\eqnn
	(j_* F, \epsilon)
	\eqnd
where $\epsilon: j_*F \circ j \to F$ is a natural transformation, and the pair is terminal among such data. 

Or, letting $j^{\op}: \cC^{\op} \to \cD^{\op}$ and $F^{\op}: \cC^{\op} \to \cE^{\op}$ be the induced functors on opposite categories, the right Kan extension may be expressed as a left Kan extension:
	\eqnn
	j_*F = ((j^{\op})_! F^{\op})^{\op}.
	\eqnd
\end{defn}

\begin{example}\label{example:left-Kan-extension-along-left-adjoints}
Suppose that $j$ happens to be a left adjoint. Letting $R$ be the right adjoint to $j$, we have from Remark~\ref{remark. left Kan extension formula} that
	\eqnn
	j_! F (d)
	\simeq
	\colimarrow_{x \in \cC_{/d}} F(x)
	\simeq
	\colimarrow_{x \in \cC_{/Rd}} F(x)
	\simeq
	F(Rd).
	\eqnd 
The final two equivalences are justified by the following: First, the slice category $x \in \cC_{/d}$ is the $\infty$-category of pairs
	\eqnn
	(x \in \cC,
	f: j(x) \to d).
	\eqnd
By adjunction, this is equivalent to the $\infty$-category of pairs
	\eqnn
	(x \in \cC,
	g: x \to Rd);
	\eqnd
that is, to the slice category $\cC_{/Rd}$. Second, this slice category obviously has a final object given by $Rd$ itself, hence the colimit over $\cC_{/Rd}$ is computed by evaluating at $Rd$.

To summarize: When $j$ is a left adjoint, left Kan extensions along $j$ are readily computed as
	\eqnn
	j_! F \simeq F \circ R
	\eqnd
where $R$ is the right adjoint to $j$.
\end{example}

\begin{remark}\label{remark:right-adjoints-are-final}
Example~\ref{example:left-Kan-extension-along-left-adjoints}, combined with 
Example~\ref{example. colimit of left Kan extension is colimit},
shows that any right adjoint is a left final functor.
\end{remark}

\begin{example}\label{example. right kan extensions as composition with left adjoint}
Dually, if $j$ is a right adjoint to some functor $L: \cD \to \cC$, then the right Kan extension of $F$ along $j$ can be computed as
	\eqnn
	j_* F (d)
	\simeq
	\limarrow_{x \in \cC_{d/}} F(x)
	\simeq
	\limarrow_{x \in \cC_{Ld/}} F(x)
	\simeq
	F(Lx).
	\eqnd
So we have that $j_* F \simeq F \circ L$.
\end{example}

\subsection{The Grothendieck construction}\label{section. grothendieck construction}

Fix a functor $\alpha: \cC \to \inftyCat$ to the $\infty$-category of $\infty$-categories. Then  one can construct an $\infty$-category $\tilde \cC_\alpha$ equipped with a map 
	\eqnn
	p_\alpha: \tilde \cC_\alpha \to \cC
	\eqnd
satisfying the property of being a {\em coCartesian fibration}. $p_\alpha$ is related to $\alpha$ in the following way: Given any $x \in \cC$, the fiber $p_\alpha^{-1}(x)$ is equivalent to the $\infty$-category $\alpha(x)$. Moreover, for any edge $f: x \to y \in \cC$ and any $\tilde x \in p_\alpha^{-1}(x)$---where one may think of $\tilde x$ as an object in $\alpha(x)$---there exists an edge $\tilde f: \tilde x \to \tilde y$ where $\tilde y$ may be identified as the image of $\tilde x$ under the functors $\alpha(x)$, and the edge $\tilde f$ (in a precise sense) exhibits $\tilde y$ as this image up to contractible choice of equivalence. (The lingo is that $\tilde f$ is a coCartesian lift of $f$.)

This passage between functors $\alpha$ and coCartesian fibrations $p_\alpha$ is called the Straightening/Unstraightening construction. See~\cite{htt} for details.

There is also a story for functors $\alpha: \cC^{\op} \to \inftyCat$; these give rise to {\em Cartesian} fibrations $p_\alpha: \tilde \cC_{\alpha} \to \cC^{\op}$.

These constructions are generalizations of the classical Grothendieck construction, so we will abuse terminology and refer to the passage from functors to fibrations as the Grothendieck construction.

\subsection{Localizations}

\begin{remark}
When we say {\em subcategory} below, we mean a sub-$\infty$-category.
\end{remark}

\begin{defn}
Let $\cC$ be an $\infty$-category and $W \subset \cC$ a subcategory. The {\em localization of $\cC$ along $W$} is the pushout of $\infty$-categories
	\eqnn
	\xymatrix{
	W \ar[r] \ar[d] & \cC \ar[d] \\
	|W| \ar[r] & \cC[W^{-1}]
	}
	\eqnd
where $|W|$ is the groupoid completion of $W$.  (Notation~\ref{notation. groupoid completion}.)
\end{defn}

\begin{remark}\label{remark. universal property of localization}
$\cC[W^{-1}]$ satisfies the following universal property: For any $\infty$-category $\cD$, precomposition along $\cC \to \cC[W^{-1}]$ induces a fully faithful inclusion
	\eqnn
	\Ffun(\cC[W^{-1}], \cD)
	\to
	\Ffun(\cC, \cD)
	\eqnd
whose essential image can be identified with those functors $\cC \to \cD$ sending morphisms in $W$ to equivalences in $\cD$.
\end{remark}

One use of localizations is in computing colimits of categories:

\begin{theorem}\label{theorem. colimit of categories as localization}
Let $\alpha: \cC \to \inftycat$ be a functor, and let $p_\alpha: \tilde \cC_\alpha \to \cC$ be the associated coCartesian fibration (see Section~\ref{section. grothendieck construction}).

Then the colimit of $\alpha$ may be computed as the localization of $\tilde \cC_\alpha$ along the subcategory of coCartesian edges.

Dually, if $\tilde \cD_\alpha \to \cC^{\op}$ is the Cartesian fibration classifying $\alpha$, the colimit of $\alpha$ may be computed as the localization of $\tilde \cD_\alpha$ along the subcategory of Cartesian edges.
\end{theorem}

We refer the reader to Corollary~3.3.4.3 of \cite{htt}. There, he proves the case for Cartesian fibrations. Given that, the case for coCartesian fibrations is proven as follows (though there are other proofs): 

\begin{proof}[Of the coCartesian case of Theorem~\ref{theorem. colimit of categories as localization}, given the Cartesian case]
Fix the functor $\alpha: \cC \to \inftycat$. We wish to compute its colimit by localizing the coCartesian fibration $p_\alpha: \tilde \cC_\alpha \to \cC$.

Consider the induced functor on opposite categories,
	\eqnn
	p_\alpha^{\op} : \tilde \cC_\alpha^{\op} \to \cC^{\op}.
	\eqnd
This is a Cartesian fibration classifying a {\em different} functor $\beta$, which informally sends $x \in \cC$ to the opposite category of $\alpha(x)$:
	$
	\beta(x) = \alpha(x)^{\op}.
	$
If $W \subset \tilde \cC_\alpha$ is the collection of coCartesian edges, then $W^{\op} \subset \tilde \cC_\alpha^{\op}$ is the collection of Cartesian edges. The localization along the Cartesian edges is a pushout of $\infty$-categories
	\eqnn
	\xymatrix{
	W^{\op} \ar[r] \ar[d] 
		&	\tilde \cC_{\alpha}^{op} \ar[d] \\
	|W^{\op}| \ar[r]
		& 	\tilde \cC_{\alpha}^{op}[(W^{\op})^{-1}].
	}
	\eqnd
We have equivalences
	\begin{align}
	(\tilde \cC_{\alpha}[W^{-1}])^{\op}
	& \simeq
		\tilde \cC_{\alpha}^{op}[(W^{\op})^{-1}] \nonumber\\
	& \simeq
		\colimarrow_{\cC} \beta \nonumber\\
	& =
		\colimarrow_{\cC} ( \alpha(-)^{\op} ) \nonumber	\\
	& \simeq
		\left(\colimarrow_{\cC} \alpha(-)\right)^{\op} \nonumber	
	\end{align}
where the second arrow is an equivalence by  Theorem~\ref{theorem. colimit of categories as localization}, and the other arrows are equivalences because $\op$ is an autoequivalence of the $\infty$-category $\inftycat$ (hence preserves pushouts---for the first line---and more generally preserves colimits indexed by $\cC$---for the last equivalence).

Again using that $\op$ is an autoequivalence, we conclude 
	\eqnn
	\tilde \cC_{\alpha}[W^{-1}]
	\simeq
	\colimarrow_{\cC} \alpha.
	\eqnd
\end{proof}

\begin{remark}
One of the fundamental asymmetries of life is the passage between the coCartesian fibration and Cartesian fibrations corresponding to $\alpha$. It sometimes happens that a localization of the Cartesian fibration is far easier to compute than that of the coCartesian fibration, or the reverse (even though both are equivalent). Indeed, in one of our applications, the Cartesian fibration is far easier to localize than the coCartesian fibration. (See Remark~\ref{remark:Conv-Deltasurj-adjunction}.)
\end{remark}

\begin{example}
Localizations often arise as left adjoints to fully faithful functors. That is, let $L: \cC \to \cD$ be left adjoint to a fully faithful inclusion $R: \cD \to \cC$. Let $W \subset \cC$ be the collection of morphisms that are sent to equivalences in $\cD$. Then for any $\infty$-category $\cE$, the collection of those functors 
	\eqnn
	\Ffun_{W^{-1}}(\cC,\cE) \subset \Ffun(\cC,\cE)
	\eqnd
sending $W$ to equivalences can be identified with those functors arising as right Kan extensions of functors $F: \cD \to \cE$ along $R$.
On the other hand, by Example~\ref{example. right kan extensions as composition with left adjoint}, a right Kan extension along $R$ is always equivalent to a precomposition by $L$. Thus we have that
	\eqnn
	\Ffun(\cD, \cE) \to
	\Ffun_{W^{-1}}(\cC,\cE)
	\qquad
	F \mapsto F \circ L
	\eqnd
is an equivalence. This proves that $\cD$ is a localization of $\cC$ along $W$, and the localization functor is exhibited by $L$ itself.
\end{example}

\clearpage

\section{A presentation of the stack $\brokenpara$}

\subsection{The stack $\brokenpara$}

\begin{defn}[$\brokenpara$]\label{defn. brokenpara}
We let $\brokenpara$ be the stack classifying families of broken paracycles.
\end{defn}

\begin{notation}\label{notation.Pt(brokenpara)}
Concretely, we let
	\eqnn
	\Pt(\brokenpara)
	\eqnd
denote the category whose objects are families of broken paracycles $L_S \to S$ and whose morphisms are maps as in Definition~\ref{defn. maps of families}. The forgetful functor
	\eqnn
	\Pt(\brokenpara)
	\to
	\Top,
	\qquad
	(L_S \to S) \mapsto S
	\eqnd
is a category fibered in groupoids. This identifies $\brokenpara(S)$ as the groupoid of families of broken paracycles over $S$; that is, a map
	\eqnn
	S \to \brokenpara
	\eqnd
is the same thing as specifying a family $L_S \to S$.
\end{notation}

Whether a pair $(\pi: L_S \to S, \pi)$ is a family of broken lines can be tested locally on $S$. Likewise a map of families of broken lines $L_S \to L_T$ may be constructed locally; hence we have:

\begin{prop}\label{prop:broken-is-a-stack}
$\brokenpara$ is a stack on the site of topological spaces (with values in groupoids).
\end{prop}

In this section, we present $\brokenpara$ as a colimit of more easily understood stacks.

Let us first give some motivation. Fix a topological group $G$. A section of a principle $G$-bundle trivializes the $G$-bundle, and (by definition) local sections always exist for $G$-bundles. This is what allows us to present the stack of $G$-bundles beginning with a cover $\ast \to BG$. 

The analogue of a trivializing section in our setting will be an {\em $I$-section}. We will show that all families of broken paracycles admit a local $I$-section (Lemma~\ref{lemma:local-I-sections-exist}); hence the stacks classifying families equipped with an $I$-section form a cover of $\broken$. We will denote these stacks by $\broken^I$.

When understanding fiber products of this cover over $\broken$, we will naturally be led to consider $I$-sections when $I$ is not only a parasimplex, but a paracyclic {\em preorder}; we will collect these into a category $\Preordpara$. (See Definition~\ref{defn:preordpara}.) When all is said and done, we will obtain a functor 
	\eqnn
	\brokenpara^\bullet: \Preordpara^{\op} \to \Stacks_{/\brokenpara},
	\qquad
	I \mapsto (\brokenpara^I \to \brokenpara)
	\eqnd
where the map $\brokenpara^I \to \brokenpara$ is the forgetful map sending a family equipped with an $I$-section to the underlying family.

Our goal in this section is to prove the following:

\begin{theorem}\label{theorem. brokenpara as a colimit}
The map
	\eqnn
	\colimarrow_{I \in \Preordpara^{\op}} \brokenpara^I \to \brokenpara
	\eqnd
is an equivalence in the $\infty$-category of $\Kan$-valued sheaves on $\Top$. that is, $\brokenpara$ is the colimit of the stacks $\brokenpara^I$.
\end{theorem}

\begin{remark}
What use is presenting a stack as a colimit of other stacks? 
We will see that each $\broken^I$ is (representable by) an honest topological space. We will prove this in Theorem~\ref{theorem. I-sections representable}, and this will eventually allow us to compute the $\infty$-categories of constructible sheaves on these spaces. Moreover, a colimit of stacks exhibits a limit of categories of sheaves; this will allow us to prove our main theorem, Theorem~\ref{theorem. main}.
\end{remark}

\subsection{Paracyclic preorders}

Let us now introduce the main combinatorial, organizing tools.

\begin{defn}\label{defn:Z-preorder}
Recall that a {\em preorder} $(I, \leq_I)$ is a set $I$ equipped with a relation $\leq \subset I \times I$ such that for any $i \in I$ we have $i \leq i$, and for any triplet $i,j,k \in I$, we have that $i \leq j \leq k \implies i \leq k$. A preorder is called linear if for any pair $i,j \in I$, we have $i \leq j$ or $j \leq i$ (or both).

A {\em $\ZZ$-equivariant} linear preorder is a countable, linear preorder $I$ equipped with a free, relation-preserving $\ZZ$-action $\ZZ \times I \to I$. That is,
	\begin{itemize}
	\item $i \leq j \implies i + 1 \leq j + 1$.
	\item $i \leq i + 1$.
	\end{itemize}
Finally, a {\em paracyclic preorder} is a $\ZZ$-equivariant preorder $I$ such that
	\begin{itemize}
	\item For any $i \in I$, $i$ is strictly less than $i+1$. That is, $i \leq i+1$, but $i+1 \not \leq i$.
	\end{itemize}
\end{defn}

\begin{notation}\label{notation:I-relation}
Given a preorder $I$, let $\sim_I$ denote the equivalence relation where $i \sim_i i'$ if and only if $i \leq i'$ and $i' \leq i$. 
\end{notation}

\begin{defn}\label{defn. essentially surjective}
We say that a map of preorders $r: I \to J$ is {\em essentially surjective} if the induced function $I/\sim_I \to J/\sim_J$ is a surjection.
\end{defn}

\begin{definition}\label{defn:preordpara}
We let $\PreordZ$ denote the category whose objects are (countable) $\ZZ$-equivariant linear preorders, and whose morphisms are $\ZZ$-equivariant maps $r: I \to J$ weakly preserving $\leq$, and which are essentially surjective.

We let $\Preordpara \subset \PreordZ$ denote the  full subcategory whose objects are paracyclic preorders.
\end{definition}

\subsection{$I$-sections}

\begin{defn}\label{defn:I-section}
Let $I$ be a paracyclic preorder and $L_S \to S$ a family of broken paracycles.

An {\em $I$-section} is a map $\sigma: S \times I \to L_S^\circ$ such that
\enum
\item\label{item:ISectionIsASection} For each $i \in I$, the map $s \mapsto \sigma(s,i)$ is a continuous section,
\item\label{item:ISectionOntoPiNought} For every $s \in S$, the map $I \to (L_s^{\circ})/\RR$ is a surjection,
\item\label{item:ISectionDistanceOrderCompatible} For every $s \in S$, if $i \geq i'$ in $I$, then $d(\sigma(s,i) , \sigma(s,i')) > -\infty$ (see Notation~\ref{notation. distance}), and 
\item\label{item:ISectionZEquivariant} $\sigma$ is $\ZZ$-equivariant. That is, $\sigma(s,i +1) = \sigma(s,i) + 1$.
\enumd
\end{defn}

\begin{remark}\label{remark:I-sections-empty-for-groupoids}
Let $L_S \to S$ be a family of broken paracycles and let $I$ be a $\ZZ$-equivariant linear preorder. We have the relation $\sim_I$ as defined in Notation~\ref{notation:I-relation}. Note that one could define the notion of an $I$-section in this generality; however, if the relation $\sim_I$ is trivial---in the sense that $(I / \sim_I) \cong \ast$ is a single point---then the set of $I$-sections of $L_S \to S$ is empty. Note that this is compatible with Remark~\ref{remark:I-sections-empty-for-groupoids}.
\end{remark}

\begin{remark}[$I$-sections pull back]\label{remark. I sections pullback}
Let $\sigma: T \times I \to L_T$ be an $I$-section as in Definition~\ref{defn:I-section} and let fix a map of families of broken paracycles as in Definition~\ref{defn. maps of families}. Then the induced map $S \times I \to L_S$ is an $I$-section.
\end{remark}

\begin{remark}[Functoriality in $I$ variable]\label{remark: I sections functorial}
Let $L_T \to T$ be a family of broken paracycles, let $\sigma: T \times J \to L_T$ be a $J$-section, and fix an essentially surjective map $r: I \to J$. Then the induced map $\sigma \circ (\id_T \times r) : T \times I \to L_T$ is an $I$-section.
\end{remark}

$I$-sections give us a simple criterion for concluding certain maps are isomorphisms of families:

\begin{prop}\label{prop:R-equiv-is-isom-of-families}
Fix a paracyclic preorder $I$, and fix two families of broken paracycles equipped with $I$-sections
	\eqnn
	(L_S \to S, \sigma)
	\qquad
	\text{and}
	\qquad
	(L_T \to T, \tau).
	\eqnd
If we have a commutative diagram of continuous maps
	\eqnn
	\xymatrix{
	L_S \ar[r]^{\tilde f} \ar[d] & L_T \ar[d] \\
	S \ar[r]^f &T
	}
	\eqnd
such that
	\enum
	\item  $f$ is a homeomorphism, 
	\item $\tilde f$ is $\RR$-equivariant, and
	\item for all $i \in I$, we have
		$
		\tilde f \circ \sigma_i = \tau_i \circ f
		$.
	\enumd
Then $(f,\tilde f)$ is an isomorphism of families of broken paracycles. (That is, $\tilde f$ is also a homeomorphism.)
\end{prop}

\begin{proof}
By Proposition~\ref{prop:R-equivariant-is-homeo}, all we need to prove is that $\tilde f$ is a bijection on each fiber.

The map $\tilde f$ is a surjection as follows: First suppose $y \in L_T^\circ$ is not a fixed point. By Definition~\ref{defn:I-section} there is some $i \in I$ such that $\tilde f (\sigma_i(f^{-1}(\pi(y))))$ is in the $\RR$-orbit of $y$. By $\RR$-equivariance, $y$ is thus in the image of $\tilde f$. By continuity, any $\RR$-fixed point of $L_T$ is also in the image of $\tilde f$.

To prove $\tilde f$ is an injection, suppose $\tilde f(x) = \tilde f(x')$. Because $\tilde f$ is non-decreasing by Remark~\ref{remark:R-equivariance-is-order-preserving}, if $x \neq x'$, assume $x < x'$ without loss of generality. We conclude that the entire closed interval $[x,x']$ is collapsed under $\tilde f$; by $\RR$-equivariance, the $\RR$-orbit of $[x,x']$ is collapsed to an $\RR$-fixed point of $L_T$. The $\RR$-orbit of $[x,x']$ must contain the image of some $\sigma$ because $\sigma$ surjects onto $L_S^\circ/\RR$; we arrive at a contradiction, because $\tilde f$'s compatibility with $\sigma$ and $\tau$ violates the hypothesis that $\tau$ has image in $L_T^\circ$. Thus $\tilde f$ is an injection.
\end{proof}

\begin{notation}[$\brokenpara^I$]
Let us denote by
	\eqnn
	\brokenpara^I
	\eqnd
the stack which assigns to a space $S$ the groupoid of pairs $(L_S\to S, \sigma)$ where $L_S \to S$ is a family of broken paracycles over $S$ and $\sigma$ is an $I$-section.

Concretely,
	\eqnn
	\Pt(\brokenpara^I)
	\eqnd
has objects given by pairs $(L_S \to S, \sigma: S \times I \to L_S)$. A morphism to $(L_T \to T, \tau: T \times I \to L_T)$ is given by a map as  in Definition~\ref{defn. maps of families},
	\eqnn
	\xymatrix{
	L_S \ar[r]^{\tilde f} \ar[d] & L_T \ar[d] \\
	S \ar[r]^f & T
	}
	\eqnd 
such that the map is compatible with $\sigma$ and $\tau$; that is,
	\eqnn
	\tau\circ (f \times \id_I) = \tilde f \circ \sigma.
	\eqnd
\end{notation}

\begin{remark}\label{remark:Preord^I-functor}
We have a functor
	\eqnn
	\Preordpara^{\op} \to \Stacks_{/\brokenpara},
	\qquad
	I \mapsto (\brokenpara^I \to \brokenpara).
	\eqnd
This follows from Remark~\ref{remark: I sections functorial}.
\end{remark}

\begin{remark}\label{remark:PreordIleftKanextension}
Moreover, this functor extends to
	\eqnn
	\PreordZ^{\op} \to \Stacks_{/\brokenpara}
	\eqnd
by left Kan extension along the inclusion $\Preordpara^{\op} \to \PreordZ^{\op}$. Concretely, if $I$ is such that $I/\sim_I$ is a singleton (see Notation~\ref{notation:I-relation}), then $\brokenpara^I = \emptyset$ is the empty stack.
\end{remark}

\subsection{$I$-sections always exist locally}

To exhibit the map $\ast \to BG$ as a cover of the stack $BG$, one invokes the local triviality condition of $G$-bundles. (Any $G$-bundle is locally trivial, hence any map $X \to BG$ locally factors through a point.) When writing the \v{C}ech nerve for this cover and exhibiting $BG$ as the associated colimit, this covering property is used to prove that, for any test space $S$, the colimit map is an essentially surjective functor between the groupoids associated to $S$.

We will likewise claim that the natural maps $\brokenpara^I \to \brokenpara$ form a cover of $\brokenpara$.\footnote{In fact, there are ``relations;'' we will exhibit $\brokenpara$ as a colimit indexed over all $I$, and there will be non-trivial maps between $I$.} To this end, we prove:

\begin{lemma}\label{lemma:local-I-sections-exist}
Let $L_S \to S$ be a family of broken paracycles. Then for any $s \in S$, there exists a parasimplex $I$ and a neighborhood $s \in U$ such that $L_U$ admits an $I$-section.
\end{lemma}

\begin{proof}[Proof of Lemma~\ref{lemma:local-I-sections-exist}.]
Fix $s \in S$ once and for all and choose an arbitrary point $\tilde s \in L_s^\circ$. Because $L_s$ is a broken paracycle, we know that there exists a parasimplex $I$ admitting an order-preserving surjection to $\pi_0(L_s^\circ) \cong L_s^\circ / \RR$. Choose a $\ZZ$-equivariant lift of this surjection:
	\eqnn
	\sigma : \{s \} \times  I \to L_s^\circ.
	\eqnd
By the local lifting property, for each $i \in I$, each $\sigma(s,i)$ extends to a local lift
	\eqnn
	\sigma(-,i) : U_{s,i} \to L_{U_{s,i}}^\circ.
	\eqnd
for some open set $U_{s,i} \subset S$. Let us consider only finitely many $i$ (for example, all $i$ between $i_0$ and $i_0 +1$ for some $i_0 \in I$), and take the intersection of the $U_{s,i}$ to obtain a single open set $U$ for which $\sigma: U \times I \to L_U^\circ$ is a $\ZZ$-equivariant, continuous map respecting the projection to $U$. So far our construction satisfies \eqref{item:ISectionIsASection} and \eqref{item:ISectionZEquivariant} of Definition~\ref{defn:I-section}.

Because the distance function is continuous (Lemma~\ref{lemma. distance continuous}), we may shrink $U$ further to assume that whenever $i \leq j$, we have that 
	\eqnn
	d(\sigma(-,i) , \sigma(-,j)) > - \infty.
	\eqnd
This ensures our section $U \times I \to L_U$ also satisfies~\eqref{item:ISectionDistanceOrderCompatible}. 

Finally, by shrinking $U$ as necessary, by Property~\ref{property:unramified}, we may write
	\eqnn
	L_U^\RR \cong \coprod_{n \in \ZZ} n + (K_0 \coprod \ldots  \coprod K_m)
	\eqnd
where each projection $K_i \to U$ is a closed embedding. By shrinking $U$ again as necessary, we may assume that each $K_i$ intersects $L_s$. In particular, for every $s' \in U$, the composite map
	\eqnn
	I \xra{\sigma(s', - ) } L_{s'}^\circ \to L_{s'}^\circ / \RR
	\eqnd 
is a surjection. So now our construction satisfies Property~\ref{item:ISectionOntoPiNought}. 

This completes the proof.
\end{proof}

\begin{remark}
We did not use Property~\ref{property:quotient-closed} in the above proof.
\end{remark}

\subsection{Pairs of $I$-sections}

\begin{defn}\label{defn:AmalgIJ}
Fix two $\ZZ$-equivariant preorders $I$ and $J$. An {\em amalgam} of $I$ and $J$ is a $\ZZ$-equivariant preorder $K$ equipped with a weakly order-preserving, $\ZZ$-equivariant injection $I \coprod J \to K$.

Fixing $I$ and $J$, we let $\Amalg(I,J)$ denote the category of amalgams of $I$ and $J$.

For convenience, we will model $\Amalg(I,J)$ as a poset as follows: We restrict attention to those $K$ whose underlying set is equal to $I \coprod J$, and the map $I \coprod J \to K$ is the identity on underlying sets. The poset structure is obtained by declaring $K \leq K'$ if and only if $\leq_K \subset \leq_{K'}$; i.e., if and only if $x \leq_K y \implies x \leq_{K'} y$.
\end{defn}

\begin{defn}\label{defn:AmalgTilde}
Let $\widetilde\Amalg$ denote the category whose objects are triplets $(I,J,K)$ where $K$ is an object of $\Amalg(I,J)$, and a morphism is a pair of $\ZZ$-equivariant maps of preorders
	\eqnn
	I \to I',
	\qquad
	J \to J'
	\eqnd
such that the induced map $K \to K'$ is a map of preorders.
\end{defn}

\begin{remark}\label{remark: forgetful amalga}
We have forgetful functors
	\eqnn
	 \PreordZ \leftarrow \widetilde\Amalg \to \PreordZ \times  \PreordZ
	\eqnd
which on objects acts by
	\eqnn
	K \mapsfrom (I,J,K) \mapsto (I,J).
	\eqnd
The fiber over $(I,J)$---of the righthand functor---is identified with $\Amalg(I,J)$.
\end{remark}

\begin{remark}\label{remark:IJ slice left final}
Fix $(I,J) \in \PreordZ \times  \PreordZ$ and consider the slice category
	\eqnn
	(\widetilde\Amalg)_{(I,J)/}.
	\eqnd
Then the inclusion 	
	\eqnn
	\Amalg(I,J) \to (\widetilde\Amalg)_{(I,J)/}
	\eqnd
is (right) initial; in particular, the opposite map
	\eqnn
	\Amalg(I,J)^{\op} \to (\widetilde\Amalg^{\op})_{/(I,J)}
	\eqnd
is (left) final.
\end{remark}

\begin{remark}\label{remark: diagonal amalgam adjoint}
Moreover, the opposite of the forgetful functor $(I,J,K) \mapsto K$ from Remark~\ref{remark: forgetful amalga} admits a right adjoint given by sending $I$ to $(I,I, I \coprod I)$ with the obvious preorder. 
\end{remark}

\begin{lemma}\label{lemma.Amalg-colimit}
Consider the composition
	\eqnn
	\Amalg(I,J) \to 
	\PreordZ
	\to
	\Stacks,
	\qquad
	(I \coprod J \to K) \mapsto K \mapsto \brokenpara^K.
	\eqnd
The induced map
	\eqnn
	\colimarrow_{K \in \Amalg(I,J)} \brokenpara^K \to \brokenpara^I \times_{\brokenpara} \brokenpara^J
	\eqnd
is an equivalence of stacks. 
\end{lemma}

\nc{\join}{\vee}

	The proof of the Lemma~\ref{lemma.Amalg-colimit} relies on the following, which is a stack version of Lemma~3.6.6 of~\cite{broken}.

	\begin{prop}\label{prop.open-covers-of-stacks}
	Fix a stack $X$.
	Let $\cA$ be a poset admitting joins\footnote{That is, for any non-empty finite collection $a_1,\ldots,a_k \in \cA$, there is a least element $a$ such that $a \geq a_i$ for all $i$. }, and let $X_\bullet: \cA^{\op} \to \shv_{\Kan}(\Top)_{/X}$ be a functor such that 
	\enum
	\item $X_\bullet$ respects meets. That is, for every $U,V \in \cA$, the natural map $X_{\join(U,V)} \to X_U \times_X X_V$ is an equivalence.
	\item $X_\bullet$ is locally surjective. That is, for every topological space $S$, every map $S \to X$, and every $s \in S$, there is some element $U \in \cA$ and some open set $S' \subset S$ containing $s$ such that the diagram
		\eqnn
		\xymatrix{
		S' \ar[r] \ar[d] & X_U \ar[d] \\
		S \ar[r] & X
		}
		\eqnd
	commutes.
	\enumd
	Then the natural map
		\eqnn
		\varinjlim_{\cA^{\op}} X_U \to X
		\eqnd 
	is an equivalence.
	\end{prop}

	\begin{proof}[Proof of Proposition~\ref{prop.open-covers-of-stacks}.]
	By assumption (2), we must merely show that the map is a monomorphism: That is, we must show that the relative diagonal
		\eqnn		
		\colimarrow_{\cA^{\op}} X_U \to 
		\left(\colimarrow_{\cA^{\op}} X_U\right) \times_{X}
		\left(\colimarrow_{\cA^{\op}} X_U\right) 
		\eqnd
	is an equivalence. Because $\shv_{\Kan}(\Top)$ is an $\infty$-topos, and in particular locally Cartesian closed, we have a natural equivalence
		\eqnn
		\colimarrow_{(U,V) \in (\cA \times \cA)^{\op}}  X_U \times_X X_V
		\to
		\left(\colimarrow_{U \in \cA^{\op}} X_U\right)  \times_{X}
		\left(\colimarrow_{V \in \cA^{\op}} X_V\right).
		\eqnd
	Note moreover that the functor $\cA \times \cA \to \cA$ sending $(U,V) \mapsto \join(U,V)$ is a left adjoint. (The right adjoint is given by sending $W \mapsto (W,W)$.) Hence the induced functor on opposite categories is a right adjoint, and final. This means the natural map
		\eqnn
		\colimarrow_{W \in \cA^{\op}} X_W
		\to
		\colimarrow_{(U,V) \in (\cA \times \cA)^{\op}}  X_U \times_X X_V
		\eqnd
	is an equivalence. Tracing through the definitions of the chains of equivalences, we are left to proving that the induced map
		\eqnn
		\colimarrow_{W \in \cA^{\op}} X_W \to 
		\colimarrow_{W \in \cA^{\op}} X_W
		\eqnd
	is an equivalence; this is obvious, as this map is induced by pulling back the functor $X_\bullet$ along the identity functor $\cA \to \cA$. 
	\end{proof}

\begin{proof}[Proof of Lemma~\ref{lemma.Amalg-colimit}.]
By definition, the points of the stack
	\eqnn
	\brokenpara^I \times_{\brokenpara} \brokenpara^J
	\eqnd
are given by a triplet $(L_S \to S, \sigma_I, \sigma_J)$ where $L_S \to S$ is a family of broken paracycles and the $\sigma_I$ and $\sigma_J$ are $I$- and $J$-sections, respectively. For brevity, let $\sigma: S \times (I \coprod J) \to L_S^\circ$ be the combined section. One can now define a preorder on the set $I \coprod J$ as follows:
	\eqn\label{eqn. preorder induced}
	a \leq b \qquad \iff \qquad \sigma(a) \leq \sigma(b)  \text{ or }  d(\sigma(a),\sigma(b)) \in (-\infty,\infty).
	\eqnd
By definition of $I$-section, this is a preorder for which the inclusions $I \to I \coprod J$ and $J \to I \coprod J$ are both $\ZZ$-equivariant preorder maps that are essentially surjective. In particular, setting $K = I \coprod J$ with the preorder as in~\eqref{eqn. preorder induced}, $\sigma$ defines a $K$-section. This shows that (2) of Proposition~\ref{prop.open-covers-of-stacks} is satisfied. 

We now verify (1). For $K, K' \in \Amalg(I,J)$, the fiber product
	\eqn\label{eqn.fiber product K K'}
	\brokenpara^K
	\times_{\brokenpara^I \times_{\brokenpara} \brokenpara^J}
	\brokenpara^{K'}
	\eqnd
has points given by triplets $(L_S \to S, \sigma_K, \sigma_{K'})$ such that $L_S \to S$ is a family of broken paracycles, and where the functions $\sigma_K, \sigma_{K'}$ {\em agree} when given the identification of sets $K = I \coprod J = K'$; that is, $\sigma_K = \sigma_{K'}$ as functions. We may construct a new preorder structure $K''$ on $I \coprod J$ by the same definition as in~\eqref{eqn. preorder induced}; this guarantees that $(\leq_{K} \cup \leq_{K'}) \subset \leq_{K''}$, so that we have a map from \eqref{eqn.fiber product K K'} to $\broken^{K \join K'}$. On the other hand, clearly a $K \join K'$-section restricts to both a $K$-section and $K'$-section (with the same underlying function $\sigma$), so we obtain a map from 
$\broken^{K \join K'}$
to
\eqref{eqn.fiber product K K'}. It is easily checked that these are mutually inverse.
\end{proof}

\subsection{Proof of Theorem~\ref{theorem. brokenpara as a colimit}.}\label{section.proof-of-colimit-theorem}
We first prove 
	\eqn\label{eqn:PreordZ equivalence}
	\colimarrow_{I \in \PreordZ^{\op}} \brokenpara^I \to \brokenpara
	\eqnd
is an equivalence. It is clearly a local surjection by Lemma~\ref{lemma:local-I-sections-exist} so we need only prove that this map is a monomorphism. That is, we must prove that the relative diagonal
	\eqn\label{eqn:relativeDiagonal}
	\colimarrow_{I \in \PreordZ^{\op}} \brokenpara^I
	\to
	\left( \colimarrow_{I_0\in \PreordZ^{\op} } \brokenpara^{I_0} \right) \times_{\brokenpara}
	\left( \colimarrow_{I_1\in \PreordZ^{\op} } \brokenpara^{I_1} \right) 
	\eqnd
is an equivalence. Let us consider the arrows
	\begin{align}
   &  \left( \colimarrow_{I_0 \in \PreordZ^{\op}} \brokenpara^{I_0} \right) \times_{\brokenpara}
    	\left( \colimarrow_{I_1 \in \PreordZ^{\op}} \brokenpara^{I_1} \right) \nonumber \\
	& \xra{\text{loc. Cart. closed}}   \colimarrow_{I_0, I_1 \in \PreordZ^{\op} \times  \PreordZ^{\op}} \brokenpara^{I_0}  \times_{\brokenpara} \brokenpara^{I_1}  \nonumber \\
	&  \xra{\text{Lemma~\ref{lemma.Amalg-colimit}}}    \colimarrow_{I_0, I_1} 		\colimarrow_{K \in \Amalg(I_0,I_1)^{\op}} \brokenpara^K \nonumber\\
	&  \xra{\text{Remark~\ref{remark:IJ slice left final}}}    \colimarrow_{I_0, I_1} 		\colimarrow_{(J_0,J_1,K) \in \widetilde \Amalg^{\op}_{/(I_0,I_1)}} \brokenpara^K \nonumber\\
	&  \xra{\text{Left Kan ext.}}    \colimarrow_{(I_0,I_1,K) \in \widetilde \Amalg^{\op}} \brokenpara^K \nonumber\\
	&  \xleftarrow{\text{Remark~\ref{remark: diagonal amalgam adjoint}}}    \colimarrow_{I \in \PreordZ^{\op}} \brokenpara^I. \nonumber \\
	\label{eqn:bigComposition}
	\end{align}
We claim every arrow drawn is an equivalence.
The first arrow is an equivalence because the $\infty$-category of stacks on $\Top$ is an $\infty$-topos, and in particular locally Cartesian closed. The next arrow is an equivalence by Lemma~\ref{lemma.Amalg-colimit}. The following arrow is an equivalence by Remark~\ref{remark:IJ slice left final}. The next equivalence follows by left Kan extension along the functor  $\widetilde \Amalg^{\op} \to \PreordZ^{\op} \times  \PreordZ^{\op}$ (as defined in Remark~\ref{remark: forgetful amalga}).\footnote{See also Remark~\ref{example. colimit of left Kan extension is colimit}.} The final arrow is also an equivalence by the adjointness mentioned in Remark~\ref{remark: diagonal amalgam adjoint}. 

Tracing through the composition of~\eqref{eqn:relativeDiagonal} with the arrows from~\eqref{eqn:bigComposition}, we find that the composition of the arrows is the map
	\eqnn
	\colimarrow_{I \in \PreordZ^{\op}} \brokenpara^I\to
	\colimarrow_{I \in \PreordZ^{\op}} \brokenpara^I
	\eqnd
induced by the identity functor $ \PreordZ^{\op} \to  \PreordZ^{\op}$; that is, the entire composition is an equivalence. This proves that~\eqref{eqn:relativeDiagonal} itself is an equivalence, which concludes the proof that~\eqref{eqn:PreordZ equivalence} is an equivalence.

So now let us prove that the natural map
	\eqnn
	\colimarrow_{I \in \Preordpara^{\op}} \brokenpara^I
	\to
	\brokenpara
	\eqnd
is an equivalence.

By Remark~\ref{remark:PreordIleftKanextension}, the functor $\PreordZ^{\op} \to \Stacks$ is a left Kan extension of the functor $\Preordpara^{\op} \to \Stacks$, hence the colimit of the two functors agree. This completes the proof of the theorem.

\clearpage

\section{Representability of the space of $I$-sections}

We have presented $\brokenpara$ as a colimit of other stacks; as we've mentioned before, we will now show that these other stacks are actually (represented by) topological spaces. This identification will be used to compute the $\infty$-category of sheaves on $\brokenpara$.

To perform our eventual computation, however, we will need to prove that the stacks in the colimit diagram of Theorem~\ref{theorem. brokenpara as a colimit} are spaces in a way compatible with the colimit diagram.

To that end, the goal of this section is to prove the following:

\begin{theorem}\label{theorem. I-sections representable}
For any paracyclic preorder $I \in \Preordpara$, 
$\brokenpara^I$ is representable by a topological space $F^{(I)}$. In fact,
\enum
\item\label{thmitem:representableOverBroken} For every $I$ we have a diagram of categories
	\eqnn
	\xymatrix{
	\Pt(\brokenpara^I) \ar[rr]^{\simeq} \ar[dr] && \Pt(F^{(I)}) \ar[dl] \\
	& \Pt(\brokenpara)
	}
	\eqnd
which commutes up to natural isomorphism and respects Cartesian edges over the base category $\Top$. That is, the above diagram exhibits an equivalence in the $\infty$-category of stacks equipped with a map to $\brokenpara$.
\item\label{thmitem:representabilityFunctorialInI} One can construct this diagram functorially in the $I$ variable. That is, one has an equivalence
	\eqnn
	\brokenpara^\bullet \simeq F^{(\bullet)}
	\eqnd 
where each of the above is treated as a functor 
	\eqnn
	\Preordpara^{\op} \to \Stacks_{/\brokenpara}
	\eqnd
from $\Preordpara^{\op}$ to the $\infty$-category of stacks equipped with maps to $\brokenpara$.
\enumd
\end{theorem}

Thus, combined with Theorem~\ref{theorem. brokenpara as a colimit}, we have the following:

\begin{cor}\label{corollary:brokenpara-as-colim-of-spaces}
$\brokenpara$ is a colimit of the topological spaces $F^{(I)}$. That is, the arrow
	\eqnn
	\colimarrow_{I \in \Preordpara^{\op}} F^{(I)} \to \brokenpara
	\eqnd
is an equivalence.
\end{cor}

\subsection{The local models $\tilde F^{(I)} \to F^{(I)}$}\label{section.F^{(I)} construction}

\begin{remark}
Recall we have the distance function from Lemma~\ref{lemma. distance continuous}. The main idea of the proof of Theorem~\ref{theorem. I-sections representable} is to take any $I$-section $\sigma$ and to pairwise measure the translation distances between the images of the sections. This gives us a sequence of elements in $[-\infty,\infty]$ satisfying certain properties, hence defines a subset of some product of $[-\infty,\infty]$. We denote this subset by $F^{(I)}$ (see Definition~\ref{defn:F^{(I)}} below).
\end{remark}

\nc{\arr}{\mathsf{Arr}}

\begin{defn}\label{defn:F^{(I)}}
Let $I$ be a paracyclic preorder and let $\arr(I)$ denote the set of all pairs $(i,j) \in I$ such that $i \leq j$. We let
	\eqnn
	F^{(I)} \subset (-\infty,\infty]^{\arr I}
	\eqnd
denote the set of those $\alpha$ satisfying the following conditions:
\enum
	\item\label{item:F^{(I)}-functor-1} $\alpha(i,i) = 0$.
	\item\label{item:F^{(I)}-functor-2} $\alpha(i,j) +\alpha(j,k) = \alpha(i,k)$ for all $i \leq j \leq k \in I$,
	\item\label{item:F^{(I)}-Zequivariant} $\alpha(i,j) = \alpha(i+1,j+1)$ for all $i \leq j \in I$, and
	\item\label{item:F^{(I)}hasfixedpoints} $\alpha(i,i+1)= \infty$.
\enumd
We endow $F^{(I)}$ with the topology inherited from the usual topology on the infinite product $(-\infty,\infty]^{\arr(I)}$.
\end{defn}

\begin{remark}
Let $B(-\infty,\infty])$ denote the category with a single object, whose endomorphism set is given by $(-\infty,\infty]$ under addition. Then one can think of 
	\eqnn
	F^{(I)} \subset \Ffun(I,B(-\infty,\infty])
	\eqnd
as a subset of the collection of functors from the preorder $I$ (thought of as a category) to $B(-\infty,\infty]$---this is the content of conditions~\eqref{item:F^{(I)}-functor-1} and~\eqref{item:F^{(I)}-functor-2}.

From this perspective, $F^{(I)}$ is topologized by endowing the morphism space of  $B(-\infty,\infty]$ with the usual topology via the subset topology for $(-\infty,\infty] \subset [-\infty,\infty]$.

Finally, one should think of $\alpha(i,j)$ as the possible ``distance functions'' arising from particular kinds of sections $I \times S \to L_S^\circ$ of a family $L_S \to S$. Condition~\eqref{item:F^{(I)}-Zequivariant} amounts to requiring that these sections be $\ZZ$-equivariant, and Condition~\eqref{item:F^{(I)}hasfixedpoints} is a consequence of $L_S$ being a family in the sense of Definition~\ref{defn:family-of-broken-paracycles}: Every point $\tilde s \in L_S^\circ$ has at least one fixed point between it and its $+1$-translate, hence the distance function will always attain an $\infty$ between a $\tilde s$ and its translate.
\end{remark}

\begin{remark}\label{remark:F^{(I)} are corners}
Fix $I \in \Preordpara$ and choose an element $i_0 \in I$. 

Suppose $I$ is a poset, rather than a preorder; then the identification
	\eqnn
	[n_I]
	\cong
	\{
	\text{$i$ such that $i_0 \leq i < i_0 +1$}
	\}
	=
	\{
	i_0 < i_1 < i_2 < \ldots < i_{n_I}
	\}
	\subset I
	\eqnd
shows that $F^{(I)}$ is homeomorphic to a ``corner'' of a cube. That is, the restriction map
	\eqnn
	F^{(I_0)} \to (-\infty,\infty]^{[n_I]},
	\qquad
	\alpha
	\mapsto (\alpha(i_0,i_1),\ldots,\alpha(i_{n_I},i_0+1)).
	\eqnd
is a closed embedding whose image is the locus of points for which at least one coordinate equals $\infty$.

When $I$ is not a poset but a preorder, the same argument shows that $F^{(I)}$ in general is an open subset of the above corner. (For example, if $i_2 \leq i_1$ and $i_1 \leq i_2$ in $I$, then for any $\alpha \in F^{(I)}$, we have that $\alpha(i_1,i_2)$ must never attain $\infty$.)
\end{remark}

\begin{remark}\label{remark:FI is a functor}
Let $I'$ and $I$ be paracyclic preorders and let $f: I' \to I$ be any map which weakly respects orders and is $\ZZ$-equivariant. Then precomposition by $f$ induces a continuous map
	\eqnn
	f^*: F^{(I)} \to F^{(I')}.
	\eqnd
This defines a functor
	\eqnn
	\Preordpara^{\op} \to \Top,
	\qquad
	I \mapsto F^{(I)}.
	\eqnd
In terms of the description of the $F^{(I)}$ as unions of certain faces of the corner of a cube (Remark~\ref{remark:F^{(I)} are corners}), if $f$ is a surjection, then $f^*$ is always the inclusion of the locus of points for which certain coordinates are equal to 0.
\end{remark}

Now we exhibit a family of broken paracycles over $F^{(I)}$, equipped with an $I$-section. This will be another ingredient in the proof of Theorem~\ref{theorem. I-sections representable}--the idea is to ensure that not only does every $I$-section give rise to a continuous map to $F^{(I)}$, but every continuous map to $F^{(I)}$ determines a family of broken paracycles equipped with an $I$-section.

\begin{notation}[$\beta_i$]
In what follows, for any element $\beta \in [-\infty,\infty]^I$, we will write $\beta_i$ for the $i$th component of $\beta$.
\end{notation}

\begin{construction}\label{construction. family over F^{(I)}}
Consider the subset
	\eqnn
	\tilde F^{(I)}
	\subset 
	F^{(I)} \times 
	[-\infty,\infty]^I
	\eqnd
consisting of those pairs $(\alpha,\beta)$ such that
\enum
\item\label{item:F^{(I)}-addition-condition} for any $i \leq j \in I$, at least one of the three following equations is satisfied:
		\eqnn
		\beta_i = \alpha(i,j) + \beta_j,
		\qquad
		\alpha(i,j) = \beta_i - \beta_j,
		\qquad
		\beta_j = \beta_i - \alpha(i,j).
		\eqnd
(See Convention~\ref{convention:adding-subtracting-infinity} and Remark~\ref{remark:equations-for-infinity-additions}.)
\item\label{item:F^{(I)}-remove-extreme-corners} $\beta$ attains the values $\pm \infty$. That is, there exists some $i_-, i_+ \in I$ such that 
	\eqnn
	\beta_{i_-} = - \infty
	\qquad\text{and}\qquad
	\beta_{i_+} = \infty.
	\eqnd
\enumd

We endow $\tilde F^{(I)}$ with an action as follows: Given $(n,t) \in \ZZ \times \RR$, we declare
	\eqn\label{eqn:F^{(I)}-ZtimesR-action}
	((n,t)\beta)_i
	=
	\beta_{i+n} + t.
	\eqnd
(Informally, $n \in \ZZ$ shifts the indices of $\beta$, while $t \in \RR$ translate all the coordinates of $\beta$.)
This action respects the fibers of the projection map 
	\eqn\label{eqn:tilde-F^{(I)}-to-F^{(I)}}
	\pi:
	\tilde F^{(I)} \to F^{(I)},
	\qquad
	(\alpha,\beta) \mapsto \alpha.
	\eqnd
\end{construction}

\begin{example}\label{example:F^{(I)}-fibers-ij}
Fix $\alpha \in F^{(I)}$ and consider the fiber $(\tilde F^{(I)})_\alpha$, along with the projection
	\eqn\label{eqn:projectF^{(I)}fiber}
	(\tilde F^{(I)})_\alpha \to  [-\infty,\infty] \times [-\infty,\infty],
	\qquad
	(\alpha, \beta) \mapsto (\beta_i,\beta_j).
	\eqnd
Based purely on Condition~\eqref{item:F^{(I)}-addition-condition} of Construction~\ref{construction. family over F^{(I)}}, we can conclude the following:
\begin{itemize}
\item When $\alpha(i,j)=\infty$ the image of \eqref{eqn:projectF^{(I)}fiber} is a union of two faces of a square, namely the $\beta_j = -\infty$ and $\beta_i = \infty$ faces.
\item If $\alpha(i,j) < \infty$, the image of~\eqref{eqn:projectF^{(I)}fiber} is a ``line'', compactified with two corner points $(-\infty,-\infty)$ and $(\infty,\infty)$ of the two-dimensional cube.
\end{itemize}
Note that by definition of $F^{(I)}$ (Definition~\ref{defn:F^{(I)}}), $\alpha(i,j)$ never attains the value $-\infty$.
\end{example}

\begin{remark}\label{remark:maximal iplus iminus}
From Example~\ref{example:F^{(I)}-fibers-ij} we see that whenever $i \leq j \in I$ and $(\alpha,\beta) \in \tilde F^{(I)}$, we have
	\eqnn
	\beta_i = -\infty \implies \beta_j = - \infty,
	\qquad
	\beta_j = \infty \implies \beta_i = -\infty.
	\eqnd
This again relies only on Condition~\eqref{item:F^{(I)}-addition-condition}.

Thus, for a fixed $(\alpha,\beta) \in \tilde F^{(I)}$, if there is some $i_-$ for which $\beta_{i_-} = -\infty$, there is a minimal such $\beta_i$. Likewise, if there is some $i_+$ for which $\beta_{i_+} = \infty$, one may ask for the maximal such $i_+$. At the same time, because $\alpha(i,j)$ attains $\infty$ for some $i \leq j$ by design, Condition~\eqref{item:F^{(I)}-addition-condition} guarantees that at least one of $i_-$ and $i_+$ exists.

Condition~\eqref{item:F^{(I)}-remove-extreme-corners} guarantees that $\tilde F^{(I)}$ only consists of $(\alpha,\beta)$ where  $i_-$ and $i_+$ both exist. 
\end{remark}

\begin{remark}
Informally, one may think that the practical effect of Condition~\eqref{item:F^{(I)}-remove-extreme-corners} is to remove the extremal corners
	\eqnn
	(\beta_i = \infty)_{i \in I}
	\qquad
	(\beta_i = - \infty)_{i \in I}
	\eqnd
from consideration.
\end{remark}

\begin{remark}\label{remark:F^{(I)}-beta-has-support-i-andi+}
In particular, for any $(\alpha,\beta) \in \tilde F^{(I)}$, let us denote by $i_-$ and $i-+$ the maximal and minimal elements as in Remark~\ref{remark:maximal iplus iminus}. Then we may assume $\beta$ has ``support'' strictly between $i_-$ and $i_+$, meaning
	\eqnn
	i_+ < i < i_- \implies \beta_i \neq \pm \infty.
	\eqnd
\end{remark}

\begin{remark}\label{remark:FI-fiber-surjection}
Thus, fixing $\alpha_0$, a point $(\alpha_0,\beta)$ in the fiber $(\tilde F^{(I)})_{\alpha_0}$ can be characterized as follows:
\begin{itemize}
\item $(\alpha_0,\beta)$ is a fixed point for the $\RR$-action if and only if $i_-$ is the successor to $i_+$. (See Notation~\ref{notation:successor++} for the notion of successor.) As in Remark~\ref{remark:maximal iplus iminus}, we have chosen $i_-$ and $i_+$ to be maximal and minimal, respectively.
\item If $(\alpha_0,\beta)$ is not a fixed point for the $\RR$-action, choose any $i_0$ strictly between $i_+$ and $i_-$. Then $(\tilde F^{(I)})_{\alpha_0}$ contains an element $(\alpha_0,\beta')$ whose coordinates take on values
	\eqnn
	\beta'_{i} =
	\begin{cases}
	\infty & i \leq i_+ \\
	\alpha_0(i,i_0) & i_+ < i \leq i_0 \\
	-\alpha_0(i_0,i) & i \geq i_0 < i_- \\
	-\infty & i \geq i_-.
	\end{cases}
	\eqnd 
And our point $(\alpha_0,\beta)$ is in the $\RR$-orbit of $(\alpha_0,\beta')$. (In fact, the two are related by an $\RR$-translation by $\beta_{i_0}$.)
\end{itemize}
\end{remark}

\begin{lemma}\label{lemma. F^{(I)} family is a family}
The data $(\tilde F^{(I)} \to F^{(I)}, \mu)$ from Construction~\ref{construction. family over F^{(I)}} exhibits a family of broken paracycles over $F^{(I)}$.
\end{lemma}

\begin{proof}[Proof of Lemma~\ref{lemma. F^{(I)} family is a family}.]
We verify each property of Definition~\ref{defn:family-of-broken-paracycles}.

\ref{property:fibers paracyclic}
Given $\alpha \in F^{(I)}$, for each $i \in I$ consider the continuous map
	\eqnn
	\rho_i: [-\infty,\infty] \to \tilde F^{(I)}
	\eqnd
sending $t \in [-\infty,\infty]$ to the element whose $j$th coordinate is given by
	\eqnn
	\rho_i(t)_j
	=
	\begin{cases}
	t & i = j \\
	\infty & t = \pm \infty, j < i , \alpha(j,i) = \infty \\
	-\infty & t = \pm \infty, j > i, \alpha(i,j) = \infty \\
	t- \alpha(i,j)  & i \leq j, \alpha \neq \infty \\
	t + \alpha(j,i) & j \leq i, \alpha \neq \infty.
	\end{cases}.
	\eqnd
The $\rho_i$ define a single $\ZZ \times\RR$-equivariant map
	\eqnn
	\rho:
	\coprod_{i \in I} [-\infty,\infty]
	\to
	(\tilde F^{(I)})_\alpha.
	\eqnd
By Remark~\ref{remark:FI-fiber-surjection}, $\rho$ is a surjection.
Moreover, its image can be identified with a quotient of
	\eqnn
	\coprod_{i \in I} [-\infty,\infty].
	\eqnd
The equivalent relation giving rise to the quotient can be gleaned from the following observation: Suppose $i'$ is the successor to $i$ in $I$. Then $\rho_{i}$ and $\rho_{i'}$ have identical image, or overlap along $\rho_i(\infty)$ and $\rho_{i'}(-\infty)$.  

It is straightforward now to verify that the image of $\rho$ (hence the fiber above $\alpha$) is a broken paracycle.

\ref{property:unramified}
For any $i \in I$, we write $i^{++}$ for the successor of $i$. (See Notation~\ref{notation:successor++}.) We can write the $\RR$-fixed point set of $\tilde F^{(I)}$ as
	\eqnn
	(\tilde F^{(I)})^\RR
	=
	\coprod_{i \in I} K_i
	\eqnd
where $K_i$ consists of those $(\alpha,\beta)$ satisfying
	\eqnn
	\alpha(i,i^{++})=\infty
	\qquad\text{and}\qquad
	\beta_{i'} = 
		\begin{cases}
		\infty & i' \leq i \\
		-\infty & i' > i.
		\end{cases}
	\eqnd
Because the corners of the cube $[-\infty,\infty]^I$ form a discrete subset, and the condition $\alpha(i,i^{++}) = \infty$ defines a closed subset of $F^{(I)}$, the projection $\coprod_{i \in I} K_i \to F^{(I)}$ is unramified. 

\ref{property:quotient-closed} This is straightforward to verify by noting that the function $\rho$ above is defined continuously in the $\alpha$ variable.

\ref{property:separated} By construction, two points $\beta, \beta'$ above the same $\alpha$ satisfy the relation $\beta \leq \beta'$ if and only $\beta_i \leq \beta_i ' \in [-\infty,\infty]$ for all $i \in I$. This is a closed condition on $[-\infty,\infty]^I$, hence on $\tilde F^{(I)} \times_{F^{(I)}} \tilde F^{(I)}$. 

\ref{property:local lifting}
Fix $(\alpha,\beta) \in (\tilde F^{(I)})^\circ$. Since this is not a fixed point of the $\RR$-action, there is some $i_0 \in I$ for which $\beta_{i_0} \neq \pm \infty$. Define a map
	\eqnn
	b: F^{(I)} \to [-\infty,\infty]^I
	\eqnd
by 
	\eqnn
	b(\alpha)_i
	=
	\begin{cases}
	\beta_{i_0} + \alpha(i,i_0) & i \leq i_0 \\
	\beta_{i_0} - \alpha(i_0,i) & i_0 < i .
	\end{cases}
	\eqnd
This is a continuous function because $\beta_{i_0} \neq \pm \infty$. Then the map
	\eqnn
	\sigma: F^{(I)} \to (\tilde F^{(I)})^{\circ},
	\qquad
	\alpha \mapsto (\alpha, b(\alpha))
	\eqnd
is a continuous section. Note that the image of $\sigma$ is contained in the non-fixed point set because $b(\alpha)_{i_0} = 0$, and in particular, $b(\alpha)$ is not a corner of the cube $[-\infty,\infty]^I$. 

(Indeed, we have shown that $(\tilde F^{(I)})^\circ \to F^{(I)}$ {\em globally} has the lifting condition.)
\end{proof}

\begin{lemma}\label{lemma:FI-admits-I-section}
Moreover, consider the function
	\eqnn
	\sigma: F^{(I)} \times I \to \tilde {F^{(I)}},
	\eqnd
sending $(\alpha,i)$ to the element $(\alpha,\beta)$ where
	\eqnn
	\beta_j =
	\begin{cases}
	\alpha(i,j) & i \geq j \\
	-\alpha(j,i) & j \geq i.
	\end{cases}
	\eqnd
Then $\sigma$ is an $I$-section of $L_{F^{(I)}} \to F^{(I)}$.
\end{lemma}

\begin{proof}
By construction $\sigma$ is $\ZZ$-equivariant, so we must simply show that for every $\alpha \in F^{(I)}$, the induced map
	\eqnn
	I \to (\tilde F^{(I)})_{\alpha}^\circ
	\eqnd
is a surjection on $\pi_0$; i.e., a surjection after passing to the $\RR$-orbit set. But by Remark~\ref{remark:FI-fiber-surjection}, any element $(\alpha,\beta) \in (\tilde F^{(I)})^{\circ}_\alpha$ determines some $i_0 \in I$ for which $\beta_{i_0} \neq \pm \infty$. Moreover, $(\alpha,\beta)$ is in the same $\RR$-orbit as the unique $(\alpha,\beta')$ for which $\beta'_{i_0} = 0$. Thus the value of the section $\sigma(\alpha,-): I \to (\tilde F^{(I)})_\alpha$ at $i_0$ maps onto $[(\alpha,\beta)]$.
\end{proof}

\subsection{Functoriality of $\tilde F^{(I)}$ with respect to $I$}

Fix $r: I \to J$ a morphism in $\Preordpara$. We have an induced diagram of continuous maps
	\eqn\label{eqn:I-J-functoriality}
	\xymatrix{
	\tilde F^{(J)} \ar[r] \ar[d] & \tilde F^{(I)} \ar[d] \\
	F^{(J)} \ar[r] & F^{(I)}
	}
	\eqnd
where the top map sends
	\eqnn
	(\alpha,\beta) \mapsto (\alpha \circ (r \times r), \beta \circ r).
	\eqnd
For example,
	$
	(\alpha \circ (r \times r))(i_0, i_1) 
	=
	\alpha( ri_0,ri_1)
	$
for $i_0 \leq i_1 \in I$. Because $r$ is essentially surjective, the top horizontal map in \eqref{eqn:I-J-functoriality} is an injection in each fiber; from this it follows that we may apply Proposition~\ref{prop:R-equivariant-is-homeo}. Hence \eqref{eqn:I-J-functoriality} is a pullback diagram. That is, the diagram represents a morphism in $\Pt(\brokenpara)$. (See Definition~\ref{defn. maps of families}.)

\begin{remark}
The map $F^{(J)} \to F^{(I)}$ induced by $r$ is an injection if $r$ is an honest surjection. Moreover, the image is cut out by the equations
	\eqnn
	\{
	ri_0=ri_1 \implies 
	\alpha(i_0,i_1) = 0
	\}
	\eqnd
hence $F^{(J)} \to F^{(I)}$ is a closed embedding in this case.
\end{remark}

\begin{remark}
Let $\tau: \tilde F^{(I)} \to F^{(I)}$ be the $I$-section and $\sigma: \tilde F^{(J)} \to F^{(J)}$ be the $J$-section constructed in Lemma~\ref{lemma:FI-admits-I-section}.  Then the diagram~\eqref{eqn:I-J-functoriality} respects these data in the following sense: The diagram
    \eqnn
        \xymatrix{
        \tilde F^{(J)} \ar[r] 
        	& \tilde F^{(I)} \\
        F^{(J)} \times J \ar[u]^{\sigma} \\
        F^{(J)} \times I \ar[u]^{\id_X \times r} \ar[r]
        	& F^{(I)} \times I \ar[uu]^{\tau}
        }
    \eqnd
commutes.

\end{remark}

\subsection{Proof of Theorem~\ref{theorem. I-sections representable}}
Fix $I$ a paracyclic preorder.
Let
	\eqnn
	\Pt(F^{(I)})  \to \Top
	\eqnd
denote the Cartesian fibration (with discrete fibers) classifying the functor represented by $F^{(I)}$. For example, the fiber above $S \in \Top$ is given by the set $\hom_{\Top}(S, F^{(I)})$---equivalently, one can think of this set as a discrete groupoid.

We now exhibit a functor $j: \Pt(\brokenpara^I) \to \Pt(F^{(I)})$ respecting the forgetful maps to $\Top$. Given an object $(L_S \to S, \sigma)$ of $\Pt(\brokenpara^I)$, we have a function
	\eqnn
	S \times \arr(I)
	\to
	[-\infty,\infty]
	\eqnd
given by
	\eqnn
	(s,i,j)
	\mapsto
	d_{L_s}(\sigma(s,i) , \sigma(s,j)).
	\eqnd
By Lemma~\ref{lemma. distance continuous}, this is continuous; it is straightforward to verify we obtain a continuous map $d_\sigma: S \to F^{(I)}$. The assignment
	\eqnn
	j: (L_S \to S, \sigma)
	\mapsto
	(S \xra{d_\sigma} F^{(I)}),
	\eqnd
with the obvious effect on morphisms, produces the functor $j: \Pt(\brokenpara^I) \to \Pt(F^{(I)})$ we seek.

The inverse functor $h: \Pt(F^{(I)}) \to \Pt(\brokenpara^I)$ sends any function $f: S \to F^{(I)}$ to the pair
	\eqnn
	(f^* \tilde F^{(I)}, \sigma \circ f \times \id_I)
	\eqnd
where $\sigma$ is the $I$-section from Lemma~\ref{lemma:FI-admits-I-section}. It is now straightforward to equip the composites $h \circ j$ and $j \circ h$ with natural isomorphisms to/from the identity functors, in a way respecting the projection to $\Top$. that is to say, $j$ and $h$ are equivalences of stacks. 

\eqref{thmitem:representableOverBroken}
Let us next exhibit the commutative triangle of the theorem. Define a map 
	\eqnn
	\tilde f: L_S \to \tilde F^{(I)},
	\qquad
	x \mapsto (f(x),\beta(x)) 
	\eqnd
where $\beta(x) \in [-\infty,\infty]^I$ is defined by
	\eqnn
	\beta(x)_i 
	=
	\left( d_i(\pi(x)), x) \right)_i.
	\eqnd
It is straightforward to verify that $\tilde f$ is $\ZZ \times \RR$-equivariant and indeed lands in $\tilde F^{(I)}$. We also leave to the reader the straightforward verification that $\tilde f$ is a bijection along the fibers. By Proposition~\ref{prop:R-equivariant-is-homeo}, this shows that the universal arrow from $L_S$ to the pullback of $F^{(I)}$ along $f$ is an isomorphism of broken lines. This universal arrow exhibits the natural isomorphism making the triangle commute.

\eqref{thmitem:representabilityFunctorialInI} It is straightforward to check that the above constructions are functorial in the $I$ variable.

This completes the proof.

\clearpage

\section{Equivalence of the two definitions of broken paracycles}\label{section:broken-definitions-equivalent}

\newenvironment{brokeneasyprops}{%
	  \renewcommand*{\theenumi}{(P\arabic{enumi})}%
	  \renewcommand*{\labelenumi}{(P\arabic{enumi})}%
	  \enumerate
	}{%
	  \endenumerate
}

We do not make use of the following theorem in the main results of this work, but we prove the following for the reader's convenience.

\begin{theorem}\label{theorem:broken-definitions-equivalent}
Fix a topological space $S$ and a pair
	\eqnn
	(\pi: L_S \to S, \mu)
	\eqnd
where $\mu$ is a fiber-wise $\ZZ \times \RR$-action on $L_S$. Then $(\pi:L_S \to S, \mu)$ is a family of broken paracycles (Definition~\ref{defn:family-of-broken-paracycles}) if and only if it satisfies the following:
\begin{brokeneasyprops}
	\item\label{easyprop:fibers-paracycles} (Fibers are broken paracycles.) For every $s \in S$, the fiber $L_s = \pi^{-1}(s)$ is a broken paracycle in the sense of Definition~\ref{defn:broken-paracycle}.	
	\item\label{easyprop:unramified} (Unramified modulo $\ZZ$.) Let $L_S^{\RR}$ denote the fixed point set of the $\RR = \{0\} \times \RR$ action. Then the map
		\eqnn
		L_S^\RR \to S
		\eqnd
	induced by $\pi$ is unramified modulo $\ZZ$. 
	\item\label{easyprop:local-triviality} (Local triviality.) For every $s \in S$, there is an open subset $U \subset S$ containing $s$ such that there exists a directed homeomorphism 
		\eqnn
		U \times (-\infty,\infty) \cong \pi^{-1}(U)
		\eqnd 
		which respects the projection to $U$ and is $\ZZ$-equivariant, where $(-\infty,\infty)$ is equipped with the standard translational $\ZZ$ action.
\end{brokeneasyprops}
\end{theorem}

\begin{remark}
Note that \ref{easyprop:local-triviality} makes no mention of how the local trivialization behaves with respect to the $\RR$-fixed point locus; the only condition of the trivialization is that it be a {\em directed} homeomorphism along each fiber (see \ref{property:directed-action-ordering} of Definition~\ref{defn:broken-paracycle}). 
\end{remark}

\begin{prop}[Local triviality]\label{prop:local-triviality}
Every family of broken paracycles is locally trivial as a directed $(-\infty,\infty)$ fiber bundle. That is, every family of broken paracycles satisfies~\ref{easyprop:local-triviality}.
\end{prop}

\begin{proof}
Let $\pi: L_S \to S$ be a family of broken paracycles and fix $s \in S$. Let $U$ be a neighborhood of $s$ such that we have a local section $\sigma: U \to L_U^{\circ}$. Let
	\eqnn
        K_\sigma = \{\tilde x \in L_U \text{ such that $\sigma(\pi(\tilde x)) \leq \tilde x \leq \sigma(\pi(\tilde x))+1$.} \}
	\eqnd
	By Property~\ref{property:closed-prime} of Lemma~\ref{lemma:closed-map-lemma}, the projection $\pi$ restricted to $K_\sigma$ is a closed map to $U$.
	
	Now, for an appropriate choice of parasimplex $I$ and shrinking $U$ if necessarily, we have an $I$-section $\sigma'$ extending $\sigma$ (so $\sigma'(i) = \sigma$ for some $i \in I$). Enumerating the elements
		\eqnn
		i= i_0 < i_1 < \ldots < i_n < i_0 +1
		\eqnd
in $I$, we have the function 
	\eqnn
	\gamma: x \mapsto \sum_{a=0, \ldots, n} \rho( d(\sigma_{i_a}(\pi(x)),x)).
	\eqnd
One can prove that the product 
	\eqnn
	\pi \times \gamma: K_\sigma \to U \times [0,1]
	\eqnd
is a homeomorphism which is directed along the fibers as in the proof of Theorem~2.4.1 of~\cite{broken}. One extends this $\ZZ$-equivariantly to exhibit a homeomorphism
	\eqnn
	L_U \to U \times (-\infty,\infty)
	\eqnd
and the result follows.
\end{proof}

\begin{proof}[Proof of Theorem~\ref{theorem:broken-definitions-equivalent}.]
Proposition~\ref{prop:local-triviality} proves one implication. The reverse implication is straightforward and can be found also in~\cite{broken}.
\end{proof}

\clearpage

\section{A review of stratifications and constructible sheaves}\label{section:constructible-sheaves}

It was a vision of Macpherson that any constructible sheaf is equivalent to a representation of the exit path category. This was proven by Lurie in~\cite{higher-algebra} for a large class of stratified spaces, generalizing a previous result of Treumann~\cite{treumann-exit-paths}; the main goal of our recollection is to state this theorem (Theorem~\ref{theorem.exit-path-theorem}) and explain that this equivalence is natural in the stratified space variable (Remark~\ref{remark:exit-paths-natural-in-X-P}). We review these facts here for the reader's convenience.
This material is drawn from Section~A.5 of~\cite{higher-algebra}. 

\subsection{Posets}

\begin{construction}[Alexandroff topology]
Let $P$ be a poset. We render $P$ as a topological space by endowing it with the {\em Alexandroff topology}, in which a subset $U \subset P$ is open if and only if it is {\em upward closed}, meaning
	\eqnn
	(x \in U) \, \& \,
	(y \geq x) \implies y \in U.
	\eqnd
Let $Q$ be another poset.
We note that a map $P \to Q$ is continuous if and only if it is a map of posets (i.e., weakly order-preserving). In this way we have a fully faithful embedding
	\eqnn
	\poset
	\to
	\Top
	\eqnd
from the category of posets to the category of topological spaces.
\end{construction}

\begin{notation}
Given a poset $P$, we let $P^{\op}$ denote its opposite poset. (The same set with the opposite order.)
\end{notation}

\begin{notation}[$P_{>p}$]\label{notation:P-bigger-than-p}
For any $p \in P$, we let
	\eqnn
	P_{>p}
	=
	\{ \text{$q \in p$ such that $q > p$} \}.
	\eqnd
\end{notation}

\begin{notation}[Nerve $N(P)$]\label{notation:nerve-of-P}
Any poset $P$ may be also considered a category---the set of objects is $P$, while there is a unique morphism from $x$ to $y$ if and only if $x \leq y$.

Accordingly, we have the simplicial set $N(P)$ given by the nerve of $P$; a $k$-simplex of $N(P)$ is exactly the data of a string of weakly increasing elements in $P$:
	\eqnn
	x_0 \leq x_1 \leq \ldots \leq x_k
	\eqnd
At the same time, we have the singular complex $\sing(P)$ of $P$; a $k$-simplex of $\sing(P)$ is precisely the data of a continuous map $|\Delta^k| \to P$ from the standard $k$-simplex.
\end{notation}

\subsection{Stratifications}

\begin{defn}[Stratification]
Let $X$, be a topological space and $P$ a poset. A {\em stratification} on $X$ (by $P$) is the data of a continuous functions $X \to P$. 

In this work, we will call the data $X \to P$ a {\em stratified space}, or {\em $P$-stratified space} when the choice of $P$ is to be emphasized.
\end{defn}

\begin{notation}\label{notation:fiber-of-stratification}
Let $X \to P$ be a stratified space. For every $p \in P$, we let $X_p$ denote the subspace of $X$ given by the preimage of $p$.
\end{notation}

\begin{defn}[Conical stratification]\label{defn. conical stratification}
Let $X \to P$ be a stratified space and fix $x \in X$ with image $p_x \in P$. We say $X$ is {\em conically stratified at $x$} if there exists an open subset $U \subset X$ containing $x$ which is homeomorphic to 
	\eqnn
	Z \times C(Y)
	\eqnd
where 
	\enum
	\item $Z$ is a topological space, 
	\item $Y$ is a $P_{>p_x}$-stratified space (see Notation~\ref{notation:P-bigger-than-p}),
	\item $C(Y) = Y \times [0,\infty) / Y \times \{0\}$ is the open cone on $Y$, stratified by the poset $P_{>p_x} \cup p_x = P_{\geq p_x}$ by sending $[(y,0)] \mapsto p_x$, and
	\item $Z \times C(Y)$ is the stratified space whose stratification is the composition $Z \times C(Y) \to C(Y) \to P_{\geq p_x}$. 
	\enumd
We say $X$ is {\em conically stratified} if it is conically stratified at $x$ for every $x \in X$.
\end{defn}

\begin{example}
If $X$ is conically stratified, so is any open subset of $X$.
\end{example}

\begin{example}\label{example:octant-and-simplex-stratification}
For any integer $n \geq 0$, let $X$ be the Euclidean octant $\RR_{\geq 0}^{n+1}$. Let $\cP(n)$ be the power set of the set $\{0,\ldots,n\}$ considered a poset under $\subset$. Then we have a stratification
	\eqnn
	\cS: X \to \cP(n)
	\qquad
	x \mapsto
	\{ \text{$i$ such that $x_i \neq 0$}.
	\}
	\eqnd
So for example, the interior of the octant is collapsed to the set $\{0,1,\ldots,n\}$ and the corner of the octant is sent to $\emptyset$.

Moreover, $X$ is conically stratified. We only check this at the origin of $X$ and leave the inductive details of the other points to the reader.

Set $Z = \ast$ to be a singleton and let $Y = |\Delta^{n}| \subset X \subset \RR^{n+1}$ be the standard $n$-simplex;, embedded as the collection of those $\vec x$ whose coordinates are non-negative and sum to 1. $Y$ is stratified by the restriction of $\cS$ to $Y$, and $X$ is homeomorphic to $C(Y)$ as a stratified space over 
	$\cP(n)
	= 
	\{\emptyset\} \cup (\cP(n))_{> \emptyset}$.
\end{example}

\begin{defn}\label{defn:category-of-stratified-spaces}
We can define the (ordinary) category of conically stratified spaces as follows: An object is a conically stratified space $X \to P$. Given another conically stratified space $Y \to Q$, a morphism is the data of a commutative diagram
	\eqnn
	\xymatrix{
	X \ar[r]^f \ar[d] & Y \ar[d] \\
	P \ar[r]^r &  Q
	}
	\eqnd
where $f$ is a continuous map and $r$ is a map of posets.
\end{defn}

\subsection{Standard stratification on the $n$-simplex}

The standard stratification of $|\Delta^n|$ is different from the one we saw in Example~\ref{example:octant-and-simplex-stratification}.

\begin{construction}\label{construction:standard-simplex}
Let $[n] = \{0 < 1 < \ldots < n\}$ be the linear poset with $n+1$ elements, and consider the map
	\eqnn
	\max: 
	\cP(n) \setminus \{\emptyset\}
	\to
	[n],
	\qquad
	S \mapsto \max S.
	\eqnd
Then we stratify the standard $n$-simplex by the composite
	\eqnn
	|\Delta^n|
	\to
	(\cP(n) \setminus \{\emptyset\})
	\xrightarrow{\max}
	[n].
	\eqnd
That is, this stratification sends $x \in |\Delta^n| \subset \RR^{n+1}$ to the maximal $i$ for which $x_i \neq 0$.
\end{construction}

\begin{example}
For example, the 1-simplex is stratified so that a single vertex is sent to the element $0 \in [1]$, while the rest of the 1-simplex is sent to $1 \in [1]$.

The 2-simplex is stratified so that the initial vertex is sent to $0 \in [2]$, the open-closed edge from the 0th vertex to the 1st vertex is sent to $1 \in [2]$, while the entire complement of the edge $\vec{01}$ is collapsed to $2 \in [2]$.
\end{example}

\begin{defn}\label{defn:standard-simplex}
We call the stratification $|\Delta^n| \to [n]$ from Construction~\ref{construction:standard-simplex} the {\em standard stratification} on the $n$-simplex.
\end{defn}

We have the following:

\begin{prop}\label{prop:nerve-of-P-is-inside-sing(P)}
Let $P$ be a poset. Then the nerve of $P$ (Notation~\ref{notation:nerve-of-P}) may be identified with the simplicial subset of $\sing(P)$ whose $k$ simplices consist of those continuous maps $|\Delta^k| \to P$ factoring through the standard stratification of $|\Delta^k|$ (Definition~\ref{defn:standard-simplex}). that is,
	\eqnn
	N(P)_k
	\cong
	\{ j: |\Delta^k| \to [k] \to P \}
	\eqnd
where $[k] \to P$ is allowed to be an arbitrary map of posets.

This identification is natural in $P$.
\end{prop}

\subsection{Exit path categories}


\begin{defn}[Exit path $\infty$-category]\label{defn:exit-path-category}
Let $X \to P$ be a conically stratified space (Definition~\ref{defn. conical stratification}). Then the {\em exit path $\infty$-category} of $X \to P$ is the pullback of simplicial sets
	\eqnn
	\xymatrix{
	\Exit(X) \ar[rr] \ar[d] && \sing(X) \ar[d] \\
	N(P) \ar[rr]^-{\text{Prop~\ref{prop:nerve-of-P-is-inside-sing(P)}} } && \sing(P).
	}	
	\eqnd
Here, the bottom horizontal arrow is the one supplied by Proposition~\ref{prop:nerve-of-P-is-inside-sing(P)}.

Note that the stratification is implicit in the notation $\Exit(X)$, which has no mention of $P$.
\end{defn}

\begin{remark}
That $\Exit(X)$ is an $\infty$-category is a non-trivial result, and relies on the {\em conically} stratified condition; see Theorem~A.6.4 of~\cite{higher-algebra}. 
Note that we write as $\Exit(X)$ would be written $\sing^P(X)$ by Lurie in~\cite{higher-algebra}.
\end{remark}

\begin{remark}
Let us at least understand the objects and morphisms of $\Exit(X)$. The 0-simplices (hence objects) are points of $X$. A 1-simplex (hence a morphism) is a path 
	\eqnn
	\gamma: [0,1] \to X
	\eqnd
subject to the following constraint: There exist two elements $p_0, p_1 \in P$ with $p_0 \leq p_1$ such that
	\eqnn
	\gamma(0) \in X_{p_0}
	\qquad \text{and}
	\qquad
	t \neq 0 \implies \gamma(t) \in X_{p_1}.
	\eqnd
If $p_0=p_1$, this means $\gamma$ ``stays'' in the stratum $X_{p_0}$ for all time; otherwise, this condition means that $\gamma$ ``immediately'' exits to the stratum $X_{p_1}$. This explains the term ``exit path.''
\end{remark}

\begin{remark}[Naturality]\label{remark:Exit-is-natural}
Let $X \to P$ and $Y \to Q$ be stratified spaces and suppose we have a map $f: X \to Y$ of stratified spaces (Definition~\ref{defn:category-of-stratified-spaces}). Because the map $N(P) \to \sing(P)$ from Proposition~\ref{prop:nerve-of-P-is-inside-sing(P)} is natural in the $P$ variable and because pullbacks are natural, we have an induced map of simplicial sets
	\eqnn
	f: \Exit(X) \to \Exit(Y).
	\eqnd
This exhibits a functor from the category of conically stratified spaces (Definition~\ref{defn:category-of-stratified-spaces}) to the $\infty$-category of $\infty$-categories. Indeed, because we have taken the simplicial set model for $\infty$-categories, the functor $\Exit$ can be made to factors through the strict category of simplicial sets.
\end{remark}

\begin{remark}[Entry paths]
There is another way to view the exit path $\infty$-categories: as dual to the entry path $\infty$-category. This may seem tautological, but the latter has the advantage of being defined through a natural fibration. 

Namely, call a stratified space a {\em basic} if it is isomorphic as a stratified space to a product
	\eqnn
	\RR^i \times C(Y)
	\eqnd
for some compact stratified space $Y$. Given any stratified space $X$, we have a functor of $\infty$-categories from the $\infty$-category of basics to the $\infty$-category $\Kan$, by sending a basic $U$ to the space of stratified open embeddings $U \to X$. The associated Grothendieck construction is a right fibration over the $\infty$-category $\Bsc$ of basics, and we call this $\Bsc_{/X}$ the {\em entry path $\infty$-category of $X$}. We have omitted various details, for which the reader may refer to~\cite{aft}. For a large class of stratified spaces, one may prove that $\Exit(X)^{\op}$ is equivalent to the entry path category of $X$ as $\infty$-categories.

We do not utilize entry path $\infty$-categories because their functoriality is less natural than that of exit path categories. For example, as we saw in Remark~\ref{remark:Exit-is-natural}, any map of stratified spaces induces a functor on exit path categories. For the above model of entry path categories, one must first rely on a contractibility argument to show that for certain maps $f: X \to Y$ of stratifies spaces, the pullback
	\eqnn
	\xymatrix{
	\cE \ar[r] \ar[dd] 
		& \Bsc_{/Y} \ar[d] \\
		& \Strat_{/Y} \ar[d]^{f^*} \\
	\Bsc_{/X} \ar[r]	& \Strat_{/X} 
	}
	\eqnd
is a trivial fibration $\cE \to \Bsc_{/X}$ (hence an equivalence).
 
(Note that above, $\Strat_{/Y}$ is the $\infty$-category of stratified spaces equipped with open embeddings into $Y$. The functor $f^*$ takes an object $V \to Y$ to the pullback $V \times_Y X$.)

Only then do we have a functor $\Bsc_{/X} \simeq \cE \to \Bsc_{/Y}$ on entry path categories. 
\end{remark}

\subsection{Constructible sheaves}
Now let $\cC$ be a compactly generated $\infty$-category. Below, by a sheaf on $X$, we mean a sheaf valued in $\cC$.

\begin{notation}
Fix a topological space $X$. We let $\Open(X)$ denote the poset of open subsets of $X$, ordered by inclusion. Given a compactly supported $\infty$-category $\cC$, we will write
	\eqnn
	\Psh(X;\cC)
	\qquad
	\text{or}
	\qquad
	\Psh(X)
	\eqnd
to mean the $\infty$-category of functors $\Open(X)^{\op} \to \cC$. We write $\Psh(X)$ when $\cC$ is implicit.

A presheaf $\cF$ is called a sheaf if for any open cover $\cU = \{U_\alpha\}_{\alpha}$ of an open set $U$, the induced map
	\eqnn
	\cF(U)
	\to
	\limarrow_{V \subset U_\alpha \in \cU} \cF(V)
	\eqnd
is an equivalence. Here, the limit is indexed by the poset of all opens $V \subset U$ that are contained in some $U_\alpha \in \cU$ of the open cover. We let
	\eqnn
	\Shv(X;\cC)
	\qquad
	\text{or}
	\qquad
	\Shv(X)
	\eqnd
denote the full subcategory of $\Psh(X)$ consisting of sheaves.
\end{notation}

\begin{remark}
For any $\infty$-category $\cC$ and for any continuous map $f: X \to Y$, we have a pushforward functor $f_*: \Shv(X) \to \Shv(Y)$; this sends any sheaf $\cF$ on $X$ to the sheaf $f_*\cF$, where
	\eqnn
	f_* \cF(V) = \cF(f^{-1}(V)),
	\qquad
	V \in \Open(Y).
	\eqnd
When $\cC$ is compactly generated, $f_*$ admits a left adjoint, $f^*$. This is the pullback of sheaves.
\end{remark}

\begin{defn}
Let $X$ be a topological space, and let $p: X \to \ast$ be the constant map to a point. A sheaf on $X$ is called {\em constant} if it is in the essential image of the pullback functor
	\eqnn
	p^* :
	\shv(\ast) \to \shv(X).
	\eqnd
Equivalently, we say a sheaf $\cF$ is constant if there exists an object $A \in \cC$ and a map from $A$ to the global sections $\cF(X)$ such that, for any $x \in X$, the composite map
	\eqnn
	A \to \cF(X) \to \cF_x
	\eqnd
to the stalk is an equivalence.
\end{defn}

\begin{defn}
A sheaf $\cF$ on $X$ is called {\em locally constant} if there exists an open cover $\coprod_\alpha U_\alpha \to X$ such that the restriction of $\cF$ to each $U_\alpha$ is a constant sheaf.
\end{defn}

\begin{remark}
Being constant and locally constant are vastly different; even classically, one can witness monodromies in locally constant sheaves.
\end{remark}

\begin{defn}
Let $X \to P$ be a stratified space and let $\cF$ be a sheaf on $X$. We say that $\cF$ is {\em constructible} (with respect to the stratification on $X$) if for every $p \in P$, the restriction $\cF|_{X_p}$  is a locally constant sheaf on $X_p$.
\end{defn}

\subsection{Exit path $\infty$-categories and constructible sheaves}

\begin{theorem}[Theorem~A.9.3 of~\cite{higher-algebra}]\label{theorem.exit-path-theorem}
Fix a reasonable conically stratified space $X \to P$ and let $\cC = \Kan$ be the $\infty$-category of Kan complexes. Then there exists an equivalence of $\infty$-categories
	\eqnn
	\Ffun(\Exit(X),\Kan)
	\simeq
	\shv^{\cbl}(X;\Kan)
	\eqnd
between the $\infty$-category of constructible $\Kan$-valued sheaves on $X$, and the $\infty$-category of functors from $\Exit(X)$ to $\Kan$.
\end{theorem}

\begin{remark}
In~\cite{higher-algebra}, ``reasonable'' means that $X$ is paracompact and locally of singular shape, while $P$ satisfies the ascending chain condition. We will not need the full power of these hypotheses. In our setting, all $P$ will be finite, while all $X$ are modeled as open subsets of corners of finite-dimensional Euclidean octants. Indeed, we may restrict attention to the class of stratified spaces considered, for example, in~\cite{aft}. 
\end{remark}

\begin{remark}
Informally, the map $\Fun(\Exit(X),\Kan) \to \Shv^{\cbl}(X;\Kan)$ is induced by sending a functor $F$ to the presheaf
	\eqnn
	U \mapsto \limarrow_{\Exit(U)} F.
	\eqnd
\end{remark}

\begin{remark}\label{remark:exit-paths-natural-in-X-P}
While the theorem as proven in~\cite{higher-algebra} does not specifically address how to generalize the case $\cC = \Kan$ to an arbitrary compactly generated $\infty$-category, and does not specifically address the naturality in the $X$ variable, this can be done. Let us first explain what we mean by naturality.

We see from Remark~\ref{remark:Exit-is-natural} that the assignment
	\eqnn
	X \mapsto \Ffun(\Exit(X),\cC)
	\eqnd
is contravariantly functorial in the $X$ variable for any $\infty$-category $\cC$. On the sheaf side, we have a pullback functor
	\eqnn
	f^*: \Shv(X) \to \Shv(Y)
	\eqnd
so long as $\cC$ is compactly generated; one can easily check that $f^*$ restricts to a map on the full subcategories of constructible sheaves. The naturality claim is that one can construct a natural transformation between these functors; so for example, one can exhibit a homotopy-commuting diagram
	\eqnn
	\xymatrix{
	\Ffun(\Exit(X),\cC) \ar[r]^-{\sim}  \ar[d]
		&\Shv^{\cbl}(X;\cC)	 \ar[d] \\
	\Ffun(\Exit(Y),\cC) \ar[r]^-{\sim} 
		&\Shv^{\cbl}(Y;\cC).	 \\
	}
	\eqnd
Moreover---as indicated---the horizontal arrows can be made equivalences for any compactly generated $\infty$-category $\cC$, generalizing Theorem~\ref{theorem.exit-path-theorem} above. 

One way to do all this is by recycling the argument of Lurie, with the key ingredients being recollement and the codescent property of exit path $\infty$-categories. Let us very quickly recall the argument.

Lurie produces a proof by induction on the depth~\cite{aft} of a stratification, which Lurie calls rank~\cite{higher-algebra}. The base case is a proof about locally constant sheaves on a reasonable space; Theorem~\ref{theorem.exit-path-theorem} can be proven in this case by showing codescent for singular complexes. In other words, the functor $X \mapsto \sing(X)$ sends the augmented \v{C}ech complex of any open cover to a homotopy colimit diagram. Informally, this is stating that homotopy types ``glue.'' (This does not involve the specific choice of $\cC$.) Choosing a $\cC$ compactly generated, we can also compute the induced limits of sheaf categories, and this concludes the base case because categories of sheaves form a sheaf. This step is clearly natural in the $X$ variable.

The inductive step involves a recollement argument, showing that the exit path category on $\cone(Y)$ is the cone category on $\Exit(Y)$. This can be proven naturally in the $Y$ variable, and this equivalence is compatible with the identification of $\Shv^{\cbl}$.
\end{remark}

\clearpage
\section{Sheaves on $\brokenpara$.}

In Theorem~\ref{theorem. brokenpara as a colimit} we expressed $\broken$ as a colimit of stacks $\broken^I$.

As an immediate consequence, we have:

\begin{cor}\label{corollary:sheaves-on-broken-limit}
For any compactly generated $\infty$-category $\cC$, we have an equivalence
	\eqnn
	\shv(\brokenpara,\cC)
	\to
	\limarrow_{I \in \Preordpara} \shv(\brokenpara^I, \cC).
	\eqnd
\end{cor}

Moreover, we saw in Theorem~\ref{theorem. I-sections representable} that each $\broken^I$ is representable by a topological space $F^{(I)}$ while naturally respecting the maps to $\broken$. So the induced arrow 
	\eqn\label{eqn:broken sheaves as limit of F^I sheaves}
	\shv(\brokenpara,\cC)
	\to
	\limarrow_{I \in \Preordpara} \shv(F^I, \cC).
	\eqnd
is an equivalence. In fact, this limit may be computed considering {\em constructible} sheaves on the righthand side, as we will soon explain. Moreover, we note that there is a natural stratification on $\brokenpara$. Let $\ZZ_{\geq 0}^{\op}$ denote the opposite of the total order $\ZZ_{\geq 0}$. Then the map
	\eqnn
	\brokenpara \to \ZZ_{\geq 0}^{\op}
	\eqnd
sends any family of broken paracycles to the stratification on the base by the isomorphism classes of fibers:
	\eqnn
	(\pi: L_S \to S)
	\mapsto
	(s \mapsto n_{\pi^{-1}(L_s)}).
	\eqnd
(See Notation~\ref{notation:n_L for broken line L}.)

\begin{theorem}\label{theorem. sheaves are Deltasurj}
For any compactly generated $\infty$-category $\cC$, we have an equivalence
	\eqnn
	\shv(\brokenpara,\cC)
	\to 
	\Ffun(\Deltaparasurj, \cC).
	\eqnd
\end{theorem}

\subsection{The stratification on $F^I$}

Recall the maps $F^{(I)} \to \brokenpara$, which classify families of broken paracycles $\tilde F^{(I)} \to F^{(I)}$. One can stratify $F^{(I)}$ by the isomorphism classes of the fibers of $\tilde F^{(I)}$. We employ some notation to make this precise.

\begin{defn}
Fix a subset $E \subset I \times I$. We write $iEj$ whenever $(i,j) \in E$.
\end{defn}

One can straightforwardly prove the following:

\begin{prop}\label{prop. convex relations}
Fix a paracyclic preorder $I$ with preorder relation $\leq$ (which by definition is a subset of $I \times I$). Fix an equivalence relation $E \subset I \times I$ such that $\leq \subset E$. The following are equivalent:
\enum
	\item $E$ is $\ZZ$-equivariant, convex, and non-trivial. This means (i) $iEj \implies (i\pm1)E(j\pm1)$, (ii) if $i\leq k$ and $iEk$, then for all $j$ such that $i \leq j \leq k$, we have $iEj$ and $jEk$, and (iii) $E \neq I \times I$.
	\item The relation induced on the quotient $I/E$ is a partial order, and the map $I \to I/E$ is an essentially surjective map of paracyclic preorders. 
	\item There exists an essentially surjective, $\ZZ$-equivariant map $p:I \to J$ of paracyclic preorders for which $J$ is a parasimplex and $iEi'$ if and only if $p(i)=p(i')$.
\enumd
\end{prop}

\begin{notation}\label{notation. Conv(I)}
Fix a paracyclic preorder $I$. We let $\Conv(I)$ denote the collection of all equivalence relations $E \subset I \times I$ satisfying any of the equivalent conditions of Proposition~\ref{prop. convex relations}. We endow $\Conv(I)$ with the poset structure of inclusion---$E \leq E'$ if and only if $E \subset E'$ in $I \times I$.
\end{notation}

\begin{construction}\label{construction:stratification of F^{(I)}}
Fix a paracyclic preorder $I$ and choose an element $\alpha \in F^{(I)}$. One can associate to $\alpha$ an equivalence relation $E_\alpha \in \Conv(I)$ by declaring
	\eqnn
	iE_\alpha j \iff \alpha(i,j) < \infty.
	\eqnd
This yields a function
	\eqnn
	F^{(I)} \to \Conv(I),
	\qquad
	\alpha \mapsto E_\alpha
	\eqnd
\end{construction}

\begin{prop}\label{prop.F^{(I)}-is-stratified}
The function $F^{(I)} \to \Conv(I)$ from Construction~\ref{construction:stratification of F^{(I)}} renders $F^{(I)}$ a conically stratified space (in the sense of Definition~\ref{defn. conical stratification}).
\end{prop}

\begin{proof}
We first prove the map is continuous. Since $\Conv(I)$ is finite, one need only verify that for any $E \in \Conv(I)$, the preimage of $\{E' \text{ such that $E \subset E'$} \}$ is an open subset of $F^{(I)}$. The pre-image is identified with the set of all $\alpha$ such that $(i \leq j) \& (iEj) \implies \alpha(i,j) \neq \infty$. This exhibits the pre-image as an intersection of open subsets of $F^{(I)}$, and this intersection is finite by $\ZZ$-equivariance of $E$ (Proposition~\ref{prop. convex relations}). This proves the map is continuous.

To see that the stratification is conical, it suffices to consider the case when $I$ is a parasimplex. (Any paracyclic preorder $I'$ is obtained by expanding the order on some parasimplex $I$ to a preorder; since $F^{(I')}$ embeds openly into $F^{(I)}$, $F^{(I')}$ is conically stratified.) When $I$ is a parasimplex, one may identify $F^{(I)}$ as the open cone on the boundary of an $(n_I-1)$-simplex, where the boundary simplex is given the stratification induced by the stratification from Example~\ref{example:octant-and-simplex-stratification}. It is straightforward---say, by induction---to check that the boundary of simplices are conically stratified.
\end{proof}

\begin{notation}\label{notation:F^{(I)}_E-stratum}
Fix a paracyclic preorder $I$ and an equivalence relation $E \in \Conv(I)$.  (See Notation~\ref{notation. Conv(I)}.) We let
	\eqnn
	(F^{(I)})_E \subset F^{(I)}
	\eqnd
denote the collection of all $\alpha$ for which $\alpha(i,j)< \infty \iff iEj$. (See Notation~\ref{notation:fiber-of-stratification} and Construction~\ref{construction:stratification of F^{(I)}}.)
\end{notation}

\begin{remark}
Fix a paracyclic preorder $I \in \Preordpara$. We know from Remark~\ref{remark:F^{(I)} are corners} that the space $F^{(I)}$ may be identified with (an open subset of) a the faces of a Euclidean octant. 
\end{remark}

\begin{remark}
Fix a paracyclic preorder $I$, and let $\sim_I$ denote the equivalence relation induced by the preorder relation (Notation~\ref{notation:I-relation}). Then $I/\sim_I$ is isomorphic to a parasimplex; let $n_I$ be the non-negative integer associated to this parasimplex as in Notation~\ref{notation:n_I}, so that $(I/\sim_I)/\ZZ$ has $n_I+1$ elements in it.

Then $\Conv(I)$ is abstractly isomorphic, as a poset, to a cube poset with a terminal corner removed. More specifically, one has an isomorphism
	\eqnn
	\Conv(I) \cong (\cP(n_I)^{\op} \setminus \emptyset).
	\eqnd
Here, $\cP(n_I)$ is the power set of the set $\{0,\ldots,n_I\}$ (Example~\ref{example:octant-and-simplex-stratification}). 

Identifying (non-canonically) $F^{(I)}$ with the faces of the Euclidean octant $(-\infty,\infty]^{n_I+1}$, one can further identify $F^{(I)}$ (as a stratified space) as a closed subspace of the Euclidean octant $(-\infty,\infty]^{n_I+1}$---namely, the closed subspace consisting of those points $x$ where at least one coordinate is equal to $\infty$. Here we give the octant a stratification similar to the one we gave in Example~\ref{example:octant-and-simplex-stratification}---we send a point $x$ to the collection of $i$ for which $i = \infty$. 

Likewise, for any paracyclic preorder $I$, we may (non-canonically) identify $F^{(I)}$ with an open subset of some $F^{(J)}$ where $F^{(J)}$ is a parasimplex. Specifically, there is a unique parasimplex $J$ admitting an injective, $\ZZ$-equivariant map of preorders $r: J \to I$ any choice of such a map induces an open embedding of $r^*: F^{(I)} \to F^{(J)}$. The complement of this open embedding is easy to describe: The complement consists of those strata $(F^{(J)})_E$ for which the relation $\sim_I$ is not contained in $E$. 
\end{remark}

\subsection{Sheaves from $\brokenpara$ are constructible on  $F^{(I)}$}

\begin{lemma}\label{lemma:pullback-to-F^{(I)}-is-constructible}
Let $F^{(I)} \cong \broken^I \to \broken$ be the composite of the forgetful functor preceded by the representing equivalence $F^{(I)} \cong \broken^I$ from Theorem~\ref{theorem. I-sections representable}. Then any sheaf on $\broken$ pulls back to a sheaf on $F^{(I)}$ which is constructible with respect to the stratification $F^{(I)} \to \Conv(I)$ (Proposition~\ref{prop.F^{(I)}-is-stratified}). 
\end{lemma}

\begin{proof}
Fix $E \in \Conv(I)$ and let $(F^{(I)})_E$ be the corresponding stratum (Notation~\ref{notation:F^{(I)}_E-stratum}). Also let $J= I/E$ be the quotient parasimplex, and let $\Delta_J \in \Conv(J)$ be the diagonal (trivial) equivalence relation. Then any section $r: J \to I$ of the projection $I \to J$ induces a map 
	\eqnn
	r': (F^{(I)})_E \to (F^{(J)})_{\Delta_J} \cong \ast
	\eqnd
which also identifies the families of broken paracycles
	\eqnn
	\tilde F^{(I)} |_{(F^{(I)})_E}
	\cong
	(r')^* 
	\left(
		\tilde F^{(J)} |_{(F^{(J)})_{\Delta_J}}
	\right).
	\eqnd
This proves that the map $(F^{(I)})_E \to F^{(I)} \to \brokenpara$ factors through a point (namely, the corner stratum $(F^{(J)})_{\Delta_J}$) hence any sheaf pulled back from $\brokenpara$ is (locally) constant along each stratum of $F^{(I)}$. This completes the proof.
\end{proof}

\begin{corollary}\label{cor.stalks-of-F^{(I)}}
Fix a sheaf $\cF$ on $\brokenpara$. Fix an integer $n \geq 0$ and let $p_n: \ast \to \brokenpara$ be the map classifying a broken paracycle $L$ with $n_L = n$ (Notation~\ref{notation:n_L for broken line L}). 

Then for any paracyclic preorder $I$ and any $E \in \Conv(I)$ such that $n_{I/E} = n$, the stalk of $\cF$ along $(F^{(I)})_E$ is given by the object $p_n^* \cF$.
\end{corollary}

\subsection{Constructible sheaves on $F^{(I)}$}
Let $F^{(I)} \to \Conv(I)$ be the stratified space from Proposition~\ref{prop.F^{(I)}-is-stratified}. By definition of exit path category (Definition~\ref{defn:exit-path-category}), there is a functor of $\infty$-categories
	\eqnn
	\Exit(F^{(I)}) \to N(\Conv(I)).
	\eqnd
(Here, we have considered $\Conv(I)$ as a category and taken its nerve to render it an $\infty$-category.)

\begin{lemma}\label{lemma:exit-F^{(I)}-is-Conv(I)}
The map $\Exit(F^{(I)}) \to N(\Conv(I))$ is an equivalence of $\infty$-categories. Moreover, this equivalence is natural in the $I$ variable. 
\end{lemma}

\begin{proof}
The above map is natural by definition of exit path categories as a pullback (Definition~\ref{defn:exit-path-category}), and by the fact that the assignment $I \mapsto F^{(I)}$ is a functor to the category of conically stratified spaces. 

To show that the functor is an equivalence, it suffices to show that for all $x_E \in (F^{(I)})_E$ with $E \subset E'$, we have that
	\eqnn
	\hom_{\Exit(F^{(I)})}(x_E,x_{E'})
	\eqnd
is contractible. If $E = E'$, this follows from the observation that $(F^{(I)})_E$ may be identified with a convex (hence contractible) subspace of $\RR^n$. If $E \neq E'$, one observes that any map $\cone(\del \Delta^k) \to x_E \cup (F^{(I)})_{E'}$ may be filled to a map $\cone(\Delta^k) \to x_E \cup (F^{(I)})_{E'})$.
\end{proof}

Finally, let us introduce a category to organize the various $\Conv(I)$ for varying $I$.

\begin{notation}[$\widetilde{\Conv}$]\label{notation:Conv-tilde}
The assignment $I \mapsto \Conv(I)$ defines a functor from $\Preordpara^{\op}$ to the category of posets; given any morphism $r: I \to J$, any relation $E \in \Conv(J)$ pulls back to a relation in $\Conv(I)$. Considering any poset as a category, we can apply the Grothendieck construction to obtain a Cartesian fibration
	\eqnn
	\widetilde{\Conv} \to \Preordpara.
	\eqnd
Concretely, an object of $\widetilde{\Conv}$ is a pair  $(I,E)$  where $I$ is a paracyclic preorder and $E$ is a convex equivalence relation satisfying any of the equivalent properties in Proposition~\ref{prop. convex relations}.  A morphism from $(I,E_I)$ to $(J,E_J)$ is the data of a map $r: I \to J$ in $\Preordpara$---that is, a $\ZZ$-equivariant and essentially surjective map---such that $r$ induces a factorization $I/E_I \to J/E_J$.
\end{notation}

\begin{remark}\label{remark:Conv-Deltasurj-adjunction}
Given any object $(I,E)$ of $\widetilde{\Conv}$, we have that $I/E$ is a parasimplex (Proposition~\ref{prop. convex relations}). The assignment $(I,E) \mapsto I/E$ is a functor to the category $\Deltaparasurj$ of parasimplices with surjective, $\ZZ$-equivariant maps (Notation~\ref{notation:Deltasurj}). 

Moreover, this assignment admits a right adjoint, given by sending a parasimplex $J$ to the pair $(J,\Delta_J)$ where $\Delta_J$ is the diagonal equivalence relation (otherwise known as ``equality''). This right adjoint is fully faithful, and it follows that the functor
	\eqnn
	\widetilde{\Conv} \to \Deltaparasurj,
	\qquad
	(I,E) \mapsto I/E
	\eqnd  
is a localization. The edges $W \subset \widetilde{\Conv}$ sent to equivalences by this localization are those that induce isomorphisms $I/E_I \to J/E_J$; these are in turn precisely the Cartesian edges of the Cartesian fibration $\widetilde{\Conv} \to \Preordpara$.
\end{remark}

\subsection{Proof of Theorem~\ref{theorem. sheaves are Deltasurj}}\label{section:proof of sheaves}
\begin{proof}
We have the following arrows:
\begin{align}
\shv(\brokenpara, \cC)
	&\xrightarrow{ \text{Corollary~\ref{corollary:sheaves-on-broken-limit}}} \limarrow_{\Preordpara} \shv(F^{(I)},\cC) \nonumber \\
	&\xrightarrow{ \text{ Lemma~\ref{lemma:pullback-to-F^{(I)}-is-constructible} }} \limarrow_{\Preordpara} \shv^{\cbl}(F^{(I)},\cC) \nonumber\\
	&\xrightarrow{ \text{ Theorem~\ref{theorem.exit-path-theorem} }} \limarrow_{\Preordpara} \Ffun(\Exit(F^{(I)}),\cC) \nonumber\\
	&\xrightarrow{ \text{ Lemma~\ref{lemma:exit-F^{(I)}-is-Conv(I)} } } \limarrow_{\Preordpara} \Ffun(\Conv(I),\cC) \nonumber\\
	&\xrightarrow{ \text{Defn of colim} } \Ffun(\colimarrow_{\Preordpara^{\op}} \Conv(I),\cC) \nonumber\\
	&\xrightarrow{ \text{ Theorem~\ref{theorem. colimit of categories as localization} } }  \Ffun( \widetilde{\Conv}[W^{-1}],\cC) \nonumber\\
	&\xrightarrow{ \text{Remark~\ref{remark:Conv-Deltasurj-adjunction} } } \Ffun( \Deltaparasurj,\cC)	\nonumber
\end{align}

The first arrow is an equivalence by Corollary~\ref{corollary:sheaves-on-broken-limit}. To be precise: we have shown that $\brokenpara$ is a colimit of the $F^{(I)}$ in Corollary~\ref{corollary:brokenpara-as-colim-of-spaces}, hence its $\infty$-category of sheaves is the limit of the induced diagram of the $\infty$-categories of sheaves. 

The second arrow is an equivalence by Lemma~\ref{lemma:pullback-to-F^{(I)}-is-constructible}---because the sheaves pulled back from $\brokenpara$ are all constructible on $F^{(I)}$, the limit may be computed by a diagram of the full subcategories $\shv^{\cbl} \subset \shv$ consisting of constructible sheaves. 

The next arrow is an equivalence by the exit path theorem (Theorem~\ref{theorem.exit-path-theorem}) along with its naturality in the $F^{(I)}$ variable (Remark~\ref{remark:exit-paths-natural-in-X-P}). 

The fourth arrow is an equivalence by Lemma~\ref{lemma:exit-F^{(I)}-is-Conv(I)}, which shows that the exit path categories of $F^{(I)}$ are equivalent as $\infty$-categories to the very posets stratifying the $F^{(I)}$.

The next arrow is by the definition of colimit. 

The penultimate arrow is due to the general fact that colimits of $\infty$-categories may be computed by localizing along the collection $W$ of the (co)Cartesian edges of a representing (co)Cartesian fibration (Theorem~\ref{theorem. colimit of categories as localization}). 

Finally, the last arrow is an equivalence because $\widetilde{\Conv}$ is a Cartesian fibration over $\Preordpara$ modeling the diagram whose colimit we want to take, and moreover, the adjunction exhibited in Remark~\ref{remark:Conv-Deltasurj-adjunction} shows that the localization of $\widetilde{\Conv}$ along its Cartesian edges is precisely $\Deltaparasurj$. 
\end{proof}

\subsection{Proof of Theorem~\ref{theorem. main}}

\begin{proof}
Combine Theorem~\ref{theorem. sheaves are Deltasurj} and Corollary~\ref{cor. Deltainj and Deltasurj}. 
\end{proof}

\clearpage

\section{The cyclic case}
In this section we give a proof of Theorem~\ref{theorem. cyclic main}.
We will be brief, as the techniques are almost identical to the paracyclic case.

\subsection{The cyclic category}

Note that the paracyclic category $\Deltapara$ admits a $\ZZ$-action on all morphism sets---a morphism $r: I \to J$ which is $\ZZ$-equivariant gives rise to another map $r+1$, where
	\eqnn
	(r+1)(i) = r(i)+1
	\eqnd
and composition is compatible with the $\ZZ$-action.

\begin{defn}\label{defn:cyclic-category}
The {\em cyclic category} $\Deltacyc$ is the category with the same objects as $\Deltapara$, and where we set
	\eqnn
	\hom_{\Deltacyc}(I,J)
	=
	\hom_{\Deltapara}(I,J)/\ZZ
	\eqnd
with the inherited composition law.
\end{defn}

\begin{remark}
The above definition of the cyclic category is a variation we learned from~\cite{lurie-waldhausen}. It is equivalent to Connes's original definition~\cite{connes-cyclique}, where he denotes $\Deltacyc$ by $\Lambda$. 
\end{remark}

\begin{defn}
Likewise, the category of cyclic preorders 
	\eqnn
	\Preordcyc
	\eqnd
is the category obtained by quotienting the hom-sets of $\Preordpara$ by the obvious $\ZZ$-action.
\end{defn}

\begin{defn}
Fix an $\infty$-category $\cC$. A {\em cyclic object} in $\cC$ is a functor
	\eqnn
	\Deltacyc^{\op} \to \cC.
	\eqnd
A {\em semicyclic object} is the data of a functor
	\eqnn
	(\Deltacycinj)^{\op} \to \cC
	\eqnd
where $\Deltacycinj \subset \Deltacyc$ consists of those morphisms arising from injective morphisms in $\Deltapara$.
\end{defn}

\subsection{Broken cycles and their families}

\newenvironment{brokencycleproperties}{%
	  \renewcommand*{\theenumi}{(C\arabic{enumi})}%
	  \renewcommand*{\labelenumi}{(C\arabic{enumi})}%
	  \enumerate
	}{%
	  \endenumerate
}

\begin{defn}\label{defn:broken cycle}
A {\em broken cycle} is the data of a topological space $C$ equipped with a continuous $\RR$-action such that
\begin{brokencycleproperties}
\item $C$ is abstractly homeomorphic to the circle $S^1$.
\item\label{property:orientation} The $\RR$-action is directed, in that it induces an orientation on $S^1$, and
\item\label{property:fixed-point-set-discrete-in-circle} The fixed point set $C^\RR$ is discrete and non-empty.
\end{brokencycleproperties}
\end{defn}

\begin{remark}
A circle with $\RR$-action is a broken cycle if and only if its lift to the universal cover is a broken paracycle (Definition~\ref{defn:broken-paracycle}).
\end{remark}

We now give a definition of family of broken cycles in line with the criteria set forth in Theorem~\ref{theorem:broken-definitions-equivalent} for broken paracycles.

\begin{defn}\label{defn:family of broken cycles}
Fix a topological space $S$ and a pair
	\eqnn
	(\pi: C_S \to S, \mu)
	\eqnd
where $\mu$ is a fiber-wise $\RR$-action on $C_S$. We say $(\pi:C_S \to S, \mu)$ is a {\em family of broken cycles} if and only if it satisfies the following:
\enum
	\item  (Fibers are broken cycles.) For every $s \in S$, the fiber $C_s = \pi^{-1}(s)$ is a broken cycle in the sense of Definition~\ref{defn:broken cycle}.	
	\item  (Unramified.) Let $C_S^{\RR}$ denote the fixed point set of the $\RR$-action. Then the map
		\eqnn
		C_S^\RR \to S
		\eqnd
	induced by $\pi$ is unramified. 
	\item  (Local triviality.) For every $s \in S$, there is an open subset $U \subset S$ containing $s$ such that there exists an orientation-preserving homeomorphism 
		\eqnn
		U \times S^1 \cong \pi^{-1}(U)
		\eqnd 
		which respects the projection to $U$.
\enumd
\end{defn}

\begin{example}
Let $S = S^1$; then there is a unique family of broken cycles whose fibers have exactly one break: the trivial family. However, there are many non-isomorphic families of broken paracycles over $S$ whose $\ZZ$-quotients are all trivial families of broken cycles.
\end{example} 

\begin{defn}
A map of families of broken cycles is a commutative diagram
	\eqnn
	\xymatrix{
	C_S \ar[r]^{\tilde f} \ar[d] & C_T \ar[d] \\
	S \ar[r]^{f} & T
	}
	\eqnd
where $\tilde f$ is $\RR$-equivariant, and exhibits $C_S$ as homeomorphic to the pullback $C_T \times_T S$.
\end{defn}

\begin{defn}
$\brokencyc$ is the stack classifying families of broken cycles.
\end{defn}

\subsection{The stack of broken cycles as a colimit}

\begin{remark}\label{remark:broken cycles lift to broken paracycles locally}
By local triviality, any family of broken cycles locally admits a lift to a family of broken paracycles. Thus the natural map $\brokenpara \to \brokencyc$ given by passing a family $L_S \to S$ to the quotient $L_S/\ZZ \to S$ is a cover.
\end{remark}

By Remark~\ref{remark:broken cycles lift to broken paracycles locally} and because families of broken paracycles admit local $I$-sections (Lemma~\ref{lemma:local-I-sections-exist}),  any family of broken cycles locally admits a map
	\eqnn
	\sigma: U \times I \to C_U
	\eqnd
where $I$ is some paracyclic preorder. This shows that the topological spaces $F^{(I)}$ cover the stack $\brokencyc$. Moreover, $\sigma$ factors through the quotient $U \times I/\ZZ$ by $\ZZ$-equivariance; this means that the functor
	\eqnn
	\Preordpara^{\op} \to \Stacks_{/\brokenpara}
	\qquad
	I \mapsto F^{(I)}
	\eqnd
(from Theorem~\ref{theorem. I-sections representable})
descends to a functor
	\eqnn
	\Preordcyc^{\op} \to \Stacks_{/\brokencyc}
	\qquad
	I \mapsto F^{(I)}.
	\eqnd

\begin{remark}
Note that $F^{(I)}$ also classifies the data of a family of broken cycles, equipped with a lift to a family of broken paracycles, which in turn is equipped with an $I$-section. This is succinctly captured in the composition
	\eqnn
	F^{(I)} \simeq \brokenpara^I \to \brokenpara \to \brokencyc.
	\eqnd
\end{remark}

\begin{theorem}
The induced map
	\eqnn
	\colimarrow_{I \in \Preordcyc^{\op}} F^{(I)} \to \brokencyc
	\eqnd
is an equivalence of stacks.
\end{theorem}

\begin{proof}
This proceeds in parallel to the proof of Theorem~\ref{theorem. brokenpara as a colimit}. Consider the category $(\Preordcyc)_\ast$ obtained from $\PreordZ$ by modding out the hom-sets by the natural $\ZZ$-action. 

Given a pair of objects $I_0, I_1 \in \Preordcyc^{\op}$, one can again define a poset $\Amalg(I_0,I_1)$, and by $\ZZ$-equivariance, the functors from Remark~\ref{remark: forgetful amalga} descend to functors
	\eqnn
	(\Preordcyc)_\ast \leftarrow \widetilde{\Amalg} \to 
	(\Preordcyc)_\ast \times 
	(\Preordcyc)_\ast.
	\eqnd
Then the analogous string of equivalences to the ones in~\eqref{eqn:bigComposition} hold:
	\begin{align}
   &  \left( \colimarrow_{I_0 \in (\Preordcyc)_\ast^{\op}} F^{(I_0)} \right) \times_{\brokencyc}
    	\left( \colimarrow_{I_1 \in (\Preordcyc)_\ast^{\op}} F^{(I_1)} \right) \nonumber \\
	& \to   \colimarrow_{I_0, I_1 \in (\Preordcyc)_\ast^{\op} \times  (\Preordcyc)_\ast^{\op}} F^{(I_0)}  \times_{\brokencyc} F^{(I_1)}  \nonumber \\
	&  \to   \colimarrow_{I_0, I_1} 		\colimarrow_{K \in \Amalg(I_0,I_1)^{\op}} F^{(K)} \nonumber\\
	&  \to    \colimarrow_{I_0, I_1} 		\colimarrow_{(J_0,J_1,K) \in \widetilde \Amalg^{\op}_{/(I_0,I_1)}} F^{(K)} \nonumber\\
	&  \to    \colimarrow_{(I_0,I_1,K) \in \widetilde \Amalg^{\op}} F^{(K)} \nonumber\\
	&  \to   \colimarrow_{I \in (\Preordcyc)_\ast^{\op}} F^{(I)}. \nonumber 
	\end{align}
We conclude by noting that the left Kan extension along $\Preordcyc^{\op} \to (\Preordcyc)_\ast^{\op}$ computes the desired colimit. We refer to Section~\ref{section.proof-of-colimit-theorem} for more details.
\end{proof}

\subsection{Proof of Theorem~\ref{theorem. cyclic main}}

\begin{proof}
This proceeds in parallel to the proof of Theorem~\ref{theorem. sheaves are Deltasurj}. (See Section~\ref{section:proof of sheaves}.) The only difference is that we must now consider the composite fibration 
	\eqnn
	q_{\cyc}: \widetilde{\Conv} \xrightarrow{q_{\para}}  \Preordpara
	\xrightarrow{/\ZZ} \Preordcyc
	\eqnd
and localize the $q_{\cyc}$-Cartesian edges, rather than the $q_{\para}$-Cartesian edges. We observe that $\widetilde{\Conv}$ admits a localization functor to $\Deltacycsurj$ (which now plays the role of $\Deltaparasurj$ in Remark~\ref{remark:Conv-Deltasurj-adjunction}). Then the string of equivalences in Section~\ref{section:proof of sheaves} have an analogous string of equivalences as follows:

\begin{align}
\shv(\brokencyc, \cC)
	&\to \limarrow_{\Preordcyc} \shv(F^{(I)},\cC) \nonumber \\
	&\to \limarrow_{\Preordcyc} \shv^{\cbl}(F^{(I)},\cC) \nonumber\\
	&\to \limarrow_{\Preordcyc} \Ffun(\Exit(F^{(I)}),\cC) \nonumber\\
	&\to \limarrow_{\Preordcyc} \Ffun(\Conv(I),\cC) \nonumber\\
	&\to \Ffun(\colimarrow_{\Preordcyc^{\op}} \Conv(I),\cC) \nonumber\\
	&\to  \Ffun( \widetilde{\Conv}[W_{\cyc}^{-1}],\cC) \nonumber\\
	&\to \Ffun( \Deltacycsurj,\cC)	\nonumber
\end{align}
where we note that $W_{\cyc}$ refers to the Cartesian edges with respect to $q_{\cyc}$, not $q_{\para}$. 
This completes the proof.
\end{proof}

\clearpage

\bibliographystyle{amsalpha}
\bibliography{biblio}


%
%


\end{document}